\documentclass[a4paper,10pt]{article}
\usepackage{verbatim,amsmath,amsthm,amsfonts,amssymb,latexsym,graphicx,mathtools,extpfeil,color,mathabx}
\usepackage{epstopdf,pinlabel}
\usepackage[all]{xy}
\usepackage{graphicx}
\usepackage{caption}
\usepackage{subcaption}
\DeclareMathAlphabet{\mathpzc}{OT1}{pzc}{m}{it}
\newcommand{\longcomment}[2]{#2}
\newcommand\zz{\mathbb{Z}}
\newcommand\qq{\mathbb{Q}}
\newcommand\rr{\mathbb{R}}
\newcommand\cc{\mathbb{C}}
\newcommand\nn{\mathbb{N}}
\newcommand\bm{{\bf m}}
\newcommand\bn{{\bf n}}
\newcommand\bk{{\bf k}}
\newcommand\bI{{\bf I}}
\newcommand\bp{{\bf p}}
\newcommand\bo{{\bf o}}

\newcommand\s{{\mathfrak s}}

\newcommand\st{{\mathfrak t}}
\newcommand\DW{{\mathcal D_W}}
\newcommand\DWP{{\mathcal D'_W}}

\newcommand\mc{{\mathfrak C(R)}}
\newcommand\mctL{{\mathfrak C(\tL)}}
\newcommand\cgf{\mathfrak C_{\textsc {gf}}(R)}
\newcommand\cgfl{\mathfrak C_{\textsc {gf}}(\Lambda)}

\newcommand\meGb{M_\eta(\bW)}
\newcommand\meGbr{M_\eta(\bW|_{G'})}
\newcommand\meGbe{M_\eta(\bW;e)}

\newcommand\meGab{M_\eta(\bW,\alpha,\beta)}
\newcommand\meGabe{M_\eta(\bW,\alpha,\beta;e)}

\newcommand\tmeGab{\widetilde M_\eta(\bW,\alpha,\beta)}
\newcommand\tmeGabe{\widetilde M_\eta(\bW,\alpha,\beta;e)}

\newcommand\feW{f_\bW^\eta}
\newcommand\tfeW{\tilde f_\bW^\eta}
\newcommand\tfeWZ{\tilde f_{\bW,\bZ}^\eta}

\newcommand\gmni{g_{\bm(i)}^{\bn(i)}}
\newcommand\gmnim{g_{\bm(i-1)}^{\bn(i-1)}}
\newcommand\nmni{n_{\bm(i)}^{\bn(i)}}
\newcommand\nmnim{n_{\bm(i-1)}^{\bn(i-1)}}
\newcommand\kmni{k_{\bm(i)}^{\bn(i)}}
\newcommand\hkmni{\bar k_{\bm(i)}^{\bn(i)}}
\newcommand\tkmni{\tilde k_{\bm(i)}^{\bn(i)}}
\newcommand\hqmni{\bar q_{\bm(i)}^{\bn(i)}}
\newcommand\Gmn{G_{\bf m}^{\bf n}}
\newcommand\Nmn{N_{\bf m}^{\bf n}}
\newcommand\Kmn{K_{\bf m}^{\bf n}}
\newcommand\Jmn{J_{\bf m}^{\bf n}}
\newcommand\bQmn{\bar Q_{\bm}^{\bn}}

\newcommand\gmnidm{g_{{\bf m}(i-2)}^{{\bf n}(i-2)}}

\newcommand\Wmn{{W_{\bf m}^{\bf n}}}
\newcommand\cWmn{{\widecheck W_{\bf m}^{\bf n}}}
\newcommand\emn{{\eta_{\bf m}^{\bf n}}}
\newcommand\feWmn{\tilde f^{\eta_{\bf m}^{\bf n}}_{\bW_{\bf m}^{\bf n},\bZ_{\bf m}^{\bf n}}}
\newcommand\fGmn{\tilde f_{G_{\bf m}^{\bf n}}}
\newcommand\fGmk{\tilde f_{G_{\bf m}^{\bf k}}}
\newcommand\fGkn{\tilde f_{G_{\bf k}^{\bf n}}}
\newcommand\fNmn{\tilde f_{N_{\bf m}^{\bf n}}}
\newcommand\fKmn{\tilde f_{K_{\bf m}^{\bf n}}}

\newcommand\Zmn{{Z_{\bf m}^{\bf n}}}
\newcommand\cZmn{{\widecheck Z_{\bf m}^{\bf n}}}
\newcommand\bWmn{{\bW_{\bf m}^{\bf n}}}
\newcommand\bZmn{{\bZ_{\bf m}^{\bf n}}}
\newcommand\bWmk{{\bW_{\bf m}^{\bf k}}}
\newcommand\bZmk{{\bZ_{\bf m}^{\bf k}}}
\newcommand\bWkn{{\bW_{\bf k}^{\bf n}}}
\newcommand\bZkn{{\bZ_{\bf k}^{\bf n}}}

\newcommand\Smn{S_{\bf m}^{\bf n}}
\newcommand\cSmn{\widecheck S_{\bf m}^{\bf n}}

\newcommand\ym{Y_{\bf m}}
\newcommand\zm{{\zeta_{\bf m}}}

\newcommand\yn{Y_{\bf n}}
\newcommand\zn{{\zeta_{\bf n}}}

\newcommand\sk{\Sigma(K)}

\newcommand\Cmn{C_{\bf m}^{\bf n}}

\newcommand\Cmni{C_{{\bf m}(i)}^{{\bf n}(i)}}

\newcommand\Cmnim{C_{{\bf m}(i-1)}^{{\bf n}(i-1)}}
\newcommand\Cmnimm{C_{{\bf m}(i-2)}^{{\bf n}(i-2)}}
\newcommand\Cmnimmm{C_{{\bf m}(i-3)}^{{\bf n}(i-3)}}

\newcommand\dmn{d_{\bf m}^{\bf n}}

\newcommand\dmni{d_{{\bf m}(i)}^{{\bf n}(i)}}

\newcommand\dmnim{d_{{\bf m}(i-1)}^{{\bf n}(i-1)}}
\newcommand\dmnimm{d_{{\bf m}(i-2)}^{{\bf n}(i-2)}}
\newcommand\dmnimmm{d_{{\bf m}(i-3)}^{{\bf n}(i-3)}}

\newcommand\jmn{j_{\bf m}^{\bf n}}

\newcommand\spinc{\text{spin}^{c}}

\newcommand\bW{\mathfrak W}

\newcommand\bZ{\mathfrak Z}
\newcommand\bU{\mathfrak U}

\newcommand\tQmn{\widetilde Q_{{\bf m}}^{{\bf n}}}

\newcommand\qmni{q_{{\bf m}(i)}^{{\bf n}(i)}}
\newcommand\ind{{\rm ind}}
\newcommand\sD{{\mathcal D}}
\newcommand\ad{\text{ad}}
\newcommand\ext{\text{ext}}
\newcommand\coker{\text{coker}}
\newcommand\image{{\rm image}}

\newcommand\Cone{{\rm Cone}}

\newcommand\PFH{{\rm PFH}}
\newcommand\tPFH{{\widetilde{\rm PFH}}}

\newcommand\tPFC{{\widetilde{\rm PFC}}}

\newcommand\PKC{{\rm PKC}}
\newcommand\PKE{{\rm PKE}}
\newcommand\tPKE{\widetilde {\rm PKE}}
\newcommand\td{{\tilde d}}
\newcommand\PFC{{\rm PFC}}

\newcommand\tLz{{\tilde \Lambda_0}}	
\newcommand\tL{{\tilde \Lambda}}

\newcommand\cU{{\mathcal K}}
\newcommand\cS{{\mathcal K}}

\newcommand\GZ{{\Gamma_{Z}}}
\newcommand\Gz{{\Gamma_{\zeta}}}

\newcommand\Gzz{{\Gamma_{\zeta_0}}}
\newcommand\Gzo{{\Gamma_{\zeta_1}}}

\newcommand\sgn{{\rm sgn}}
\newcommand\mU{{\mathcal U}}
\newcommand*{\Scale}[2][4]{\scalebox{#1}{$#2$}}%

\theoremstyle{plain}
\newtheorem{theorem}{Theorem}[section]
\newtheorem{theorem-intro}{Theorem}
\newtheorem{remark-intro}{Remark}
\newtheorem{corollary-intro}{Corollary}
\newtheorem{lemma}[theorem]{Lemma}

\newtheorem{corollary}[theorem]{Corollary}
\newtheorem{proposition}[theorem]{Proposition}

\theoremstyle{definition}
\newtheorem{definition}[theorem]{Definition}
\newtheorem{remark}[theorem]{Remark}
\newtheorem{example}[theorem]{Example}
\newtheorem{hypothesis}{Hypothesis}

\longcomment{
\newtheorem{prethm}{{\bf Theorem}}[section]

\newtheorem{prepro}{{\bf Proposition}}[section]

\newtheorem{precor}{{\bf Corollary}}[section]

\newtheorem{preconj}{{\bf Conjecture}}

\newtheorem{preremark}{{\bf Remark}}
\newenvironment{remark}{\begin{preremark}\rm{\hspace{-0.5
               em}{\bf.}}}{\end{preremark}}

\newtheorem{predef}{{\bf Definition}}[section]

\newtheorem{preex}{{\bf Example}}[section]
\newenvironment{example}{\begin{preex}\rm{\hspace{-0.5
               em}{\bf.}}}{\end{preex}}

\newtheorem{prelem}{{\bf Lemma}}[section]

\newenvironment{proof}[1]{\noent{\bf Proof.}{\rm #1}\hfill{$\Box$}}{}
}

\title{\large \bf Abelian Gauge Theory, Knots and Odd Khovanov Homology}

\begin{document}
\maketitle

\begin{abstract}
	A homological invariant of 3-manifolds is defined using abelian Yang-Mills gauge theory. It is shown that the construction, in an appropriate sense, is functorial
	with respect to the families of 4-dimensional cobordisms. This construction and its functoriality are used to define several link invariants.
	The strongest version of these invariants has the form of a filtered chain complex that can recover Khovanov homology of the mirror image as a bi-graded group.
\end{abstract}

Gauge theoretical methods are proven to be useful tools to study low dimensional manifolds. As one of the first instances of this approach, Donaldson constructed a family of  invariants, usually known as {\it Donaldson invariants}, for smooth closed 4-manifolds that satisfy a mild topological assumption \cite{Don:Polynomial}. To construct these invariants, one fixes a Hermitian vector bundle of rank 2 on a Riemannian 4-manifold and considers the space of connections that satisfy a geometrical PDE (aka {\it anti-self-duality} equation) modulo the action of bundle automorphisms. Although the anti-self-duality equation depends on the Riemannian metric, a careful study of the moduli spaces of anti-self-dual connections leads to numerical invariants that depend only on the diffeomorphism type of the 4-manifold. More recently, Kronheimer applied a modification of the original construction to Hermitian vector bundles of rank $n$, with $n\geq 3$, and defined similar 4-manifold invariants. In the case $n=1$, the relevant moduli spaces can be characterized by the classical topological invariants (i.e. $\hspace{-10pt}$ cohomology groups of the underlying 4-manifolds) and hence they do not give rise to any interesting numerical invariant that can be used to study the topology of 4-manifolds.

Let $Y$ be a 3-manifold which has the same integral homology groups as the 3-dimensional sphere. In \cite{F:IF}, Floer used the moduli space of anti-self-dual connections on $\rr \times Y$ to produce topological invariants of $Y$. In parallel to the case of 4-manifolds, one expects that anti-self-dual connections on vector bundles of higher rank can be also utilized to define  3-manifold invariants. Kronheimer and Mrowka used this approach in a greater generality, and among other topological invariants, they constructed invariants of 3-manifolds using higher rank bundles. 

The purpose of this article is to show that the moduli space of anti-self-dual connections on a Hermitian line bundle can be used to extract non-trivial data in 3-manifold topology. We will use a slight variation of these spaces to define an invariant that is called {\it plane Floer homology}. Plane Floer homology of a 3-manifold $Y$ is a graded vector space over a field $\Lambda$ and is denoted by $\PFH(Y)$. As any other Floer homology theory, this vector space is the homology of a chain complex which in this case is denoted by $\PFC(Y)$. An important feature of plane Floer homology is functoriality with respect to cobordisms. That is to say, if $W:Y_0 \to Y_1$ is a 4-dimensional cobordism from $Y_0$ to $Y_1$, then there is a chain map $\PFC(W):\PFC(Y_0) \to \PFC(Y_1)$. This cobordism map is defined with the aid of an appropriate metric on $W$. More generally, for a family of metrics on the cobordism $W$, we can define cobordism maps from $\PFC(Y_0)$ to $\PFC(Y_1)$. The following theorem is the heart of this article:

\begin{theorem-intro} \label{chain-com-spec}
	If $L$ is a framed link with $n$ connected components in a 3-manifold $Y$, then there is a chain complex $\PFC_L(Y)$ that has the same chain homotopy type as 
	$\PFC(Y)$. The underlying chain group of $\PFC_L(Y)$ is equal to:
	$$\bigoplus_{\bm \in \{0,1\}^n} \PFC(Y_\bm)$$
	For $\bm=(m_1,\dots,m_n) \in \{0,1\}^n$, the 3-manifold $Y_\bm$ is given by performing the $m_i$-surgery along the $i^{th}$ connected component of $L$ for each $i$. 
	The differential of this complex is also defined using cobordism maps associated with certain families of metrics.
\end{theorem-intro}

For a detailed version of this theorem see Theorem \ref{PFH-tri-1} and Corollary \ref{new-complex-1}. To have a closer look at Theorem \ref{chain-com-spec}, let $L$ have one connected component and $Y_{(1)}$ (respectively, $Y_{(0)}$) be the result of 1-surgery (respectively, 0-surgery) along $L$. Let also $W: Y_{(1)} \to Y_{(0)}$ be the standard cobordism that is given by attaching a 2-handle to $[0,1]\times Y_{(1)}$ along $L$. Therefore, there is a chain map $\PFC(W):\PFC(Y_{(1)}) \to \PFC(Y_{(0)})$. In this special case, $\PFC_L(Y)$ is the mapping cone of the chain map $\PFC(W)$. As a consequence, there is an exact triangle of the following form:

\begin{equation} \label{add-cross-iso}
	\xymatrix{
		&\PFH(Y)\ar[dr]&\\
		\PFH(Y_{(0)})\ar[ur]&&\PFH(Y_{(1)})\ar[ll]
	}
\end{equation}

In the more general case that $L$ has more than one connected component, there is again a filtration on the chain complex $\PFC_L(Y)$. This filtration determines a spectral sequence that abuts to $\PFH(Y)$. This spectral sequence should be thought as the counterpart of the {\it surgery spectral sequence} of \cite{OzSz:HF-branch} for plane Floer homology. Similar spectral sequences for various other Floer homologies are established in \cite{KM:Kh-unknot,Bl:HM-branch,CS:iCube}. Note that in all these cases, the pages of the spectral sequence depend on the framed link $L$ and are not 3-manifold invariants.

Let $K$ be a link in $S^3$ and $\Sigma(K)$ be the branched double cover of $S^3$ branched along $K$. The chain complexes given by Theorem \ref{chain-com-spec} can be used to construct additional structures on $\PFH(\Sigma(K))$ which are not otherwise apparent in the original chain complex. Let $D$ be a link diagram for $K$. One can use this diagram to produce a framed link in $\Sigma(K)$ (cf. section \ref{classical-links}). Therefore, it can be used as an input for Theorem \ref{chain-com-spec} to construct a chain complex that is denoted by $\PKC(D)$. We shall define two gradings on $\PKC(D)$ which are called {\it homological} and $\delta$-gradings. The differential of the chain complex $\PKC(D)$ does not decrease the homological grading and increases the $\delta$-grading by 1. We call a chain complex with a bi-grading that satisfies the above properties a {\it filtered $\zz$-graded chain complex}:

\begin{theorem-intro} \label{PKC-thm}
	The homotopy type of $\PKC(D)$, as a filtered $\zz$-graded chain complex, depends only on $K$ and hence is a link invariant. 
\end{theorem-intro}

For the definition of the homotopy type of a filtered $\zz$-graded chain complex see subsection \ref{PKH}. The homological filtration and the $\delta$-grading of $\PKC(D)$ induces a spectral sequence $\{\PKE_r(K)\}$ where $\PKE_r(K)$ is a bi-graded vector space. As an immediate corollary of Theorem \ref{PKC-thm}, we have:

\begin{corollary-intro} \label{spec-seq}
	For each $r\geq 2$, there is a bi-graded vector space $\PKE_r(K)$ that is an invariant of $K$.	
\end{corollary-intro}

\begin{remark-intro}
	The spectral sequence in Corollary \ref{spec-seq} can be utilized to define a series of concordance homomorphisms. In fact, an equivariant version of 
	$\PKC(D)$ gives rise to other spectral sequences similar to that of Corollary \ref{spec-seq}. These spectral sequences are the starting point to define other 
	concordance invariants. This circle of ideas will be explored in a forthcoming work. 
\end{remark-intro}

Corollary \ref{spec-seq} provides us with a series of homological link invariants which begs for further study. In the last part of this article, we take up the task of understanding the first invariant in this list and show that it is related to {\it Khovanov homology}. Khovanov homology is a categorification of Jones polynomial defined in \cite{Kh:Jones-cat}. In \cite{OzRaSz:Odd-Kh}, an elaborate modification of the definition of Khovanov homology was used to define another categorification of Jones polynomial that is known as odd Khovanov homology. Subsequently, Khovanov's original link invariant is also known as even Khovanov homology. Even and odd Khovanov homology agree when one uses a characteristic two ring to define these homological link invariants. However, if a characteristic zero coefficient ring is used, then these invariants diverge significantly from each other. 

\begin{theorem-intro} \label{2nd-page}
	The invariant $\PKE_2(K)$ is isomorphic to the Khovanov homology of the mirror image of $K$ with coefficients in the field $\Lambda$. 
\end{theorem-intro}

The field $\Lambda$ has characteristic two and is the quotient of a ring $\tL$ with characteristic zero (cf. subsection \ref{ori-loc}). The vector space $\PFH(Y)$ can be lifted to an invariant $\tPFH(Y)$ which is a $\tL$-module and is called the {\it oriented plane Floer homology} of $Y$. There is an analogue of the spectral sequence $\{\PKE_r(K)\}$ in the context of oriented plane Floer homology:

\begin{theorem-intro} \label{oriented-2nd-page}
	There is a spectral sequence $\{\tPKE_r(K)\}$ with coefficients in $\tL$ that abuts to $\tPFH(\Sigma(K))$. The second page of this spectral sequence, $\tPKE_2(K)$, is 
	isomorphic to odd Khovanov homology of the mirror image of $K$ with coefficients in the ring $\tL$.
\end{theorem-intro}

This article is organized as follows. In section \ref{Ab-ASD-con}, we focus on the definition of plane Floer homology of 3-manifolds and the cobordism maps. Section \ref{ex-tri} is devoted to providing the proof of Theorem \ref{chain-com-spec}. In section \ref{classical-links}, we use the results of section \ref{ex-tri} in the special case of the branched double covers and upgrade plane Floer homology of branched double covers to a stronger invariant. In particular, Theorem \ref{PKC-thm} is proved in this section. In Section \ref{odd-kh}, we present a slight reformulation of odd Khovanov homology and prove Theorems \ref{2nd-page} and \ref{oriented-2nd-page}.

{\it Acknowledgement.} The author would like to thank his advisor, Peter Kronheimer, for many illuminating discussions and support. This work was motivated by the thesis problem suggested to the author by Peter Kronheimer. I am also grateful to Clifford Taubes, Tomasz Mrowka, Sucharit Sarkar, Robert Lipshitz, Steven Sivek, Katrin Wehrheim, and Jonathan Bloom for interesting conversations. I thank Steven Sivek for his helpful comments on an early draft of this paper. 

\section{Abelian Anti-Self-Dual Connections} \label{Ab-ASD-con}
Plane Floer homology is defined with the aid of the moduli spaces of $\spinc$ connections that satisfy a version of anti-self-duality equation. The local behavior of these spaces is governed by the anti-self-duality operator. In subsection \ref{abelian-asd-op}, we recall the definition of appropriate versions of this operator that serve better for our purposes. Then we review the basic properties of this operator. These are fairly standard facts whose proves are scattered in the literature. The moduli space of anti-self-dual $\spinc$ connections are constructed in subsection \ref{ab-asd-eq}. We use these spaces to define a primary version of plane Floer homology and cobordism maps in subsection \ref{cob-maps}. In the last subsection of this section, we discuss orientations of the moduli spaces and  certain local coefficient systems on these spaces. These structures are the main ingredients in the definition of cobordism maps for oriented plane Floer homology in subsection \ref{ori-loc}. 

\subsection{Abelian ASD Operator} \label{abelian-asd-op}

Suppose $W^\circ :Y_0 \to Y_1$ is a compact, connected cobordism between Riemannian 3-manifolds $Y_0$ and $Y_1$. Although $Y_1$ can have several (possibly zero) connected components, we require that $Y_0$ to be connected. We equip $W^\circ$ with a Riemannian metric such that the metric in a collar neighborhood of $Y_0$ and $Y_1$ is isometric to the product metric corresponding to the metrics on $Y_0$ and $Y_1$. One can glue cylindrical ends $\rr^{\leq0} \times Y_0$ and  $\rr^{\geq 0}\times Y_1$ with the product metric to $W^\circ$ in order to produce a non-compact and complete Riemannian manifold denoted by $W$. For a real number $\delta$, let $\psi_\delta$ be a smooth positive function on $W$ such that its value for $(t,y)\in \rr^{\leq0} \times Y_0  \coprod \rr^{\geq 0}\times Y_1$ is equal to $e^{\delta|t|}$. The weighted Sobolev space $L^2_{k,\delta}(W)$ is defined to be the space of functions $f$ such that $\psi_\delta f \in  L_k^2(W)$. Roughly speaking, if $\delta$ is positive, then the elements of this space have exponential decay over the ends of $W$. Otherwise, they are allowed to have controlled exponential growth. This Banach spase is independent of the choice of $\psi_\delta$. If different weights $\delta_0$ and $\delta_1$ are used on the ends $\rr^{\leq0} \times Y_0 $ and $\rr^{\geq 0}\times Y_1$, then the resulting Banach space is denoted by $L^2_{k,\delta_0,\delta_1}(W)$. Given a vector bundle $E$ with an inner product, the definition of $L^2_{k,\delta}(W)$ can be extended to $L^2_{k,\delta}$ sections of $E$. The dual of the Banach space of $L^2_{k,\delta}$ sections of $E$ is the space of $L^2_{k,-\delta}$ sections of the same vector bundle. The $L^2_k$-inner product defines the pairing between these Banach spaces. 

The exterior derivative gives rise to a map $d:L^2_{k,\delta}(W,\Lambda^l) \to L^2_{k-1,\delta}(W,\Lambda^{l+1})$. As usual the formal adjoint of $d$ is denoted by $d^*:L^2_{k,\delta}(W,\Lambda^{l+1})\to L^2_{k-1,\delta}(W,\Lambda^{l})$. Since $W$ is a Riemannian 4-manifold, a 2-form on $W$ decomposes as a summation of a self-dual form and an anti-self-dual form. Let $d^+: L^2_{k,\delta}(W,\Lambda^{1})\to L^2_{k-1,\delta}(W,\Lambda^{+})$ be the composition of $d$ and the projection on the space of self-dual 2-forms. Define the {\it anti-self-duality operator} ({\it ASD operator}) to be the differential operator $- d^* \oplus d^+: L^2_{k,\delta}(W,\Lambda^{1}) \to L^2_{k-1,\delta}(W,\Lambda^{0}\oplus \Lambda^+)$. 

In order to understand the ASD operator on a cobordism with cylindrical ends, we firstly consider the ASD operator on a cylinder $(a,b)\times Y$ with the product metric. A 1-form $A$ and an anti-self-dual form $B$ on this cylinder can be decomposed in the following way:
\begin{equation}\label{decom-1}
	A=\alpha(t)dt+\beta(t) \hspace{2cm} B=\frac{1}{2}(dt\wedge \gamma(t)+*_3 \gamma(t))
\end{equation}	
where for each $t\in (a,b)$, $\alpha(t)$ is a 0-form, and $\beta(t)$, $\gamma(t)$ are 1-forms on $Y$. Therefore, a 1-form on the cylinder is given by a 1-parameter family of 0-forms and 1-forms, and an anti-self-dual 2-form is determined by a 1-parameter family of 1-forms. Given this, we can rewrite the ASD operator as $\frac{d}{dt}-Q$ where $Q$ is equal to the following differential operator acting on $\Lambda^0(Y)\oplus \Lambda^1(Y)$:
\begin{equation} \label{al-be}
	Q:=\left(
	\begin{array}{cc}
		0&d_3^*\\
		d_3&-*_3d_3
	\end{array}
	\right)
\end{equation} 

Here $d_3$, $d_3^*$, and  $*_3$ are the exterior derivative, its adjoint, and the Hodge operator on $Y$. The operator $Q$ is an unbounded and self-adjoint operator acting on $L^2(Y,\Lambda^0\oplus \Lambda^1)$. There exists an orthonormal basis $\{\phi_i\}_{i=1}^\infty$ for $L^2(Y,\Lambda^0\oplus \Lambda^1)$ that consists of the eigenvectors of $Q$. If $\lambda_i$ is the eigenvalue corresponding to $\phi_i$, then $\{\lambda_i\}_{i=1}^\infty$ is a discrete subset of $\rr$. Therefore the set of non-zero eigenvalues has an element with the smallest magnitude. Fix $\delta$ to be a positive number smaller than the magnitude of this element. The kernel of $Q$ can be also identified as:
\begin{equation} \label{kerQ}
	\left(
	\begin{array}{c}
		\alpha\\
		\beta
	\end{array}
	\right)\in \ker(Q)
	\iff
	d\alpha=0, d\beta=0, d^*\beta=0
\end{equation} 
That is to say $\alpha$ is a constant function and $\beta$ is a harmonic 1-form. Therefore, $\ker(Q)$ has dimension $b_1(Y)+1$. 

In general, a 1-form $a=\alpha(t)dt+\beta(t)$ on a cylinder can be decomposed with respect to the spectrum of $Q$ in the following way:
\begin{equation}\label{spec-dec-1}
	\left(
	\begin{array}{c}
		\alpha(t)\\
		\beta(t)
	\end{array}
	\right)= \sum _i f_i(t) \phi_i
\end{equation} 
Now assume that $a$ is defined on $\rr^{\geq 0} \times Y$ (respectively, on $\rr^{\leq 0}\times Y$) and $d^*(a)=0$, $d^+(a)=0$. Vanishing of $a$ by the ASD operator implies that $f_i(t)=c_ie^{\lambda_it}$. The 1-form $a$ has a finite $L^2_{k,\delta}$ norm if and only if $a$ has a  finite $L_k^2$-norm because if $||a||_{L_k^2}$ is finite, then $c_i=0$ unless $\lambda_i$ is negative (respectively, positive). This in turn suffices to ensure that $||a||_{L_{k,\delta}^2}$ is finite. For our purposes, we need to work with slightly larger Banach spaces. The {\it extended weighted Sobolev space} $L^2_{k,\ext}(Y\times \rr^{\geq 0}, \Lambda^1)$ consists of the 1-forms $a$ that can be written as a sum $b+h$ where $b\in L^2_{k,\delta}(Y\times \rr^{\geq 0}, \Lambda^1)$ and $h$ is the pull-back of a harmonic 1-form on $Y$. Note that this decomposition is unique and hence $L^2_{k,\ext}$ is a Banach bundle over the set of harmonic 1-forms with fiber $L^2_{k,\delta}$. 

Going back to the case of a cobordism $W: Y_0 \to Y_1$, let $\delta$ be smaller than the smallest of the eigenvalues of the operator $Q$ associated with the 3-manifold $\overline {Y_0}\coprod Y_1$. The Banach space $L^2_{k,\ext}(W,\Lambda^1)$ for a cobordism $W$ consists of the 1-forms that have finite $L^2_{k}$ norm on the compact subsets of $W$ and their restrictions to the ends lie in $L^2_{k,\ext}(Y_0\times \rr^{\leq 0}, \Lambda^1)$ and $L^2_{k,\ext}(Y_1\times \rr^{\geq 0}, \Lambda^1)$. Let $C_0$ and $C_1$ be the spaces of harmonic 1-forms on $Y_0$ and $Y_1$. Also, suppose $\varphi: C_0 \oplus C_1 \to L^2_{k,\ext}(W,\Lambda^1)$ is a linear map that sends an element $h \in C_0 \oplus C_1$ to a smooth 1-form that is equal to the pull back of $h$ on the ends. This map can be constructed for example by the aid of appropriate cut-off functions on the ends. Each element of $L^2_{k,\ext}(W,\Lambda^1)$ can be uniquely written as the sum $b+\varphi(h)$ where $b\in L^2_{k,\delta}(W,\Lambda^1)$ and $h\in C_0 \oplus C_1$. Therefore, $L^2_{k,\ext}(W,\Lambda^1)$ is isomorphic to $L^2_{k,\delta}(W,\Lambda^1)\oplus C_0 \oplus C_1$. This decomposition can be used to equip $L^2_{k,\ext}(W,\Lambda^1)$ with an inner product. The dual of this Banach space is $L^2_{k,-\delta}(W,\Lambda^1)\oplus C_0 \oplus C_1$. In the following, we will write $\iota_i$ for the projection of $L^2_{k,\ext}(W,\Lambda^1)$ into $C_i$. 

It is a standard fact that the ASD operator, acting on $L^2_{k,\delta}(W,\Lambda^1)$, (and hence on $L^2_{k,\ext}(W,\Lambda^1)$) is elliptic and Fredholm. In the following lemma we characterize the kernel and the cokernel of the ASD operator. Firstly we need to introduce the following notation: let $H^i_{comp}(W)$ be the $i^{\rm th}$ de Rham cohomology group with compact support. The intersection pairing $Q:H^2_{comp}(W) \times H^2(W)\to\zz$ on $W$ induces a non-degenerate bi-linear form on $I(W):=\image(i:H^2_{comp}(W) \to  H^2(W))$. This space can be decomposed  as $I^+(W)\oplus I^-(W)$ where $I^+(W)$ and $I^-(W)$ are choices of a maximal positive definite and a maximal negative definite subspaces of $I(W)$. The dimension of these spaces are denoted by $b^+(W)$ and $b^-(W)$.

\begin{lemma} \label{ind-DW-1}
	The kernel and the cockerel of $- d^* \oplus d^+:L^2_{k,\ext}(W,\Lambda^1) \to L^2_{k-1,\delta}(W,\Lambda^0\oplus \Lambda^+)$ can be identified with 
	$H^1(W;\rr)$ and $I^+(W)\oplus \rr$, respectively.
\end{lemma}

\begin{proof}
	For $\alpha \in L^2_{k,\ext}(W,\Lambda^1)$ in the kernel of $- d^* \oplus d^+$, we have:
	\begin{equation} \label{ASD=>closed-1}
		0=\int_{W}d\alpha\wedge d\alpha=
		\int_{W}|d^+\alpha|^2-|d^-\alpha|^2=-\int_{W}|d^-\alpha|^2\implies d\alpha=0
	\end{equation}
	where the first equality holds by Stokes' theorem and the assumption $k\geq 1$. 
	Therefore, $\ker(- d^* \oplus d^+)$ consists of $L^2_{k,\ext}$ 1-forms that are annihilated by both $d$ and $d^*$. 
	There is a natural map from this space, denoted by $\mathcal H^1(W)$, to $H^1(W)$. 
	We shall show that this map is an isomorphism (cf. also \cite{APS:spec-asym}).\\
	If $\alpha \in \mathcal H^1(W)$ represents a zero cohomology class, then the restriction of 
	$\alpha$ to $\{-t\}\times Y_0 \coprod \{t\}\times Y_1$ for any value of $t$ is an exact 1-form. 
	But this 1-form is exponentially asymptotic to a harmonic 1-form on $ \overline{Y_0} \coprod Y_1$. 
	Therefore, the harmonic 1-form that $\alpha$ is asymptotic to on the ends of $W$ vanishes  and $\alpha \in L^2_{k,\delta}(W,\Lambda^1)$. 
	In particular, $\alpha$ is an $L^2$ harmonic 1-form that represents the zero cohomology class. 
	In the proof of Proposition 4.9 of  \cite{APS:spec-asym} it is shown that any such $\alpha$ has to be zero.
	To address surjectivity, fix a cohomology class $\zeta\in H^1(W)=H^1(W^\circ)$ and let $\beta$ be a closed 1-form on $W^\circ$ representing $\zeta$. 
	By subtracting an exact 1-form, we can assume that $\beta$  in a collar neighborhood of $W^\circ$
	 is the pull-back of a harmonic 1-form on $ \overline{Y_0} \coprod Y_1$. 
	Note that $W^\circ \backslash \partial W^\circ$ is diffeomorphic to $W$. We can use such a diffeomorphism to pull back the 1-form $\beta$ to $W$ to 
	produce the closed 1-form $\gamma$ that its restriction to the ends is pull back of harmonic 1-forms and it represents $\zeta$. 
	Now the Fredholm alternative for the Fredholm operator 
	$d^*d:L^2_{k+1,\delta}(W) \to L^2_{k-1,\delta}(W)$ implies that there exists $f\in L^2_{k+1,\delta}(W)$ such that $d^*d f=d^* \gamma$. 
	Thus $df \in L^2_{k,\delta}(W,\Lambda^1)$
	and the harmonic 1-form $\gamma-df \in L^2_{k,\ext}(W,\Lambda^1)$ represents the cohomology class $\zeta$.
	
	Suppose $(f,\omega)\in L^2_{k,-\delta}(W,\Lambda^0\oplus \Lambda^+)$ is in the co-kernel of $- d^* \oplus d^+$. 
	If $\alpha$ is a smooth 1-form with compact support, then:
	$$0=\int_{W}\langle (d^*\alpha,d^+\alpha),(f,\omega)\rangle=\int_{W}\langle d^*\alpha,f\rangle+\int_{W}\langle d\alpha,\omega\rangle=
	\int_{W}\langle \alpha,df+d^*\omega \rangle$$
	
	Thus $df+d^*\omega=0$. In particular, $(f,\omega)$ on the cylindrical ends is asymptotic to an element in $\ker(Q)$. 
	This in particular justifies the following identities:
	$$||df||_{L^2}^2=||d^* \omega||_{L^2}^2=-\langle df ,d^*\omega\rangle=0 \implies df=0, d^*\omega=0,d\omega=0$$ 
	Furthermore, if $\omega$ is asymptotic to $h \in C_0\oplus C_1$, then for an element $h'\in C_0 \oplus C_1$, we have:
	$$\int_{W}\langle (d^*\varphi(h'),d^+\varphi(h')),(f,\omega)\rangle=\int_{W}\langle \varphi(h'),df+d^*\omega \rangle+
	\int_{\overline{Y_0}\coprod Y_1}\langle h',h \rangle=
	\int_{\overline{Y_0}\coprod Y_1}\langle h',h \rangle \implies h=0$$
	Consequently, $\omega \in L^2_{k,\delta}$. 
	Therefore, cokernel of $- d^* \oplus d^+$ consists of the pairs $(f,\omega)$ where $f$ is a constant function and $\omega$ is an $L^2_{k,\delta}$ self-dual and harmonic
	2-form. By \cite{APS:spec-asym} this space represents $I^+(W)\oplus \rr$.
\end{proof}

\begin{remark} \label{dif-oper-ind}
	Lemma \ref{ind-DW-1} implies that the index of $- d^* \oplus d^+:L^2_{k,\ext}(W,\Lambda^1) \to L^2_{k-1,\delta}(W,\Lambda^0\oplus \Lambda^+)$ is equal 
	to $b^1(W)-b^+(W)-1$. An examination of the long exact sequence of cohomology groups for the pair $(W,Y_0 \coprod Y_1)$ shows that this number is 
	equal to: 
	$$-\frac{\chi(W)+\sigma(W)}{2}+\frac{b_1(Y_0)+b_1(Y_1)}{2}-\frac{b_0(Y_0)+b_0(Y_1)}{2}.$$ 
	There are two other operators that are of interest to us:
	\[\Scale[.85]{\DW:L^2_{k,\ext}(W,\Lambda^1) \to L^2_{k-1,\delta}(W,\Lambda^0)_0 \oplus L^2_{k-1,\delta}(W,\Lambda^+) \oplus C_1 \hspace{1.5cm} 
	\DWP:L^2_{k,-\delta,\delta}(W,\Lambda^1)\to L^2_{k,-\delta,\delta}(W,\Lambda^0\oplus \Lambda^+)}\] 
	\[\Scale[.85]{\DW(\alpha):=(- d^*(\alpha),d^+(\alpha),\iota_1(\alpha)) \hspace{5cm}  \DW(\alpha):=(- d^*(\alpha),d^+(\alpha))}\] 
	where $L^2_{k-1,\delta}(W,\Lambda^0)_0$ is the subspace of $L^2_{k-1,\delta}(W,\Lambda^0)$ that consists of the functions that their integrals over $W$ vanish. 
	An argument similar to the proof of Lemma \ref{ind-DW-1} shows that:  
	\begin{equation} \label{ker-eq}
		\ker (\DW)=\ker (\DWP)=\{\alpha\in L^2_{k,\exp}(W,\Lambda^1) \mid d\alpha=0,d^*\alpha=0,\iota_1(\alpha)=0\}
	\end{equation}
	\begin{equation}	
		\coker (\DW)=\coker (\DWP)=\{\omega \in L^2_{k,-\delta,\delta}(W,\Lambda^+) \mid d\omega=0\}
	\end{equation}
	A step in the proof of (\ref{ker-eq}) involves showing that if $\alpha\in L^2_{k,-\delta,\delta}(W,\Lambda^1)$ is a closed and co-closed 1-form then 
	$\alpha\in L^2_{k,\ext}(W,\Lambda^1)$ with $\iota_1(\alpha)=0$. Because $\alpha\in L^2_{k,-\delta,\delta}$, the 1-form $\alpha$ is asymptotic to zero on the end 
	$Y_1$. On the other hand, the relations $d\alpha=0$ and $d^*\alpha=0$ assert that $\alpha$ is asymptotic to an element of $\ker(Q)$ on the end $Y_0$. 
	Since $Y_0$ is connected, an application of Stokes theorem for the closed 3-form $*\alpha$ shows that this element in $\ker(Q)$ cannot have a component in $\Omega^0(Y)$ and hence 
	$\alpha\in L^2_{k,\ext}(W,\Lambda^1)$.
	Lemma  \ref{ind-DW-1} shows that the index of the operator $\DW$ (and hence $\DWP$) 
	is equal to: 
	\begin{equation} \label{index}
		-\frac{\chi(W)+\sigma(W)}{2}+\frac{b_1(Y_0)-b_1(Y_1)}{2}-\frac{b_0(Y_0)+b_0(Y_1)}{2}+1
	\end{equation}	
\end{remark}

A {\it family of metrics} on a cobordism $W^{\circ}$ parametrized by a manifold $G$ is a fiber bundle $\bW^{\circ}$ with the base $G$ and the fiber $W^{\circ}$ that is equipped with a {\it partial metric}. A partial metric on $\bW^{\circ}$ is a 2-tensor $g\in \Gamma(T\bW^{\circ}\otimes T\bW^{\circ})$ such that its restriction to each fiber is a metric. Furthermore, we assume that in a collar neighborhood of the boundary the metric is the product metric corresponding to the fixed metrics on $Y_0$ and $Y_1$. That is to say, there exists a sub-bundle $([0,1] \times Y_0 \coprod [-1,0]\times Y_1)\times G$ of $\bW^{\circ}$ such that the restriction of the partial metric to this sub-bundle is the product metric for the fixed metric metric on $\partial W$. Adding the cylindrical ends to the fibers of $\bW^{\circ}$ results in a fiber bundle over $G$ such that each fiber is diffeomorphic to $W$. We will write $\bW$ for this family of metrics with cylindrical ends.

The family of metrics $\bW$ define a family of Fredholm operators parametrized by $G$. For each element $g\in G$, the corresponding operator is $\mathcal D_{W^g}$ where $W^g$ is the fiber of $\bW$ over $g\in G$. The index of this family of Fredholm operators is an element of the real $K$-group of the base $G$ (cf. \cite{Ati:Hil}), and is denoted by $\ind(\bW;G)$. We can equivalently use the operator $\DWP$ to define $\ind(\bW;G)$. However, the operator $\DW$ is more suitable for the geometrical set up of this paper. There is only one point (Lemma \ref{add-ind}) that it is more convenient for us to work with $\DWP$.

Above discussion can be further generalized by working with {\it broken Riemannian metrics} \cite{KM:Kh-unknot}. 
A broken Riemannian metric, strictly speaking, is not a metric on $W$. However, it can be considered as the limit of a sequence of metrics on $W$. We firstly discuss model cases for such family of metrics. Let $T$ be a compact orientable codimension 1 sub-manifold of $W$ with $i$ connected components $T_1$, $\dots$, $T_i$ . Also, define $T_0:=Y_0$, $T_{i+1}=Y_1$. We call $T$ a {\it cut}, if removing $T$ decomposes $W$ into a union of $i+1$ cobordisms $W_0$, $\dots$, $W_i$ where $W_k$ is a cobordism with exactly one incoming end, which is $T_k$, and several (possibly zero) outgoing ends.  For each $1\leq k \leq i+1$, the connected 3-manifold $T_k$ appears as one of the outgoing ends of a cobordism that is denoted by $W_{o(k)}$. The simplest arrangement is when $W_k: T_k \to T_{k+1}$ and $W=W_0\circ \dots \circ W_i$. In this case $o(k)=k-1$. For the most part, we are interested in such cuts. However, we need the more general cuts in the proof of exact triangles. It is worthwhile to point out that the formula (\ref{index}) for such decompositions of $W$ is additive, i.e.,  $\ind(\DW)=\ind(\mathcal D_{W_0})+\dots+\ind(\mathcal D_{W_i})$.

A broken metric on $W$ with a cut along $T$ is a metric $g$ with cylindrical ends on the union of the cobordisms $W_0 \coprod \dots \coprod W_i$ such that the product metrics on $\rr^{\geq 0} \times T_k \subset W_{o(k)}$ and $\rr^{\leq 0} \times T_k \subset W_k$ are modeled on the same metric of $T_k$. Given $g$, we can construct a family of (possibly broken) metrics parametrized by $[0,\infty]^i$ on $W$ in the following way: for $(t_1,\dots,t_i)\in [0,\infty]^i$, if $t_k\neq \infty$, we remove the cylindrical ends $\rr^{\geq 0} \times T_k \subset W_{o(k)}$, $\rr^{\leq 0} \times T_k \subset W_{k}$, and glue $[-t_k,t_k] \times T_k$ by identifying $\{-t_k\} \times T_k$ with $\{0\}\times T_k\subset W_{o(k)}$ and $\{t_k\} \times T_k$ with $\{0\}\times T_k\subset W_k$. On the remaining points of $W$, we use the same metric as $g$. More generally, let $\bW_k$ be a family of (non-broken) metrics on $W_k$ parametrized with $G_k$ such that for each $k$, the metrics on the ends  $\rr^{\leq 0}\times T_k$ and $\rr^{\geq 0}\times T_k$, induced by $\bW_k$ and $\bW_{o(k)}$, are modeled on the same metric of $T_k$. Then we can construct a family of (possibly broken) metrics on $W$ parametrized by $[0,\infty]^i\times G_0\times \dots \times G_i$. From this point on, we work with the following more general definition of family of metrics: a family of metrics $\bW$ on the cobordisms $W$ parametrized by the cornered manifold $G$ is a bundle over $G$ with a partial metric such that for any codimension $i$ face of $G$, the family over the interior of the face has the form $\bW_0 \times \dots \times \bW_i$. Furthermore, the family over a neighborhood of the interior of this face has the form of the above family parametrized by $[0,\infty]^i\times G_0\times \dots \times G_i$. For an organized review of manifolds with corners we refer the reader to \cite{LS:KhoHomTyp}. Our treatment of cornered manifolds in this paper, for the sake of simplicity of exposition, is rather informal.

\begin{lemma} \label{add-ind}
	Suppose $\bW$ is a family of metrics on $W$ parametrized by $G$. Then there exists an element of $KO(G)$, called the index bundle of the family $\bW$ and 
	denoted by $\ind(\bW;G)$, such that the following holds. Let $G'$ be a face of $G$ corresponding to a cut $T \subset W$. 
	Assume that $W\backslash T=W_0 \coprod \dots \coprod W_i$, $G'=G_0\times \dots \times G_i$ and the family of metrics over $G'$ has the form 
	$\bW_0\times\dots \times \bW_i $ where $\bW_k$ is a family of (non-broken) metrics on $W_k$ parametrized by $G_k$. 
	Then the restriction of the index bundle to $G'$ is isomorphic to the following element of $KO(G')$:
	\begin{equation} \label{dir-sum}
		\ind(\bW_0,G_0)\oplus \dots \oplus \ind(\bW_i,G_i)
	\end{equation}	
\end{lemma}
Note that $\bW_k$ consists of non-broken metrics and hence $\ind(\bW_k,G_k)$ in (\ref{dir-sum}) is already defined and we do not need to appeal to the lemma to define it. With a slight abuse of notation, $\ind(\bW_j,G_j)$ in (\ref{dir-sum}) denotes an element of $KO(G')$ which is given by the pull-back of $\ind(\bW_j,G_j)\in KO(G_i)$ via the projection map.

\begin{proof}
	This lemma is a global manifestation of the additivity of the numerical index of the operator $\DW$. 
	Firstly let $\bW_0 \times \dots \times \bW_i$ be a family of broken metrics parametrized by $G_1 \times \dots \times G_i$
	on $W$ corresponding to a decomposition of $W$ to cobordisms $W_0$,  $\dots$, $W_i$. Here we assume that $G_j$ is a compact space (not necessarily a manifold)
	that parametrizes the family of non-broken metrics $\bW_j$. These families determine a family of metrics parametrized by 
	$[0,\infty]^i \times G_0\times \dots \times G_i$ on $W$. If $T_0$ is large enough and one uses the operator $\DWP$ in the definition of the index bundles, then the 
	arguments in \cite[section 3.3]{Don:YM-Floer} shows that the pull-back of the
	index bundle:
	$$\ind(\bW_0,G_0)\oplus \dots \oplus \ind(\bW_i,G_i)$$
	 to $(T_0,\infty]^i \times G_0\times \dots \times G_i$ gives an element of $KO((T_0,\infty]^i \times G_0\times \dots \times G_i)$ whose restriction to each face is 
	 isomorphic to the index bundle of the corresponding family of metrics. In the general case, the parametrizing set $G$ is a union of the sets of the above 
	 form. As it can be seen from the arguments of \cite[section 3.3]{Don:YM-Floer}, the isomorphisms in the intersection of these sets can be made compatible.
	 Therefore, one can construct the desired determinant bundle over $G$.
\end{proof}

The orientation bundle of $\ind(\bW;G)$ forms a $\zz/2\zz$-bundle over $G$ which is called the {\it determinant bundle} of the family of metrics $\bW$ and is denoted by $o(\bW)$. For each $g\in G$ the fiber $o(\bW)|_{g}$ consists of the set of orientations of the line $\Lambda^{\max}(\ker(\mathcal \sD_{g}))\otimes (\Lambda^{\max}(\coker(\sD_{g}))^*$. 

As it is mentioned in the proof of Lemma \ref{add-ind}, in the construction of the index bundles, we need to identify the index bundles for broken metrics with the index of non-broken metrics that are converging to the broken ones. These identifications (which are generalization of those of \cite{Don:YM-Floer}) involve choosing cut-off functions and hence are far from being unique. However, the set of choices is contractible. Therefore, any two different choices in the construction of the index bundles give rise to isomorphic determinant bundles and the choice of the isomorphism is unique. 

For a cobordism $W$ and an arbitrary metric $g$ on $W$, the determinant bundle for the one point set $\{g\}$ consists of two points. A homology orientation for $W$ is a choice of one of these two points. Any other metric $g'$  can be connected to $g$ by a path of metrics. This path produces an isomorphism of the determinant bundles for $g$ and $g'$. Furthermore, this isomorphism is independent of the choice of the path. As a result, choice of the homology orientation for one metric determines a canonical choice of the homology orientation for any other metric. Therefore, it is legitimate to talk about homology orientations of $W$ without any reference to a metric on $W$. We will write $o(W)$ for the set of homology orientations of $W$.

Suppose the cobordism $W$ is the composition of cobordisms $W_0:Y_0 \to Y_1$ and $W_1:Y_1 \to Y_2$. Let also $g_0$ and $g_1$ be metrics on these two cobordisms. There is a 1-parameter family of metrics parametrized by $(0,\infty]$ such that the metric over $\infty$ is the broken one on $W=W_0 \circ W_1$ determined by $g_0$ and $g_1$, and the metrics over the other points of $(0,\infty]$ are non-broken. The index bundle for this family over the point $\infty$ has the form $\ind(W_0^{g_0}) \oplus \ind(W_1^{g_1})$. Consequently, this family of metrics produces an isomorphism of $o(W_1)\otimes_{\zz/2\zz} o(W_2)$ and $o(W)$. In particular, if $W_0$ and $W_1$ are given homology orientations, then $W$ inherits a homology orientation that is called the {\it composition} of the homology orientations of $W_0$ and $W_1$. As in the previous paragraph it is easy to see that the composition of homology orientations is independent of the choice of $g_0$ and $g_1$. 

Suppose $W_0$, $\dots$, $W_i$ is a list of cobordisms with homology orientations. These homology orientations can be used to define a homology orientation for $W_0\coprod \dots \coprod W_i$ which is denoted by $o(W_0, \dots, W_i)$. This homology orientation depends on the order of $W_0$, $\dots$, $W_i$. For example:
$$o(W_0,\dots, W_{j+1},W_j,\dots,W_i)=(-1)^{ind(\mathcal D_{W_{j}})\cdot ind(\mathcal D_{W_{j+1}})}o(W_0,\dots, W_j,W_{j+1},\dots,W_i)$$
Next, suppose that there is a cut $T$ in a cobordism $W$ such that $W\backslash T=W_0\coprod \dots \coprod W_i$. Lemma \ref{add-ind} asserts that the homology orientations of  $W_0$, $\dots$, $W_i$ define a homology orientation for $W$. If it does not make any confusion, we denote this homology orientation of $W$ with $o(W_0, \dots, W_i)$, too. This is a generalization of the composition of homology orientations in the previous paragraph.

\subsection{Abelian ASD equation} \label{ab-asd-eq}

Suppose $Y$ is a (possibly disconnected) closed Riemannian 3-manifold. A $\spinc$ structure $\st$ on $Y$ is a principal $Spin(3)$-bundle (or equivalently a principal $U(2)$-bundle) $P$ such that the induced $SO(3)$-bundle, determined by the adjoint action $\ad:Spin^c(3) \cong U(2) \to SO(3)$, is identified with the framed bundle of $TY$. We can also use the determinant map $\det:U(2) \to U(1)$ to construct a complex line bundle that is called the {\it determinant bundle of $\st$} and is denoted by $L_{\st}$. A smooth connection $\underline B$ on $\st$ is $\spinc$ if the induced connection on $\ad(\st)$ is the Levi-Civita connection. The $\spinc$ connection $\underline B$ also determines a connection on $L_\st$ that is called the {\it central part} of $\underline B$ and is denoted by $B$. Since the data of connections on $\st$ is equivalent to that of connections on $TY$ and $L_{\st}$, a $\spinc$ connection $\underline B$ is uniquely determined by its central part. We will write $A_f(Y,\st)$ for the space of all $\spinc$ connections on $\st$ with flat central parts. Note that $L_{\st}$ admits a flat connection if and only if $c_1(L_\st)$ is a torsion element of $H^2(Y;\zz)$, and then any two flat connections on this bundle differs by a closed 1-form on $Y$. Therefore, if $c_1(L_\st)$ is torsion, then $A_f(Y,\st)$ is an affine space modeled on the space of closed 1-forms.

An automorphism of the $\spinc$ structure $\st$ is a smooth automorphism of $\st$ as a principal $U(2)$-bundle that acts trivially on the tangent bundle of $Y$. The {\it gauge group of $Y$}, denoted by $\mathcal G(Y)$, is the group of all such elements, and can be identified with the space of smooth maps from  $Y$ to $S^1$. The space $\mathcal G(Y)$ acts on the space of $\spinc$ connections by pulling back the connections. Given $u: Y \to S^1$ and a $\spinc$ connection $\underline B$ on $\st$, this action sends $\underline B$  to a connection with the central part $B-2u^{-1}du$. Therefore, this action changes the central part $B$ by the twice of an integral closed 1-form. The quotient of $A_f(Y,\st)$ with respect to the action of $\mathcal G(Y)$ is denoted by $\mathcal R(Y,\st)$. Let $\mathfrak t$ be a torsion $\spinc$ structure, and fix a base point in $A_f(Y,\st)$. Then $\mathcal R(Y,\st)$ can be identified with $J(Y):=H^1(Y;\rr)/2H^1(Y;\zz)$ which is a rescaling of the Jacobian torus of $Y$. We will write $\mathcal R(Y)$ for the union $\cup_\st \mathcal R(Y,\st)$. Therefore, this space has a copy of $J(Y)$ for each torsion $\spinc$ structure $\st$.

Next consider a cobordism $W^{\circ}:Y_0\to Y_1$ as in the previous section. In particular, we assume that $Y_0$ is a connected 3-manifold. As in the case of 3-manifolds, a $\spinc$ structure $\s$ on $W^{\circ}$ (or equivalently $W$) is a principal $Spin^c(4)$-bundle such that the induced $SO(4)$-bundle, induced by the adjoin action $\ad:Spin^c(4) \to SO(4)$, is identified with  the frame bundle of $TW^{\circ}$. This principal bundle gives rise to a determinant bundle that is denoted by $L_{\s}$. A $\spinc$ connection $\underline A$ on $\s$ is also uniquely determined by the induced connection on $L_\s$. Again, this induced connection is called the central part of $\underline A$ and is denoted by $A$. Here we are interested in $\spinc$ connections defined on $W$ (rather than $W^{\circ}$) and for analytical purposes we have to work with connections which are in an appropriate Sobolev space and behaves in a controlled way on the ends of $W$. To give the precise definition of this space of connections, fix a connection $\underline {A_0}$ on $\s$ such that the restriction of $A_0$ to the ends $\rr^{\leq 0} \times Y_0 \coprod \rr^{\geq 0} \times Y_1$ is the pull-back of flat connections on $Y_0$ and $Y_1$. The {\it space of weighted Sobolev $\spinc$ connections} for the $\spinc$ structure $\mathfrak s$ is defined as:
$$ \mathcal A(W,\s):=\{\underline A \mid A=A_0+ia, a\in L^2_{k,\ext}(W,\Lambda^1)\}$$
where $k$ is an integer number greater than $1$. 

The restriction of $\mathfrak s$ to $\{0\}\times Y_0$ and $\{0\}\times Y_1$ gives rise to the $\spinc$ structures $\st_0$ and $\st_1$ on $Y_0$ and $Y_1$, respectively. Furthermore,  the limit of $\underline{A}|_{ \{-t\}\times Y_0}$ and $\underline{A}|_{ \{t\}\times Y_1}$ as $t$ goes to $\infty$ induces $\spinc$ connections on $\st_0$ and $\st_1$ with flat central parts. Therefore, there is a map $R: \mathcal A(W,\s) \to A_f(Y_0,\st_0) \times A_f(Y_1,\st_1)$.  Note that if $\st_i$ is non-torsion, then $\mathcal A(W,\mathfrak s)$ is empty, and from now on we assume that $\mathfrak s|_{Y_i}$ is torsion. 

A smooth map $w:W \to S^1$ is {\it harmonic on the ends} if the restriction of $w$ to $ \rr^{\leq0} \times Y_0 \coprod \rr^{\geq0} \times Y_1$ is pull-back of a harmonic circle-valued map on $Y_0$ and $Y_1$. The {\it weighted gauge group} $\mathcal G(W)$ is defined as follows:
$$\mathcal G(W):=\{u:W \to S^1 \mid \exists v,w: \hspace{1mm} u=vw, v \in L^2_{k+1,\delta}(W,S^1),\hspace{1mm} w:W \to S^1 \text{ is harmonic on the ends}\}$$
Here $L^2_{k+1,\delta}(W,S^1)$ is the set of maps $v:Y \to S^1$ such that $v^{-1}dv \in L^2_{k,\delta}(W,\cc)$. Roughly speaking, an element of $\mathcal G(W)$ is exponentially asymptotic to harmonic circle-valued maps on $Y_0$ and $Y_1$. It is straightforward to see $\mathcal G(W)$ is an abelian Banach Lie group with the Lie algebra $L^2_{k+1,\delta}(W,i\rr)$.
The set of connected components of $\mathcal G(W)$ can be identified with $H^1(W,\zz)\cong [W,S^1]$. The gauge group $\mathcal G(W)$ acts on $\mathcal A(W,\mathfrak s)$ by pulling-back the connections. If $u:W\to S^1$ is an element of $\mathcal G(W)$ and $\underline A$ is a $\spinc$ connection, then the action of $u$ maps $\underline A$ to a $\spinc$ connection with the central part $A-2u^{-1}du$. The stabilizer of any point in $\mathcal A(W,\mathfrak s)$ consists of the constant maps in $\mathcal G(W)$. 

The quotient space $\mathcal B(W,\mathfrak s):=\mathcal A(W,\mathfrak s)/\mathcal G(W)$ is a smooth Banach manifold and is called the {\it configuration space of weighted Sobolev $\spinc$ connections} on $\s$. The tangent space at $[\underline A] \in \mathcal B(W,\mathfrak s)$ is isomorphic to the kernel of $d^*:L^2_{k,\ext}(W,i\Lambda^1)\to L^2_{k-1,\delta}(W,i\rr)_0$. Moreover, the action of the gauge group is compatible with $R: \mathcal A(W,\s) \to A_f(Y_0,\st_0) \times A_f(Y_1,\st_1)$. Thus $R$ produces the {\it restriction maps} $r_0: \mathcal B(W,\s) \to \mathcal R(Y_0,\st_0)$ and $r_1: \mathcal B(W,\s) \to \mathcal R(Y_1,\st_1)$. We will write $\mathcal B(W)$ for the union $\bigcup_\mathfrak s \mathcal B(W,\mathfrak s)$. There are ressriction maps from $\mathcal B(W)$ to $\mathcal R(Y_0)$ and $\mathcal R(Y_1)$ which are also denoted by $r_0$ and $r_1$.

\begin{remark}
	In fact, $\mathcal B(W,\mathfrak s)$ can be globally identified with the quotient of $\ker(d^*)$ 
	with respect to the action of the discrete group $H^1(W,\zz)$.
\end{remark} 

By definition, the central part of a connection $\underline A \in \mathcal A(W,\mathfrak s)$ can be written as a sum of the smooth connection $A_0$ that is flat on the ends and an imaginary 1-form $ia\in L^2_{k,\ext}(W,i\Lambda^1)$. Thus $F_c(\underline A)$, the curvature of $A$, is an element of $L^2_{k-1,\delta}(W,i\Lambda^2)$. Consequently, $F_c^+(\underline A)$, the self-dual part of the curvature of $A$, is in $L^2_{k-1,\delta}(W,i\Lambda^+)$. The action of the gauge group does not change the curvature and we have a well-defined map $F_c^+: \mathcal B(W,\mathfrak s) \to L^2_{k-1,\delta}(W,i\Lambda^+)$. In local coordinates around $[\underline A]$, this map is equal to $d^++F_c^+(\underline A): \ker(d^*)\subset L^2_{k,\ext}(W,i\Lambda^1) \to L^2_{k-1,\delta}(W,i\Lambda^+)$ where $d^+$ is the self-dual part of the exterior derivative. For an arbitrary $\nu\in L^2_{k-1,\delta}(W,\Lambda^2)$, define the following moduli spaces:
$$M_\nu(W,\mathfrak s):=\{[\underline A]\in \mathcal B(W,\mathfrak s) \mid F_c^+(\underline A)=i\nu^+\}$$
$$M_\nu(W):=\bigcup_\mathfrak{s} M_\nu(W,\mathfrak s)$$

If there is a chance of confusion about the metric $g$ on $W$ which is used to define $M_\nu(W)$, we denote this moduli space with $M_\nu(W^g)$. With a slight abuse of notation, we will also write $r_0:M_\nu(W) \to \mathcal R(Y_0)$ and $r_1:M_\nu(W) \to \mathcal R(Y_1)$ for the restriction maps restricted to the moduli space $M_\nu(W)$.

\begin{lemma} \label{ind-F^+-1}
	The map $F_c^+:\mathcal B(W) \to L^2_{k-1,\delta}(W,i\Lambda^2)$ is smooth and Fredholm 
	with $\dim(\ker)=b^1(W)$ and $\dim({\rm coker})=b^+(W)$. 
\end{lemma}
\begin{proof}
	Locally, the map $F_c^+$ is affine and hence is smooth. The rest of the lemma follows easily from Lemma \ref{ind-DW-1}.
\end{proof}
This lemma implies that the index of $F_c^+$ is equal to: 
$$b^1(W)-b^+(W)=-\frac{\chi(W)+\sigma(W)}{2}+\frac{b_1(Y_0)+b_1(Y_1)}{2}.$$ 
A global description of the moduli space $M_\nu(W)$ is given in the following lemma:

\begin{lemma} \label{cob-moduli-1}
	The moduli space $M_\nu(W,\mathfrak s)$ is either empty or can be identified with $J(W):=H^1(W;\rr)/2H^1(W;\zz)$, 
	and this identification is canonical up to a translation on $J(W)$. If $b^+(W)>0$, then $\nu$ 
	can be chosen in such a way
	that $M_\nu(W,\mathfrak s)$ is empty, 
	and in the case $b^+(W)=0$, the moduli space, for any choice of $\nu$, is isomorphic to $J(W)$.
\end{lemma}
\begin{proof}
	Suppose $M_\nu(W,\s)$ is not empty and $[\underline{A_0}]\in M_\nu(W,\mathfrak s)$. 
	For $[\underline A] \in \mathcal B(W,\mathfrak s)$:
	$$F^+_c(\underline A)=F^+_c(\underline {A_0}+(\underline A-\underline {A_0}))=
	F^+_c(\underline {A_0})+d^+(A-A_0)=\nu^++d^+(A-A_0)$$
	Thus $[\underline A]\in M_\nu(W,\mathfrak s)$	if and only if $d^+(A-A_0)=0$, 
	which in turn is equivalent to $d(A-A_0)=0$ by (\ref{ASD=>closed-1}). 
	Two connections $\underline A$ and $\underline A'$ give rise to the same element of $M_\nu(W,\mathfrak s)$
	if $A'=A-2u^{-1}du$ for $u \in \mathcal G(W)$, i.e., $A'$ and $A$ differ by twice of an integral closed 1-form. 
	In summary, after fixing $\underline {A_0}$, $M_\nu(W,\mathfrak s)$ can be identified with 
	$J(W)=H^1(W; \rr)/2H^1(W;\zz)$. 
	
	In order to prove the second part, start with an arbitrary $\spinc$ connection $\underline{A_0}$
	and try to modify this element by $a\in L^2_{k,\delta}(W,\Lambda^1)$:
	\begin{equation} \label{mod-space-non-empty-1}
		F^+(A_0+ia)=i\nu^+ \iff d^+ia=i\nu^+-F^+(A_0)
	\end{equation}
	Lemma \ref{ind-DW-1} states that $d^+$ is surjective in the case $b^+(W)=0$, 
	and hence Equation (\ref{mod-space-non-empty-1}) has solution for any choice of $\nu$. 
	If $b^+(W)>0$, then $d^+$ is not
	surjective and $\nu$ can be picked such that 
	Equation (\ref{mod-space-non-empty-1}) does not have any solution.
\end{proof}
\begin{remark} \label{res-recast-1}
	By our previous discussions, $J(Y)$ and $J(W)$ act on $\mathcal R(Y,\st)$ and $M_\nu(W,\mathfrak s)$, respectively. These actions
	give isomorphisms of $M_\nu(W,\mathfrak s)$, $\mathcal R(Y,\st)$ 
	with $J(W)$, $J(Y)$, in the case that $M_\nu(W,\mathfrak s)$, $\mathcal R(Y,\st)$ are non-empty.
	If these isomorphisms are chosen appropriately, 
	then the restriction map $r=(r_0,r_1):M_\nu^g(W,\mathfrak s)\to\mathcal R(Y_0,\st_0) \times \mathcal R(Y_1,\st_1)$ 
	can be identified with the restriction map $i=(i_0,i_1):J(W) \to J(Y_0)\times J(Y_1)$, induced by the inclusion of $Y_0$ and $Y_1$ in $W$.
\end{remark}

We can also define similar moduli spaces  in the case that $W$ is equipped with a broken metric $g$. Suppose $g$ is a broken metric in correspondence with the cut $T$, and removing $T$ decomposes $W$ into the union of cobordisms $W_0$, $\dots$, $W_i$. Let $T_1$, $\dots$,$T_i$ be the connected components of $T$ such that $T_j$ is the incoming end of the cobordism $W_j$ and one of the outgoing ends of $W_{o(j)}$. Let also $T_0= Y_0$ and $T_{i+1}=Y_1$. The metric $g$ induces a metric $g_j$ on the cobordism $W_j$ and hence we can define the moduli spaces $M_{\nu_j}(W_j^{g_j})$ for a perturbation term $\nu_j$. Since $T_j$ is the incoming end of $W_j$, there is a restriction map $r^0_j:M_{\nu_j}(W_j^{g_j}) \to \mathcal R(T_j)$. Furthermore, there is a restriction map from $M_{\nu_{o(j)}}(W^{g_{o(j)}}_{o(j)})$ to $\mathcal R(T_j)$ which will be denoted by $r_j^1$. Define:
$$M_{\nu}(W^{g}):=\{([\underline A_0],\dots,[\underline A_{i}])\in M_{\nu_0}(W_0^{g_0})\times \dots \times M_{\nu_i}(W_i^{g_{i}})\mid r^j_0([\underline A_{j}])=r_1^{j}([\underline A_{o(j)}])\hspace{3mm} 1\leq j\leq i\}$$
One can still define the restriction maps  $r_0:M_\nu(W^g)\to\mathcal R(Y_0)$ and $r_1:M_\nu(W^g)\to\mathcal R(Y_1)$ with the aid of the restriction maps on the moduli spaces $M_{\nu_0}(W^{g_0}_0)$ and $M_{\nu_{o(i+1)}}(W^{g_{o(i+1)}}_{o(i+1)})$. 

The broken metric $g$ determines a family of metrics parametrized with $[0,\infty]^i$ on $W$. For each element $g'=(t_1, \dots, t_i)\in [1,\infty]^i$ of this family and $([\underline A_0],\dots,[\underline A_{i}]) \in M_{\nu}(W^{g})$, we can construct a $\spinc$ connection on $W$. To that end, fix a function $\chi:\rr \to \rr^{\geq 0}$ such that $f(t)=1$ for $t\leq \frac{1}{2}$ and $f(t)=0$ for $t\geq 1$. This function can be used to define a function $\chi_{t_j}:W_j \to \rr^{\geq 0}$ for $0\leq j \leq i$. The function $\chi_{t_j}$ is defined to be equal to 1 on the compact part $W_j^{\circ}$. For a point $(t,y)\in \rr^{\leq 0} \times Y_j$ in the incoming end, define $\chi_{t_j}(t,y):= \chi(-\frac{t}{t_j})$. Use a similar definition to extend $\chi_{t_j}$ to the outgoing ends. Now suppose $\underline A'_j$ is the $\spinc$ connection on $W_j$ which is equal to $\underline A_j$ on $W_j^\circ$ and is equal to $\chi_{t_j} \underline A_j+ (1-\chi_{t_j})\pi_2^*(R_{k}(\underline A_j))$ for each end of the form $\rr^{\leq 0}\times T_k$ or $\rr^{\geq 0}\times T_k$ where $1\leq k \leq i$. Here $\pi_2:\rr^{\leq 0}\times T_k \coprod \rr^{\geq 0} \times T_k \to T_k$ is the projection to the second factor, and $R_k: \mathcal A(W) \to A_f(T_k)$ is the analogue of the map $R$, defined earlier. Therefore, $\underline A'_j$ is a $\spinc$ connection whose restriction to the subset $(-\infty,-t_j]\times T_j$ is pull back of a connection on $T_j$. A similar property holds for the outgoing ends. Because each cobordism $W_j$ has one incoming end, there is a unique way to glue the connections $\underline A_j'$ to define an element $\phi_{g'}(\underline A_0,\dots,\underline A_{i})$ of the configuration space of $\spinc$ connections on $W$ with the (possibly broken) metric $g'$. Note that the  construction of this element depends on the choice of the lift of $[\underline A_j]\in \mathcal B(W_j)$ to $\underline A_j \in \mathcal A(W_j)$. However, for a non-broken metric $g'\in [1,\infty)^i$, $\phi_{g'}(\underline A_0,\dots,\underline A_{i})$ determines a $\spinc$ connection on $W$ that is independent of the choice of $g'$. We assign this $\spinc$ connection to $([\underline A_0],\dots,[\underline A_{i}])$. In particular, the partition of the moduli space $M_{\nu}(W^{g})$ with respect to the $\spinc$ structures is still well-defined in the case of broken metrics. 

Let $\bW$ be a family of metrics on $W$ parametrized by a cornered manifold $G$. For each choice of integers $l$ and $k$, we shall construct a space $\mathcal V^l_k(\bW)$ with a projection map $\pi: \mathcal V^l_k(\bW) \to G$. The fiber of $\pi$ over a non-broken metric $g\in G$ is equal to $L^2_{k,\delta}(W^g,\Lambda^l)$. These fibers together form a Banach bundle on the interior of $G$. Let $g\in G$ be a broken metric determined by a Riemannian metric on the decomposition $W_0\coprod \dots \coprod W_i$ of the complement of a cut $T$ in $W$. We define the fiber of $\mathcal V^l_k(\bW)$ over $g$ to be equal to $L^2_{k,\delta,c}(W_0,\Lambda^l)\oplus L^2_{k,c,c}(W_1,\Lambda^l) \oplus\dots \oplus L^2_{k,c,c}(W_{i-1},\Lambda^l) \oplus L^2_{k,c,\delta}(W_i,\Lambda^l)$. The space $L^2_{k,\delta,c}(W_0,\Lambda^l)$ is the subspace of elements of $L^2_{k,\delta}(W_0,\Lambda^l)$ that are supported in $\rr^{\leq 0} \times Z_0\cup W_0^\circ\cup [0,1]\times (\partial^{out}W_0)$ where $\partial^{out}W_0$ is the outgoing boundary of $W_0$. The index $c$ in the other Banach spaces should be interpreted similarly. If we instead use the Banach space $L^2_{k,c,c}(W_0,\Lambda^l)\oplus L^2_{k,c,c}(W_1,\Lambda^l) \oplus\dots \oplus L^2_{k,c,c}(W_{i-1},\Lambda^l) \oplus L^2_{k,c,c}(W_i,\Lambda^l)$, the resulting space will be dented by $\mathcal V^l_k(\bW^\circ)$. Starting with such a metric, there is a family of metrics on $W$ parametrized by $[0,\infty]^i$ that forms part of $G$. The point of considering forms that vanish on the intermediate ends is that any element of $L^2_{k,\delta,c}(W_0,\Lambda^l)\oplus \dots \oplus L^2_{k,c,\delta}(W_i,\Lambda^l)$ determines an element of $L^2_{k,\delta}(W^{g'},\Lambda^l)$ where $g'\in [0,\infty]^i$. We can also define $\mathcal V_k^+(\bW)$ as a subspace of $\mathcal V^2_k(\bW)$ that consists of the self-dual 2-forms.

The definition of the moduli spaces can be extended to the family of metrics $\bW$. A perturbation term for this family is a smooth section $\eta: G \to \mathcal V^2_{k-1}(\bW^\circ)$. Smoothness of $\eta$ at the boundary points of $G$ should be interpreted in the following way. Fix a face $G'=G_0 \times \dots \times G_i$ of codimension $i$ in $G$. This face parametrizes a family of metrics broken along a cut $T$. Removing $T$ decomposes $W$ into the union of cobordisms $W_0 \coprod \dots \coprod W_i$, and $G_k$ parametrizes a family of metrics $\bW_k$ on $W_k$ such that $\bW|_{G'}=\bW_0 \times \dots \times \bW_i$. Smoothness of $\eta$ over the face $G'$ implies that there are $\eta_k\in\mathcal V_{k-1}^2(\bW_k^{\circ})$ and a neighborhood of $G'$ in $G$ of the form $[0,\infty]^i\times G_0 \times \dots \times G_i$ such that:
$$\eta(t_1,\dots,t_i,g_0,\dots,g_i)|_{W_k}=\eta_k(g_k)$$

The moduli space $\meGb$ is defined to be:
$$\bigcup_{g\in G}M_{\eta(g)}(W^g)$$
Again, use $\pi$ to denote the projection map from $\meGb$ to $G$. The definition of topology on this moduli space is standard and we refer the reader to \cite{KM:monopoles-3-man}. Given a face $G'$, we will write $\meGbr$ for the part of the moduli space that is mapped to $G'$ by $\pi$. 

The decomposition of the moduli spaces $\meGb$ with respect to $\spinc$ structures on $W$ can be recovered partly. Firstly let $[\underline A]\in \mathcal B(W,\s)$. Then the restriction of $\s$ to $\partial W$ is a torsion $\spinc$ structure and hence self-intersection of $c_1(\s)$, as a rational cohomology class, is well-defined. Therefore:
$$E([\underline A]):=- c_1(\s)\cdot c_1(\s)=-c_1(L_\s)\cdot c_1(L_\s)=\frac{1}{4\pi^2}\int_W F(A)\wedge F(A)$$
assigns a rational number to $[\underline A]$ that is called the {\it energy} of $[\underline A]$. In fact, if $m$ is the number of torsion elements in $H^2(\partial W,\zz)$, then the restriction of $mc_1(\s)$ to $\partial W$, as an integer class, vanishes. Consequently, $m^2e([\underline A])$ is an integer number. In particular, the set of possible values for $E([\underline A])$ is a discrete subset of the rational numbers. If $g$ is a broken metric, inducing the decomposition $W_1\coprod \dots \coprod W_i$ of $W$, and $[\underline A]\in M_{\nu}(W^{g})$, then:
$$E([\underline A]):= \sum_{j=1}^iE([\underline A_j])$$
where $[\underline A_j]$ is the induced connection on $W_j$. If $\s$ is the $\spinc$ structure associated with the broken $\spinc$ connection $[\underline A]$, then it is still true that $E([\underline A])=- c_1(\s)\cdot c_1(\s)$. Therefore, the energy of an element of the moduli space $M_{\nu}(W^{g})$ again has the form $\frac{k}{m^2}$ for an integer number $k$.

For a metric $g$ (broken or non-broken), let $M_{\nu}(W^{g},e)$ be the set of elements of $M_{\nu}(W^{g})$ with energy $e$. More generally, if $\bW$ is a family of metrics, then define:
$$\meGbe:=\{[\underline A]\in \meGb \mid E([\underline A])=e\}$$
Since the set of possible values for $e$ is discrete, $\meGbe$ is a union of the connected components of $\meGb$. 

\begin{lemma} \label{finite-spinc}
	Suppose $\bW$ is a family of metrics parametrized by a compact cornered manifold $G$. For any choice of a smooth perturbation $\eta$, the moduli space 
	$\meGbe$ is compact. Moreover, there is a positive number $\epsilon$ such that if $||\eta(g)||_2< \epsilon$ for any $g\in G$, 
	then $\meGbe$ is empty for negative values of $e$. 
\end{lemma}

\begin{proof}
	Suppose $[\underline A] \in \meGbe$ and $g\in G$ is the corresponding metric. Then:
	\begin{equation} \label{F^+-F^-_1}
		E([\underline A] )= -\frac{1}{4\pi^2} (||F^+(A)||^2_2-||F^-(A)||^2_2)\geq -\frac{1}{4\pi^2} ||\eta(g)||_2^2
	\end{equation}	
 	This inequality verifies the second claim in the lemma, because the set of possible values for $E$ is discrete.
	
	Compactness of $\meGbe$ can be proved by a standard argument. For the convenience of the reader, we sketch the main steps of the proof. 
	Let $\{z_i\}_{i\in \nn}\subset  \meGbe$ and $\pi([\underline z_i])=g_i$. Compactness of $G$ implies that,
	after passing to a subsequence, there is $g_0\in G$ such that the sequence converges to $g_0$. 
	By abuse of notation, we use the same notation $\{z_i\}_{i\in \nn}$ each time that we pass to a subsequence. 
	Firstly assume that $g_0$ is in the interior of $G$ and parametrizes a non-broken metric on $W$. Therefore, we can assume that $g_i$ is a non-broken metric.
	By trivializing $\bW$ in a neighborhood of $g_0$, each element of the sequence $\{z_i\}$ is represented by a $\spinc$ connection $\underline A_i$ on a fixed cobordism 
	$W$. As a consequence of (\ref{F^+-F^-_1}): 
	\begin{equation} \label{bound-L^2}
		||F(A_i)||_2^2= ||F^{+_i}(A_i)||^2_2+||F^{-_i}(A_i)||^2_2\leq 4\pi^2 e+2||\eta(g_i)||_2^2
	\end{equation}	
	where $F^{+_i}(A_i)$ and $F^{-_i}(A_i)$ are the self-dual and the anti-self-dual parts of $F(A_i)$ with respect to the metric $g_i$. 
	The $L^2$ norm in (\ref{bound-L^2}) is also computed with respect to $g_i$. Because $g_i$ is convergent to $g_0$ and $\eta$ is a continuous section over $G$, 
	the $L^2$ norm of $F(A_i)$ is uniformly bounded. By Uhlenbeck compactness theorem, there is a sequence $u_i\in \mathcal G(W)$ such that after passing
	to a subsequence, $u_i \cdot \underline A_i$ is $L^2_k$ convergent over a regular neighborhood of $W^{\circ}$. 
	Note that we are using the Uhlenbeck compactness theorem for $U(1)$ connections (which is 
	essentially a result of the Hodge theory). In particular, the uniform bound on $||F(A_i)||_2$ can be an arbitrary number (as oppose to the non-abelian case that the 
	uniform bound has to be small enough.) Since the perturbations $\eta(g_i)$ is zero on the cylindrical ends, the restriction of the connection $A_i$ to the ends is ASD 
	and the technique of \cite[Chapter 4]{Don:YM-Floer} can be invoked to show that, after passing to a subsequence, the restriction of $A_i$ to the ends is 
	$L^2_{k,\delta}$-convergent. (In fact $k$ can be any arbitrary positive number.)  Consequently, the constructed subsequence is convergent over $W$ 
	by a standard patching argument \cite{DK:geo-4-man}. It is worth noting that because of the linear nature of our moduli spaces, the energy of a subsequence of the 
	elements of $\meGbe$ does not slide off the ends.

	Next, let $g_0$ be a broken metric and $W_0 \coprod \dots \coprod W_i$ be the corresponding decomposition of $W$. A similar argument as before proves the 
	convergence of a subsequence on the compact set $W_0^\circ \coprod \dots \coprod W_i^\circ$, the incoming end of $W_0$ and the outgoing end of $W_{o(i+1)}$.
	For the intermediate ends, an analogue of \cite[Proposition 4.4]{Don:YM-Floer} can be applied to finish the proof of compactness.
\end{proof}

We need to show that the moduli space $\meGb$ is a ``nice'' space for a ``generic'' choice of $\eta$. We start with the case that $\bW$ consists of only non-broken metrics. Define the {\it family of configuration spaces} $\mathcal B(\bW)$ to be the Banach bundle over $G$ that its fiber over $g\in G$ is $\mathcal B(W^g)$. We denote an element of $\mathcal B(\bW)$ by $([\underline A],g)$ to emphasize that $[\underline A]$ is in the fiber over $g$. The space $\mathcal V_{k-1}^+(\bW)$ also defines a Banach bundle over $G$. For a perturbation $\eta:G \to \mathcal V^2_{k-1}$, we can define:
$$F_G^+:\mathcal B(\bW) \to \mathcal V_{k-1}^+(\bW)$$
$$F_G^+([\underline A],g)=F^{+_g}(A)-i\eta(g)^{+_g}$$
 The moduli space $\meGb$ can be realized as the inverse image of the zero section of $\mathcal V_{k-1}^+(\bW)$ by the map $F_G^+$.  Lemma \ref{ind-F^+-1} implies that $F^+_G$ is Fredholm and its index is equal to $b^1(W)-b^+(W)$. Therefore, if $F^+_G$ is transverse to the zero section of $\mathcal V_{k-1}^+(\bW)$, then $\meGb$ is a smooth manifold of dimension $b^1(W)-b^+(W)+\dim(G)$. The following lemma states that for an appropriate choice of $\eta$ a stronger version of this transversality assumption holds:

\begin{lemma} \label{super-transversality-1}
	Let $\bW$ be a family of non-broken metrics on the cobordism $W$ parametrized by the manifold $G$. 
	A smooth manifold $X$ and a smooth map $\Phi: X \to \mathcal R(Y_0) \times \mathcal R(Y_1)$ are also given. 
	Then there is a perturbation $\eta$ such that $F^+_G$ is transversal to the zero section of $\mathcal V_{k-1}^+(\bW)$ and 
	$r=(r_0,r_1):\meGb \to \mathcal R(Y_0) \times  \mathcal R(Y_1)$ is transversal to $\Phi$.
\end{lemma}

\begin{proof}
	Firstly it will be shown that there is $\eta$ such that the first required transversality holds.
	Choose two open covers $\{U_i\}$ and $\{V_i\}$ of $G$ such that:
	$$G=\bigcup_{i=1}^n U_i =\bigcup_{i=1}^n V_i \quad \overline {V_i}\subset U_i$$
	and $\bW|_{U_i}$ is a trivial bundle.
	We construct the perturbation $\eta$ inductively on $\bigcup_{i=1}^k U_i$ 
	such that the transversality assumption holds on $\bigcup_{i=1}^k V_i$. Suppose $\eta$ is 
	constructed on $\bigcup_{i=1}^k U_i$ and $\phi_1$ and $\phi_2$ are two smooth functions on $G$ with the following properties: $\phi_1$ is supported in 
	$\bigcup_{i=1}^k U_i$ and is non-zero on $\bigcup_{i=1}^k V_i$. Similarly, $\phi_2$ is supported in $U_{k+1}$ and non-zero on $V_{k+1}$. 
	Furthermore, we assume that $\phi_1+\phi_2|_{\bigcup_{i=1}^{k+1} V_i}=1$ .
	Define:
	\begin{align}
		\Psi&:\mathcal B(\pi^{-1}(\bigcup_{i=1}^{k+1} U_i)) \times L^2_{k-1,c}(W,\Lambda^2) \to\mathcal  V_{k-1}^+(\bW)\nonumber \\
		&([\underline A],g,\nu) \to F^{+_g}(A)-i\phi_1\eta^{+_g}(g)-i\phi_2 \nu^{+_g} \nonumber
	\end{align}	
	Note that the fibers of $\bW$ over $U_{k+1}$ can be identified with $W$. In the definition of $\Psi$ we fix one such identification and hence elements of  	
	$L^2_{k-1,c}(W,\Lambda^2)$ can be 
	considered as 2-forms on $W^g$ when $g\in U_{k+1}$. Because of our assumption on the supports of $\phi_1$ and $\phi_2$, it is clear that $\Psi$ is well-defined.
	
	For $g\in\bigcup_{i=1}^{k+1} V_i$, if $\phi_2(g)=0$, then $g\in \bigcup_{i=1}^k V_i$ and $\phi_1(g)=1$. Our assumption on $\eta$ implies that $\Psi$ is 
	transverse to the zero section in the fiber over $g$. 
	If $\phi_2(g)>0$ and $\Psi([\underline A],g,\nu)=0$, then post-composing the derivative of $\Psi$ at $f=([\underline A],g,\nu)$ with the projection to the fiber of 
	$\mathcal  V_{k-1}^+(\bW)$ gives rise to a linear map: 
	\begin{align}
		D_f\Psi&: \ker(d^*)\oplus T_gG \oplus L^2_{k-1,c}(W,\Lambda^2) \to L^2_{k-1,\delta}(W,\Lambda^{+_{g}})\nonumber \\
		&(a,0,\mu) \to d^{+_{g}}(a)-\phi_2\mu^{+_{g}} \nonumber
	\end{align}
	In order to construct the right inverse for $D_f\Psi$, observe that the set of harmonic $g$-self-dual forms, denoted by $\mathcal H^+$, 
	does not have any non-zero element that is orthogonal to $L^2_{k-1,c}(W,\Lambda^2)$. 
	Because otherwise that element would be supported in the cylindrical ends and this is absurd according to the discussion of the solutions of the equation
	$\frac{d}{dt}-Q$ in the previous subsection.
	Therefore, there is a linear map $T_1: \mathcal H^+ \to L^2_{k-1,c}(W,\Lambda^{+_g})$ such that 
	$$\omega-T_1(\omega) \in (\mathcal H^+)^\bot \subset L^2_{k-1,\delta}(W,\Lambda^{+_g}).$$
	where the orthogonal complement $(\mathcal H^+)^\bot$ of $\mathcal H^+$ is defined with respect to the $L^2$ norm defined by $g$.
	On the other hand, there is a linear operator $T_2: (\mathcal H^+)^\bot \to \ker(d^*)$ such that $d^{+_g} \circ T_2=id$. 
	Using these two operators, we can construct a right inverse for $D_f\Psi$:
	$$L: L^2_{k-1,\delta}(W,\Lambda^{+_{g}})=\mathcal H^+\oplus (\mathcal H^+)^\bot \to \ker(d^*)\oplus T_gG \oplus L^2_{k-1,c}(W,\Lambda^2)$$
	$$L(\omega,v)=(T_2(v+\omega-T_1(\omega)),0,-T_1(\omega))$$
	
	By the implicit function theorem for maps between Banach manifolds, the inverse image of the zero section, 
	$\Psi^{-1}(0) \subset\mathcal B(\pi^{-1}(\bigcup_{i=1}^{k+1} U_i)) \times L^2_{k-1,c}(W,\Lambda^2)$, 
	is a regular Banach manifold.
	The projection map from the Banach manifold $\Psi^{-1}(0)$ to 
	$L^2_{k-1,c}(W,\Lambda^2)$ is a Fredholm map whose index is equal to $dim(G)+b^1(W)-b^+(W)$. 
	Thus by the Sard-Smale theorem there is a residual subset of $L^2_{k-1,c}(W,\Lambda^2)$ 
	which are regular values of this projection map. 
	
	Next we want to show that $\eta$ can be modified such that $r$ is transverse to $\Phi$. 
	For the simplicity of the exposition we assume that $G$ consists of one point $g$. Essentially the same argument treats the general case.
	The key point is that replacing $\eta(g)$ with $\eta(g)+d^{+_g}a$ for an arbitrary $a\in L^2_{k,\ext}(W,\Lambda^1)$ 
	produces a moduli space that can be identified with $M_\eta(W^g)$. 
	In fact, if $F^{+_g}( A)=i\eta(g)^{+_g}$, then $F^{+_g}({A+ia})=i\eta(g)^{+_g}+d^{+_g}a$. 
	Moreover, regularity of $M_\eta(W^g)$ at $[\underline A]$ implies the regularity of $M_{\eta+d^+a}(W^g)$ 
	at $[\underline {A+ia}]$.
	Consider the map:
	\begin{align}
		\Psi'&:M_\eta(W^g)\times C_0 \oplus C_1 \to \mathcal R(Y_0) \times \mathcal R(Y_1)\nonumber\\
		&([\underline A],h)\to r([\underline{A+i\varphi(h)}])\nonumber
	\end{align}	
	This map is clearly a submersion. Therefore again for a generic choice of 
	$h\in C_0 \oplus C_1$ the restriction map of $M_{\eta+d^+\varphi(h)}(W^g)=M_\eta(W^g)+\varphi(h)$ is transversal to $\Phi$. 
\end{proof}

\begin{remark} \label{ext-per-1}
	The proof of Lemma \ref{super-transversality-1} shows that the following relative version is also valid. 
	Suppose $H \subset G$ is a compact sub-manifold and a perturbation $\eta_0$, defined on an open neighborhood of $H$, is given such that 
	the transversality assumptions of Lemma \ref{super-transversality-1} holds. 
	Then there exists a smooth perturbation $\eta$ defined on $G$ such that $\eta$ agrees with $\eta_0$ in a probably smaller neighborhood of $H$. 
\end{remark}

A perturbation $\eta$ for a family of non-broken metrics $\bW$ is called {\it admissible} if $F^+_G$ is transverse to the zero section of $\mathcal V_{k-1}^+(\bW)$. In the presence of broken metrics, we need to require more in order to have a well-behaved moduli spaces. In this case, if $G'$ is a face of $G$, then it parametrizes a family of non-broken metrics on a cobordism $W\backslash T=W_0 \coprod \dots \coprod W_i$ for a cut $T$. By our assumption $G'$ is equal to $G_0 \times \dots \times G_i$ and for $0\leq k \leq i$, there is a family of (non-broken) metrics $\bW_k$ on $W_k$,  parametrized with $G_k$, such that $\bW|_{G'}=\bW_0 \times \dots \times \bW_i$. By definition, if $\eta$ is a smooth perturbation for $\bW$, then there are perturbations $\eta_k$ for the family of metrics $\bW_k$ that determine $\eta|_{G'}$. We say $\eta$ is admissible over $G'$, if it satisfies the following properties. Firstly we require $\eta_k$ to be an admissible perturbation for $\bW_k$. Consider the restriction map:  
\begin{equation} \label{r_G'}
	r_{G'}:M_{\eta_0}(\bW_0)\times \dots \times M_{\eta_i}(\bW_i) \to \mathcal R(Y_0)\times \mathcal R(Y_1)\times \prod_{j=1}^{i} (\mathcal R(Z_j)\times \mathcal R(Z_j))
\end{equation}	
that is defined by the restriction maps of the moduli spaces $M_{\eta_k}(\bW_k)$. Also, consider: 
\begin{align} \label{phi-delta}
\Phi_\Delta&: \prod_{k=1}^{i} \mathcal R(Z_k)  \to  \prod_{k=1}^{i}  (\mathcal R(Z_k)\times \mathcal R(Z_k))\nonumber\\
&([\underline {B_1}],\dots,[\underline {B_i}]) \to ([\underline {B_1}],[\underline {B_1}],\dots,[\underline {B_i}],[\underline {B_i}])
\end{align}
As another assumption on the admissible perturbation $\eta$, we demand that $r_{G'}$ is transverse to $\Phi_\Delta$. Let $\meGb|_{G'}$ be the subset of $\meGb$ that is mapped to $G'$ by the projection map $\pi$. The space $\meGb|_{G'}$ is equal to the following fiber product:
\begin{equation} \label{brok-mod-spc}
	\xymatrix{
		 \meGb|_{G'}\ar[d]\ar[r]& \mathcal R(Y_0)\times \mathcal R(Y_1) \times \prod_i \mathcal R(Z_i)  \ar[d]_{(id,id,\Phi_\Delta)}\\
		M_{\eta_0}(\bW_0)\times \dots \times M_{\eta_i}(\bW_i) \ar[r]^{r_{G'}\hspace{11mm}}&{\mathcal R(Y_0)\times \mathcal R(Y_1) \times \prod_{j=1}^{i} (\mathcal R(Z_j)\times \mathcal R(Z_j))}\\
	}
\end{equation}
Consequently, if $\eta$ is an admissible perturbation, then $\meGb|_{G'}$ is a smooth manifold of dimension:
$$\dim(G')-\frac{\chi(W)+\sigma(W)}{2}+\frac{b_1(Y_0)+b_1(Y_1)}{2}$$

For a smooth admissible perturbation $\eta$ the moduli space is a cornered manifold of the expected dimension. The main step to show this claim is the following proposition:

\begin{proposition} \label{glu-com}
	Suppose $T$ is a cut in a cobordism $W$ with connected components $T_1$, $\dots$, $T_i$. Suppose also $W\backslash T= W_0 \coprod \dots \coprod W_i$, and 
	$\bW_i$ is a family of (non-broken) metrics on the cobordism $W_k$ parametrized with $G_k$. 
	These families define a family of metrics $\bW'$ on $W\backslash T$ that is parametrized by $G':=G_0 \times \dots \times G_i$.
	Assume that for each element of $G'$ the product metric on the ends corresponding to $T_i$ is induced by the same 
	metric on $T_i$. Therefore, this family can be extended to a family of metrics $\bW$ on the cobordism $W$ parametrized with 
	$[0,\infty]^i\times G'$. Let also $\eta_k$ be an admissible perturbation for $\bW_k$ and $\eta'$ be the induced perturbation for $\bW'$. We can also pull back $\eta'$
	to produce a perturbation for the family $\bW$.
	We require that the restriction map $r_{G'}$ in (\ref{r_G'}) is transverse to the map: 
	$$(id_{\mathcal R(Y_0)\times \mathcal R(Y_1)},\Phi_\Delta):\mathcal R(Y_0)\times \mathcal R(Y_1) \times \prod_{k=1}^{i} \mathcal R(Z_k)  \to  \mathcal R(Y_0)\times \mathcal R(Y_1) 
	\prod_{k=1}^{i}  (\mathcal R(Z_k)\times \mathcal R(Z_k)).$$
	Then there exists $\tau \geq 1$ such that the followings hold:
	\vspace{-10pt}
	\begin{itemize}
		\item[i)] $\eta$ is an admissible perturbation for the non-broken metrics on $W$ parametrized by $(\tau,\infty)^i\times  G'$ and
		 the corresponding moduli space is diffeomorphic to $(\tau,\infty)^i \times M_{\eta}(\bW)|_{G'}$. 
		More generally, $\eta$ defines an admissible perturbation for the family of metrics that lie over 
		$H_1 \times \dots \times H_i\times G'$ where $H_j$ is either $(\tau,\infty)$ or $\{\infty\}$ for each $j$. 
		The corresponding moduli space is diffeomorphic to $H_1  \times \dots \times H_i \times M_{\eta}(\bW)|_{G'}$.
		\item[ii)] The moduli space $M_{\eta}(\bW)_{(\tau,\infty]^i \times G'}$ is homeomorphic to $(\tau,\infty]^i \times M_{\eta}(\bW)|_{G'}$.
	\end{itemize}
	Moreover, if a map $\Phi:X \to \mathcal R(Y_0)\times \mathcal R(Y_1)$ is given and $r_{G'}$ is transverse to $(\Phi, \Phi_\Delta)$,
	then $\tau$ can be chosen large enough such that the map $r:M_{\eta}(\bW) \to \mathcal R(Y_0)\times \mathcal R(Y_1)$ 
	restricted to the part of the moduli space that is parametrized by a set of the form $H_1 \times \dots \times H_i\times G'$ 
	is transverse to $\Phi$. The fiber products: 
	$$M_{\eta}(\bW)|_{(\tau,\infty]^i \times G'} \times_{\Phi} X \hspace{2cm} 
	M_{\eta}(\bW)|_{H_1  \times \dots \times H_i\times G'} \times_{\Phi} X$$
	are respectively homeomorphic and diffeomorphic to: 
	$$(\tau,\infty]^i\times (M_{\eta}(\bW)|_{G'}  \times_{\Phi} X) \hspace{1cm}
	H_1 \times  \dots \times H_i \times (M_{\eta}(\bW)|_{G'}  \times_{\Phi} X)$$
\end{proposition}

The techniques of \cite[Chapter 4]{Don:YM-Floer} can be used to prove this proposition and we leave the details of the proof for the reader. 

\begin{proposition} \label{fam-mod-tra-1}
	For a family of metrics $\bW$ parametrized by a compact cornered manifold $G$, there is an admissible perturbation $\eta$ such that the following holds. 
	For a face $G'$ of $G$ with co-dimension $i$, $M_{\eta}(\bW)|_{G'}$ is a smooth manifold of dimension:
	$$\dim(G)-i-\frac{\chi(W)+\sigma(W)}{2}+\frac{b_1(Y_0)+b_1(Y_1)}{2}.$$
	The total moduli space $\meGb$ is a cornered $C^0$-manifold. 
	Moreover, there is a continuous map of cornered manifolds $\pi:\meGb\to G$ which maps a codimension $i$ face of $\meGb$ to a codimension $i$ 
	face of $G$ by a smooth map. The perturbation $\eta$ can be chosen such that for any face $G'$ of $G$, inducing the decomposition $W_0 \coprod \dots \coprod W_i$,
	the restriction map $r_{G'}$ is transverse 
	to a given map $\Phi_{G'} :X_{G'} \to \mathcal R(\partial W_0) \times \dots \times \mathcal R(\partial W_i)$. 
\end{proposition}

\begin{proof}
	Let $G'$ be a 0-dimensional face of $G$ that defines a broken metric on $W$, and 
	$W_0 \coprod \dots \coprod W_i$ is the induced decomposition of $W$. 
	Fix admissible perturbations $\eta_0$, $\dots$, $\eta_i$ on $W_0$, $\dots$, $W_{i}$, and let $\eta$ be the induced perturbation for the 0-dimensional face $G'$. 
	Then use the argument in the second part of Lemma \ref{super-transversality-1} to insure that $M_\eta(\bW)|_{G'}$ is cut out transversally and $r_{G'}$ is transverse to
	$\Phi_{G'}$. Repeat this construction to extend $\eta$ to all 0-dimensional faces of $G$.
	Use Lemma \ref{super-transversality-1}, Proposition \ref{glu-com}, and Remark \ref{ext-per-1} to extend $\eta$, as an admissible perturbation, to the edges of $G$
	such that the required transversality assumptions hold. The same argument can be used to extend $\eta$ inductively to all faces of $G$. Proposition \ref{glu-com} 
	implies that $\meGb$ is a cornered $C^0$-manifold. This proposition also asserts the existence of the map $\pi$ with 
	the claimed properties.
\end{proof}

The following example is extracted from \cite{KMOS:mon-lens}:
\begin{example} \label{example-family-metrics-1}
	The 4-manifold $U=D^2\times S^2 \# \overline{\cc P}^2$ satisfies $b_1(U)=b^+(U)=0$ and can be considered 
	as a cobordism from $Q=S^1\times S^2$ to the empty set.
	There exists a disc $D$ and a sphere $S_0$ in the connected summand $D^2\times S^2$
	such that the following assumptions hold. 
	The disc $D$ bounds  $S^1\times \{x\} \subset Q$ for an arbitrary point $x\in S^2$,
	the sphere $S_0$ intersects $D$ in exactly one point, and it has self-intersection $0$. 
	Pick also an embedded sphere $S_1$ in the summand $\overline{\cc P}^2$ of $U$ that has self intersection -1 and is 
	disjoint from $D$ and $S_0$. Let $T_1$ be the boundary of a regular neighborhood of $S_1$ that is diffeomorphic 
	to $S^3$. Consider a metric on $W$ which is the standard product metric in a neighborhood of $T_1$. 
	We can construct a family of metrics parametrized with $[0,\infty]$ by stretching this metric along $T_1$. 
	In particular, the metric corresponding to $\infty$ is broken along the cut $T_1$, i.e., it is a metric with cylindrical ends on 
	$U\backslash T_1$ with connected components 
	$U_0\cong D^2\times S^2\backslash D^4$ and $U_1 \cong \overline{\cc P}^2\backslash D^4$.
	
	Next, consider the 2-sphere $S_1'=S_1\# S_0$. This sphere also has self-intersection -1 and its intersection with $D$ is 1.
	By replacing $S_1$ with $S_1'$, we can generate another family of metrics parametrized by $[-\infty,0]$ where $-\infty$ 
	corresponds to a broken metric with a cut along another copy of $S^3$ that is denoted by $T_1'$. 
	The 4-manifold $U\backslash T_1'$ has connected components $U'_0\cong D^2\times S^2\backslash D^4$ 
	and $U'_1\cong \overline{\cc P}^2\backslash D^4$.
	We can glue the two families to obtain a family of metrics $\mathbb U$ parametrized by $G=[-\infty,\infty]$.
	 There is only one torsion $\spinc$ structure on $Q$ which is the spin structure $\st_0$ on $Q$. 
	 Let $\mathfrak s$ be a $\spinc$ structure on $U$ such that $\mathfrak s|_Q=\st_0$. 
	 Any such $\spinc$ structure $\mathfrak s$ on $U$ is uniquely determined with an odd integer which is the pairing
	of $c_1(\mathfrak s)$ and $S_1$.  Suppose $\mathfrak s_{2k+1}$ be the choice of the $\spinc$ structure such that  the above number 
	is $2k+1$. Then $\overline{\s_{2k+1}}$, the complex conjugate of $\mathfrak s_{2k+1}$, has pairing $-2k-1$ with $S_1$ and hence is equal to $\s_{-2k-1}$. 
	The $\spinc$ structures $\s_{2k+1}$ and $\s_{-2k-1}$ have energy $(2k+1)^2$. 
	Because $G$ is contractible, it makes sense to talk about the $\spinc$ connection $\s_{2k+1}$ for the family 
	$\mathbb U|_{(-\infty,\infty)}$.
	
	Since $b^+(U)=0$, any choice of the perturbation gives a regular moduli space (Lemma \ref{cob-moduli-1}). 
	We will write $M(\mathbb U,\s_{2k+1})$ for the moduli space 
	with zero perturbation and the $\spinc$ structure $\mathfrak s_{2k+1}$.
	For each $g\in G$ we have a unique connection $\underline A^g$ on $\s_{2k+1}$ such that $F^{+_g} (A^g)=0$. 
	Thus the moduli space $M(\mathbb U,\s_{2k+1})$ is an interval for each choice of $k\in \zz$. 
	Self-duality of the curvature of $A^g$  implies that $\frac{1}{2\pi i}F(A^g)$ is a $g$-harmonic form representing 
	$c_1(\s_{2k+1})$. As $g$ tends to $\infty$, these forms converge to a harmonic 2-form on $U\backslash T_1$ 
	that vanishes on $U_0$ and represents $(2k+1)PD(S_1)$ on $U_1$. A similar discussion applies to $g=-\infty$.
	
	Note that $\mathcal R(Q)$ is a circle and is parametrized by the holonomy of its elements around $\partial D_0$ 
	which take the following values:
	$$hol(\underline B,\partial D_0)\in \mathcal R=\left \{\left(
	\begin{array}{cc}
		e^{2\pi i\theta}&0\\
		0&e^{2\pi i\theta}
	\end{array}\right)h_0\mid \theta\in [0,1]\right \}$$
	where $\underline B$ is a $\spinc$ connection on $\st_0$, 
	and $h_0\in SU(2)$ is the holonomy of the spin connection around $\partial D$. By Gauss-Bonnet:
	$$\det (hol(\underline A^g,\partial D))=\exp \left ( \int_{D} F(A^g)\right )$$
	
	Therefore $ \int_{D} F(A^\infty)=0$ because $D \subset U_0$ and $A^{\infty}$ is flat on $U_0$. Similarly, $ \int_{D} F(A^{-\infty})=2 \pi i(2k+1)$ 
	because $S_1'$ and $D$ have intersection 1. As a result, $\det (hol(\underline A^g,\partial D))$ starts from 1 and traverses the unit circle an odd number
	of times. Thus the image of $r_k:M(\mathbb U|_{(-\infty,\infty)},\s_k) \to \mathcal R(Q)$ is given by a path between $h_0$ and $-h_0$ in $\mathcal R$
	that consists of $(2k+1)$ half-circles. 
	For the complex conjugate $\spinc$ structure, $ \s_{-2k-1}$, this path is the complex conjugate of the above path. 
	Union of these two paths form a loop in $\mathcal R$ with winding number $2k+1$. In summary, image of $M(\mathbb U,(2k+1)^2)$ in $\mathcal R(Q)$
	is a loop with winding number $(2k+1)$. 
\end{example}

\subsection{Cobordism maps}\label{cob-maps}

Plane Floer homology of a 3-manifold $Y$, denoted by $\PFH(Y)$, is defined to be the homology of the representation variety $\mathcal R(Y)$. In this subsection, we use the elements of the following Novikov ring as the coefficient ring:
\begin{equation*}
	\Lambda_0:=\{\sum_{q_i\in \qq}a_iu^{q_i}|a_i \in \zz/2\zz, a_i\neq 0, q_i \in \qq^{\geq 0}, \lim_{i\to \infty}q_i=\infty\}
\end{equation*}
where $u$ is a formal variable. If we relax the condition that $q_i$ is non-negative, then the resulting field is denoted by $\Lambda$. The space $\mathcal R(Y)$ is a union of $| {\rm Tor} (H_1(Y,\zz))|$ copies of $J(Y)$. Therefore, we have: 
$$\PFH(Y)\cong H_*(J(Y),\Lambda_0)^{\oplus |{\rm Tor}(H_1(Y,\zz))|}\cong \Lambda_0^{2^{b_1(Y)}\cdot  |{\rm Tor}(H_1(Y,\zz))|}$$

Later we will need to work with the lift of $\PFH(Y)$ to the level of chain complexes. To fix one such chain complex, let $f$ be a  Morse-Smale function on the compact manifold $\mathcal R(Y)$ and define $(\PFC(Y),d_p)$ to be the Morse complex for the homology of $\mathcal R(Y)$ associated with $f$. Therefore, the generators of $\PFC(Y)$ is equal to the set of the critical points of $f$ and the differential $d_p$ maps an index $i$ critical point of $f$ to a linear combination of index $i-1$ critical points. The definition of $d_p$ involves counting unparameterized trajectories between critical points of $f$. Note that, however, because $\Lambda_0$ has characteristic 2, we do not need to orient the moduli spaces of trajectories between critical points. This Morse complex is called the {\it plane Floer complex} of $Y$. Equip $\PFC(Y)$ with a grading that its value on a critical point $\alpha$ of $f$ is equal to:
$$\deg_p(\alpha):= 2\cdot\ind(\alpha)-b_1(Y)$$
where $\ind(\alpha)$ stands for the Morse index of $\alpha$. 

For the remaining part of Section \ref{Ab-ASD-con}, we fix a family of metrics $\bW$ on a cobordism $W:Y_0 \to Y_1$ which is parametrized by a compact cornered manifold $G$. As before, $Y_0$ is required to be connected. A typical face of codimension $i$ of $G$ is denoted by $G'$. This face parametrizes a family of broken metrics with a cut along a sub-manifold $T=T_1 \cup T_2 \cup \dots \cup T_i$. The restriction of $\bW$ to this face is given by the family of metrics $\bW_0 \times \dots \times \bW_i$ on $W \backslash T=W_0 \coprod \dots \coprod W_i$. Here $\bW_j$ is a family of (non-broken) metrics on $W_j$, parametrized by a manifold $G_j$. In particular, $G'=G_0 \times \dots \times G_i$. In order to set up the stage, fix Morse-Smale functions $f_0$ on $\mathcal R(Y_0)$, $f_1$ on $\mathcal R(Y_1)$, and $h_j$ on $\mathcal R(T_j)$ for each $T_j$ as above. By the results of the last subsection, there is an admissible perturbation $\eta$ for the family of metrics $\bW$ such that the following hypothesis holds:

\begin{hypothesis} \label{eta-hypo-1}
	Let $\alpha$ and $\beta$ be critical points of $f_0$ and $f_1$, respectively. We will write $U_\alpha$ (respectively, $S_\beta$) for 
	the unstable (respectively, stable) manifold of $\alpha$ (respectively, $\beta$). Let 
	$\Phi(\alpha):U_\alpha \xhookrightarrow{}\mathcal R(Y_0)$ and $\Phi(\beta):S_\beta \xhookrightarrow{}\mathcal R(Y_1)$ be the inclusion maps.  
	For a face $G'$ of $G$ with codimension $i$, suppose that the connected components of $T= \bigcup_{j=1}^i T_j$ is divided to two disjoint sets 
	$\{T_{i_1}, T_{i_2},\dots\}$ and $\{T_{j_1}, T_{j_2},\dots\}$. For each $i_k$, consider critical points $\gamma_{i_k}$, $\gamma'_{i_k}$ of $h_{i_k}$ and let
	 $\Phi_{i_k}: S_{\gamma_{i_k}} \times U_{\gamma'_{i_k}} \xhookrightarrow{} \mathcal R(T_{i_k}) \times \mathcal R(T_{i_k})$ denote the inclusion map.
	 For each $j_l$, define $\Phi_{j_l}:\mathcal R(T_{j_l}) \times [0,\infty) \to \mathcal R(T_{j_l})\times \mathcal R(T_{j_l})$ by 
	 $\Phi_{j_l}(z,t)=(\varphi^t_{j_l}(z),\varphi^{-t}_{j_l}(z))$ where $\varphi^t_{j_l}$ is the negative-gradient-flow of $h_{j_l}$ on 
	 $\mathcal R(T_{j_l})$. Given these maps, we introduce:
	 \begin{equation*}
		\Phi:U_{\alpha} \times S_{\beta} \times  \prod_{i_k} (S_{\gamma_{i_k}} \times U_{\gamma'_{i_k}} ) \times \prod_{j_l}(\mathcal R(T_{j_l}) \times [0,\infty)) \to 
		\hspace{4cm}
	\end{equation*}
	\begin{equation*}
		\hspace{4cm} \mathcal R(Y_0) \times \mathcal R(Y_1) \times  \prod_{i_k}  (\mathcal R(T_{i_k}) \times \mathcal R(T_{i_k})) \times  \prod_{j_l}  
		(\mathcal R(T_{j_l}) \times \mathcal R(T_{j_l}))
	\end{equation*}
	\begin{equation*}
		\Phi:=(\Phi(\alpha),\Phi(\beta),\Phi_{i_1},\dots,\Phi_{j_1},\dots)
	\end{equation*}	
	The smooth admissible perturbation $\eta$ is such that:
	$$r_{G'}: M_{\eta_0}(\bW_0)\times \dots \times M_{\eta_i}(\bW_i) \to \mathcal R(Y_0) \times \mathcal R(Y_1) \times  \prod_{i_k}  \mathcal R(T_{i_k}) \times 
	\mathcal R(T_{i_k}) \times  \prod_{j_l} \mathcal R(T_{j_l}) \times \mathcal R(T_{j_l})$$
	is transverse to all such maps $\Phi$. Also, $\eta$ is chosen such that $\meGbe$ is empty for negative values of $e$.
\end{hypothesis}

There are a few remarks about this hypothesis in order: firstly the domain of $\Phi$ is a cornered manifold. The transversality assumption means that the restriction of $\Phi$ to each face is transverse to $r_{G'}$. Secondly, for each face $G'$ there are different but finitely many choices for $\Phi$. This does not make any issue, as we can consider the union of the domains of all these different choices, and define a total map $\Phi_{G'}$ in order to use Proposition \ref{fam-mod-tra-1}.

By Hypothesis \ref{eta-hypo-1} and Proposition \ref{fam-mod-tra-1}, the following set: 
\begin{equation} 
	\meGab=\{z\in \meGb \mid (r_0(z),r_1(z))\in U_\alpha\times S_\beta\}
\end{equation}
is a cornered manifold and the restriction of $\pi$ maps the faces of $\meGab$ to the faces of $G$. We also have:
\begin{align} \label{dim}
	\dim(\meGab)&=\dim(\meGb)+\dim(U_\alpha\times S_\beta)-\dim(\mathcal R(Y_0)\times \mathcal R(Y_1))\nonumber \\
			     &=\frac{1}{2}(\deg_p(\alpha)-\deg_p(\beta)-\sigma(W)-\chi(W))+\dim(G)
\end{align}

Let $\meGabe$ be the space $\meGab \cap \meGbe$. This space is also a cornered manifold. Unlike $\meGbe$, the moduli space $\meGabe$ is not necessarily a compact manifold. To see what might go wrong, consider a sequence $\{z_i\}_{i\in \mathbb N}\in \meGabe$. Since $\meGbe$ is compact, this sequence is convergent to $z_0 \in \meGbe$ after passing to a subsequence. The sequence  $r_0(z_i)$ converges to $r_0(z_0)$. However, the unstable manifold of $\alpha$ is not necessarily compact and $r_0(z_0)$ might not lie on this set. Nevertheless, $z_0$ lives on the unstable manifold of a critical point $\alpha'$ that $\deg_p(\alpha')\leq \deg_p(\alpha)$ and equality holds if and only if $\alpha=\alpha'$. (In fact, there exists a broken trajectory from $\alpha$ to $\alpha'$ \cite{W:m-trajectory}.)  Similarly, $z_0$ is mapped by $r_1$ to the stable manifold of a critical point $\beta'$ of $f_1$ such that $\deg_p(\beta')\geq \deg_p(\beta)$ and equality holds if and only if $\beta=\beta'$. In summary, $z_0\in M_\eta(\bW,\alpha',\beta';e)$ and the dimension of this moduli space is not greater than that of $\meGabe$. Furthermore, equality holds if and only if $\alpha'=\alpha$ and $\beta'=\beta$. 

Let $\alpha$ and $\beta$ be chosen such that $\meGab$ is 0-dimensional, i.e.:
\begin{equation} \label{deg-m-f-g-1}
	\deg_p(\beta)-\deg_p(\alpha)=2\dim(G)-\sigma(W)-\chi(W)
\end{equation}	
The discussion of the previous paragraph implies that in this special case, $\meGabe$ is compact. Furthermore, the transversality assumptions imply that the elements of $\meGab$ live over the interior of $G$. The 0-dimensional space $\meGab$ is subset of a {\it secondary moduli space} $\tmeGab$. The space $\tmeGab$ is defined over an {\it open fattening} $\bf G$ of $G$, i.e., there is a map $\tilde \pi:\tmeGab \to \bf G$, and ${\tilde \pi}^{-1}(G)=\meGab$.

In order to construct ${\bf G}$, firstly note that $G$ is the union of all products $(0,\infty]^i \times G'$ when $i$ is an integer number and $G'$ is a codimension $i$ face of $G$. The set:
\begin{equation} \label{subset-G''}
	\{\infty\}\times \dots \{\infty\}\times \underbracket{(0,\infty)}_{k^{\rm{th}}\text{ factor}} \times \{\infty\} \times \dots \times \{\infty\} \times G'
\end{equation}	
can be regarded as an open subset of a codimension $i-1$ face $G''$ in $G$. We can arrange the standard neighborhoods $(0,\infty]^i \times G'$ and $(0,\infty]^{i-1}\times G''$ of these two faces such that:
\begin{equation} \label{identify-G'-G''}
	(t_1,\dots,t_i,g) \in (0,\infty]^i \times G'\hspace{2cm} (t_1,\dots,t_{k-1},t_{k+1},\dots,t_i,(t_k,g)) \in (0,\infty]^{i-1} \times G''
\end{equation}	
represent the same points in $G$, where $t_k\in(0,\infty)$ and $(t_k,g)$ is considered as a point in (\ref{subset-G''}).

\begin{figure}
	\centering
		\includegraphics[width=.3\textwidth]{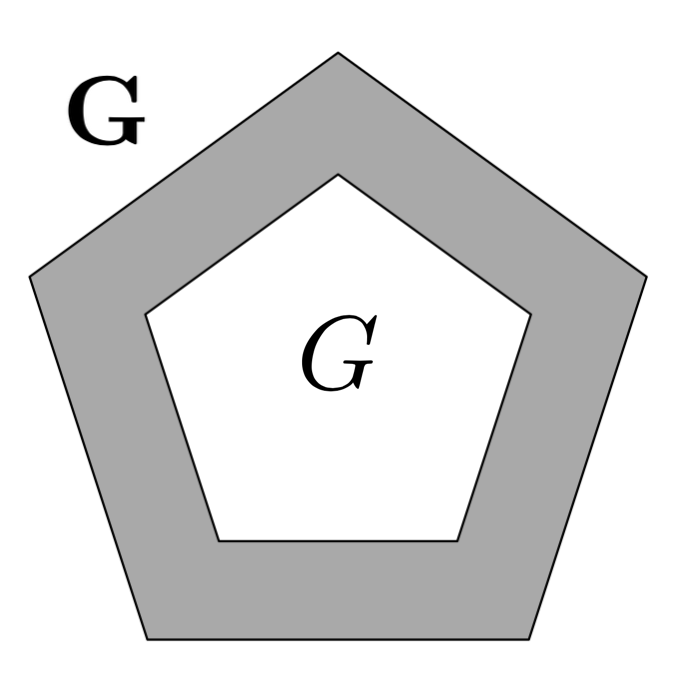}
		\caption{An open fattening of the white pentagon $G$ is given by adding the shaded region.}
		\label{2-handle}
\end{figure}

Let  $\overline{(0,\infty)}$ be the result of gluing the intervals $(0,\infty]$ and $[0.\infty)$ by identifying the $\infty$ end of the first interval (the {\it primary} interval) with the $0$ end of the second interval (the {\it secondary} interval). The cornered manifold $\bf G$ is the union of all the products of the form $\overline{(0,\infty)}^i \times G'$ with the following equivalence relation. For faces $G'$ and $G''$ we identify 
$$\overline{(0,\infty)}^{k-1}\times (0,\infty) \times \overline{(0,\infty)}^{i-k}\times G'$$
in which the $k^{\rm th}$ factor is a primary interval, with an open subset of $\overline{(0,\infty)}^{i-1}\times G''$ using (\ref{identify-G'-G''}). The manifold $\bf G$ is non-compact and can be compactified to a cornered manifold that is homeomorphic to $G$. This compactification is given by a similar construction except that one replaces the secondary interval $[0,\infty)\subset \overline{(0,\infty)}$ with the closed interval $[0,\infty]$. If the product $G_1 \times G_2$ is a face of $G$, then the corresponding face of the compactification of $\bf G$ is the same as $ \bf {G_1}\times {\bf G_2}$. A contraction map $\rho: {\bf G} \to G$ can be defined in the following way. The subset $\overline{(0,\infty)}^i \times G' \subset \bf G$ can be decomposed to $2^n$ sets of the following form:
\begin{equation} \label{fat-1}
	H=A_{l_1}\times \dots \times A_{l_i} \times G' 
\end{equation}	
where each $l_k$ is either 1 or 2 and determines whether the $k^{\text th}$ interval is primary or secondary, i.e., $A_1=(0,\infty]$ and $A_2=[0,\infty)$. Define:
$$\rho:H \to (0,\infty]^i\times G'$$
$$(g,t_1,t_2,\dots,t_i) \to (g,\epsilon_{l_1}(t_1),\epsilon_{l_2}(t_2),\dots,\epsilon_{l_i}(t_i))$$
where $\epsilon_1:A_1\to (0,\infty]$ is the identity map and $\epsilon_2:A_2\to (0,\infty]$ is the constant $\infty$ map. These maps together, for different involved choices, define the contraction map $\rho:{\bf G} \to G$. 

The moduli space $\tmeGab$ will be produced by constructing moduli spaces over the subsets $H\subset \bf G$ of the form of (\ref{fat-1}), and then putting them together. We denote the moduli space over $H$ by $\widetilde M_\eta(\bW,\alpha,\beta)|_H$. This space can be defined without any assumption on the dimension of $\meGab$. 
Without loss of generality, assume that in the definition of $H$ we have $l_1=\dots=l_k=1$, $l_{k+1}=\dots=l_{i}=2$. The subset $I=(0,\infty]^{k}\times \{\infty\}^{i-k}\times G'$ of $G$ parametrizes a family of metrics $\bW'$ on $W':=W\backslash \bigcup_{k+1\leq j\leq i} T_j$ with $i-k+1$ connected components. The perturbation $\eta$ produces a perturbation $\eta'$ for the family of metrics $\bW'$ and hence we can form the moduli space $M_{\eta'}(\bW',\alpha,\beta)$. Each 3-manifold $T_j$, $k+1 \leq j \leq i$, appears both as the incoming end of one of the components of $W'$, and also as an outgoing end of another component of $W'$. Therefore, we have the restriction maps $r^{in}_j:M_{\eta'}(\bW',\alpha,\beta) \to \mathcal R(T_j)$ and $r^{out}_j:M_{\eta'}(\bW',\alpha,\beta) \to \mathcal R(T_j)$. The space $\widetilde M_\eta(\bW,\alpha,\beta)|_H$ is defined to be the fiber-product of the following diagram:
\begin{equation} \label{fib-prod-1}
	\xymatrix{
		\widetilde M_\eta(\bW,\alpha,\beta)|_H \ar[d]\ar[rrr]&&&\prod_{k+1\leq j\leq i} \mathcal R(T_j)\times [0,\infty)\ar[d]^{\Phi}\\
		M_{\eta'}(\bW',\alpha,\beta)\ar[rrr]^{\prod_{k+1\leq j\leq i}r^{in}_j\times r^{out}_j\hspace{20pt}}&&&
		\prod_{k+1\leq j\leq i} \mathcal R(T_j)\times \mathcal R(T_j)\\
	}
\end{equation}
where  $\Phi$ is defined as:
$$\Phi(([\underline B_{k+1}],t_{k+1}),\dots,([\underline B_{i}],t_{i})) \to 
(\varphi_{k+1}^{t_{k+1}}([\underline B_{k+1}]),\varphi_{k+1}^{-t_{k+1}}([\underline B_{k+1}]),\dots,\varphi_{i}^{t_{i}}([\underline B_{i}]),\varphi_{i}^{-t_{i}}([\underline B_{i}])))$$
The map $\varphi_j^t$ is the negative-gradient-flow of the Morse-Smale functions $h_j$. Hypothesis \ref{eta-hypo-1} implies that $\widetilde M_\eta(\bW,\alpha,\beta)|_H$ is a smooth cornered manifold. This set consists of the following points:
\begin{equation} \label{sub-tmGabe}
	\widetilde M_\eta(\bW,\alpha,\beta)|_H=\{([\underline A],g,t_{k+1},\dots,t_i) \mid ([\underline A],g) \in M_{\eta'}(\bW',\alpha,\beta),
	\hspace{1pt}  \varphi_j^{t_j}(r^{out}_{j}([\underline A]))=\varphi_j^{-t_j}(r^{in}_{j}([\underline A]))\}
\end{equation}
We can also define $\tilde \pi:\widetilde M_\eta(\bW,\alpha,\beta)|_H\to H$ by:
$$\tilde \pi([\underline A],g,t_{k+1},\dots,t_i)=(g,t_{k+1},\dots,t_i)$$
The dimension of $\widetilde M_\eta(\bW,\alpha,\beta)|_H$ is equal to that of $\meGab$. Thus in the case that $\meGab$ is zero dimensional, $\widetilde M_\eta(\bW,\alpha,\beta)|_H$ is 0-dimensional and ${\tilde \pi}^{-1}(\partial H)$ is empty. Therefore, we can put together the moduli spaces $\widetilde M_\eta(\bW,\alpha,\beta)|_H$, for different choices of $H$, to form a zero dimensional moduli space $\tmeGab$ and a map $\tilde \pi:\tmeGab \to \bf G$. The energy functional can be extended to $\tmeGab$ and the spaces $\tmeGabe$ is defined to be the elements of $\tmeGab$ with energy equal to $e$.

\begin{lemma} \label{0-dim-comp-1}
	The moduli space $\tmeGabe$ is compact.
\end{lemma}

\begin{proof}
	Suppose $H=A_{l_1}\times \dots \times A_{l_i} \times G' $ is as in (\ref{fat-1}), and $H'\subseteq H$ is the set $A_{l_1}'\times \dots \times A_{l_i}' \times G'$ such that 
	$A_{l_j}'=[1,\infty]$ if $A_{l_j}$ is a primary interval, and $A_{l_j}'=A_{l_j}$ if $A_{l_j}$ is secondary. 
	It suffices to show that for each choice of $H$, $\widetilde M_\eta(\bW,\alpha,\beta;e)|_{H'}:=\pi^{-1}(H')\cap \tmeGabe$ is compact. 
	For the simplicity of exposition, we assume that $H$ has the above special form, i.e., the 
	first $k$ intervals are primary and the remaining ones are secondary.
	Let $\{([\underline A^j],g^j,t^j_{k+1},\dots,t^j_i)\}_{j \in \mathbb N}$ be a sequence in $\widetilde M_\eta(\bW,\alpha,\beta;e)|_{H'}$.
	This sequence after passing to a subsequence is convergent to $([\underline A^0],g^0,t^0_{k+1},\dots,t^0_i)$ with 
	 $t^0_l \in [1,\infty]$, and $([\underline A^0],g^0) \in M_{\eta'}(\bW',\alpha',\beta';e)$ 
	 for critical points $\alpha'$ of $f_0$
	 and $\beta'$ of $f_1$. Furthermore,  $\deg_p(\alpha')\leq \deg_p(\alpha)$ and $\deg_p(\beta')\geq \deg_p(\beta)$ and the equalities
	 hold if and only if $\alpha'=\alpha$ and $\beta'=\beta$. Suppose $t^0_l=\infty$ for some $l$. 
	 There is a trajectory $\gamma_l: [0,t^j_l] \to \mathcal R(T_l)$ of the flow
	 $\varphi_l^t$ that $\gamma_l(0)=r^{out}_l([\underline A^j])$ and $\gamma_l(2t^j_l)=r^{in}_l([\underline A^j])$. 
	 Since $t_l^j$ tends to $t^0_l=\infty$, according to \cite{W:m-trajectory} there is a broken trajectory that starts from 
	 $r^{out}_l([\underline A^0])$ and terminates at $r^{in}_l([\underline A^0])$. 
	 In particular, $(r^{in}_l([\underline A^0]),r^{out}_l([\underline A^0]))\in U_{\alpha_l}\times S_{\beta_l}$  wehre $\alpha_l$ and $\beta_l$
	 are two critical points of $h_l$, and $\dim(U_{\alpha_l})+\dim(S_{\beta_l})\leq \dim(\mathcal R(T_l))$. 
	 The point $([\underline A^0],g^0,t^0_{k+1},\dots,t^0_i)$ lives in the following fiber product:
	 \begin{equation} \label{fib-prod}
	\xymatrix{
		\mathfrak M\ar[d]\ar[rrr]&&&\prod_{k+1\leq l\leq i} \mathfrak R_l\ar[d]^{\Phi'}\\
		M_{\eta'}(\bW',\alpha',\beta')\ar[rrr]^{\prod_{k+1\leq l\leq i}r^{in}_l\times r^{out}_l\hspace{20pt}}&&&
		\prod_{k+1\leq l\leq i} \mathcal R(T_l)\times \mathcal R(T_l)\\
	}
	\end{equation}
	If $t_l\in [0,\infty)$, $\mathfrak R_l=\mathcal R(T_l)\times [0,\infty)$ , otherwise $\mathfrak R_l=U_{\alpha_l}\times S_{\beta_l}$ is 
	chosen as above. The smooth map $\Phi$ is defined as in Hypothesis \ref{eta-hypo-1}. Hypothesis \ref{eta-hypo-1} implies that, 
	$\prod_{k+1\leq l\leq i}r^{in}_l\times r^{out}_l$ 
	is transverse to $\Phi$. The inequalities:
	\begin{equation} \label{dim-ineq}
		\dim( \mathfrak R_l)\leq\dim(\mathcal R(T_l)\times [0,\infty))\hspace{1cm}\dim(M_{\eta'}(\bW',\alpha',\beta'))
		\leq\dim(M_{\eta'}(\bW',\alpha,\beta))
	\end{equation}	
	imply that for a choice of $\eta$ as above:
	$$\dim(\mathfrak M)\leq\dim(\widetilde M_\eta(\bW,\alpha,\beta)|_{H'})$$
	But the dimension of $\widetilde M_\eta(\bW,\alpha,\beta)|_{H'}$ is zero. Therefore, $\mathfrak M$ is non-empty only if the
	inequalities in (\ref{dim-ineq}) become equalities. That is to say, $\alpha'=\alpha$, $\beta'=\beta$, and all $t^0_l$ are finite
	and hence $([\underline A^0],g^0,t^0_{k+1},\dots,t^0_i)\in \widetilde M_\eta(\bW,\alpha,\beta;e)|_{H'}$. 
	Therefore, $\widetilde M_\eta(\bW,\alpha,\beta;e)|_{H'}$ is compact.
\end{proof}

Now we are ready to define the map $\feW:\PFC(Y_0) \to \PFC(Y_1)$. Let $\alpha$ be a critical point of $f_0$ and $\langle \alpha \rangle$ be the corresponding generator of $\PFC(Y_0)$:
\begin{equation} \label{feG}
	\feW(\langle \alpha \rangle)=\sum_{\beta\in Crit(f_1)} \left( \sum_{p\in\tmeGab}u^{e(p)} \right) \langle \beta \rangle
\end{equation}

Here $\beta$ ranges over all critical points of $f_1$ that $\meGab$ is 0-dimensional. Lemma \ref{0-dim-comp-1} insures that the above sum is a well-defined element of $\Lambda_0$. Note that this is the first place that we need to work with the Novikov ring $\Lambda_0$ rather than, say, $\zz/2\zz$. We call the coefficient of $\langle \beta \rangle$ in the above expression, the {\it matrix entry} for the pair $(\alpha,\beta)$ and denote it by $\feW(\alpha,\beta)$. As a result of (\ref{deg-m-f-g-1}):
\begin{equation} \label{morphism-morse-deg}
	\deg_p(\feW)=2\dim(G)-\sigma(W)-\chi(W)
\end{equation}	

In order to study the homotopic properties of $\feW$, we construct secondary moduli spaces when $\meGab$ is 1-dimensional. In this case, the space $\widetilde M_\eta(\bW,\alpha,\beta)|_H$ is a 1-manifold whose boundary lies over the codimension 1 faces of $H$ via the projection map $\tilde \pi$. If $H$ and $H'$ are two sets of the form (\ref{fat-1}) that share a face, then the moduli spaces over them match with each other. In the case that the common face has codimension at least 2, this is clear because the moduli spaces over the common face are empty and hence equal. To prove the case of codimension 1, let $H$ be as above and $H'$ be the following subset of $\bf G$:
$$A_{l'_1}\times \dots \times A_{l'_i}\times G_b$$
where $l'_1=l_1=1$, $\dots$, $l'_{k-1}=l_{k-1}=1$, $l'_k=2$, $l'_{k+1}=l_{k+1}=2$, $\dots$, $l'_{i}=l_{i}=2$. Then the moduli spaces over the common face are equal to:
$$(\widetilde M_\eta(\bW,\alpha,\beta)|_H)|_{H\cap H'}=(\widetilde M_\eta(\bW,\alpha,\beta)|_{H'})|_{H\cap H'}=\{([\underline A],g,t_{k+1},\dots,t_i) \mid ([\underline A],g) \in M_{\eta'}(\bW',\alpha,\beta),
\hspace{25pt}$$
\vspace{-5pt}
$$\hspace{35pt} g\in \underbrace{[0,\infty] \times \dots \times [0,\infty]}_{k-1} \times \{\infty\} \times G_b, \hspace{1pt}
	 r^{out}_{k}([\underline A])=r^{in}_{k}([\underline A]), \hspace{1pt}\varphi_j^{t_j}(r^{out}_{j}([\underline A]))=\varphi_j^{-t_j}(r^{in}_{j}([\underline A]))\} $$
By gluing together these moduli spaces, form the 1-dimensional secondary moduli space $\tmeGab$. As in the 0-dimensional case, the energy functional can be extended to $\tmeGab$, and the set of elements of $\tmeGab$ with energy $e$ is denoted again by $\tmeGabe$. The 1-dimensional manifold $\tmeGab$ is a union of open intervals and simple closed curves. An open interval can be compactified by adding two endpoints, and a punctured neighborhood of these endpoints are called an {\it open end} of the interval.

Let $G^1,\dots,G^m$ be the codimension 1 faces of $G$ such that each $G^i$ parametrizes a family of metrics with a cut along a connected 3-manifold $T_i$. There are two possibilities for each $T_i$. Removing $T_i$ decomposes $W$ to either cobordisms $W^i_1:Y_0 \to T_i$ and $W^i_2:T_i \to Y_1$ or cobordisms $W^i_1:Y_0 \to T_i\coprod Y_1$ and $W^i_2:T_i \to \emptyset$. In each of these cases, there are families of metrics $\bW^i_1$, $\bW^i_2$ over the cobordisms  $W^i_1$, $W^i_2$, parametrized by $G_1^i$, $G_2^i$, such that $G^i=G^i_1\times G^i_2$ and the restriction of $\bW$ to the face $G^i$ is given by $\bW^i_1 \times \bW^i_2$. 

\begin{lemma}\label{ends-bar-M-1}
	Let $\meGab$ be a 1-dimensional moduli space. Then the open ends of the moduli space $\tmeGab$ can be identified with the following set:
	$$\bigcup_{i,\gamma\in Crit(h_i)} \widetilde {M}_\eta(\bW^{i}_1,\alpha,\gamma) \times 
	\widetilde {M}_\eta(\bW^{i}_2,\gamma,\beta) \cup $$
	\begin{equation} \label{bdry-1}
		\bigcup_{\substack{\gamma\in Crit(f_0)\\ \ind(\alpha)-\ind(\gamma)=1}} 
		\mathcal M(\alpha,\gamma)\times \widetilde {M}_\eta(\bW,\gamma,\beta) \hspace{5mm}\cup 
		\bigcup_{\substack{\gamma\in Crit(f_1)\\ ind(\gamma)-ind(\beta)=1}} 
		\widetilde {M}_{\eta}(\bW,\alpha,\gamma) \times \mathcal M(\gamma,\beta)
	\end{equation}	
	where $\mathcal M(\alpha,\gamma)$ $($respectively, $\mathcal M(\gamma,\beta)$$)$ 
	is the set of unparameterized Morse trajectories from $\alpha$ to $\gamma$ $($respectively, from $\gamma$ to $\beta$$)$.
	Furthermore, adding the endpoints of the moduli space $\tmeGabe$ makes this moduli space compact.
\end{lemma}
\begin{proof}
	The proof of this lemma is standard and consists of two parts. Firstly we need to show that for each point in (\ref{bdry-1}) there is an open end of $\tmeGab$. 
	In the second part, we need to show that after adding the endpoints in (\ref{bdry-1}) the moduli space $\tmeGabe$ is compact. The latter can be proved in the similar 
	way as in the proof of Lemma \ref{0-dim-comp-1}. 	
	In order to prove the first part, fix a codimension 1 face $G^i$ of $G$ and a critical point $\gamma$ of the function $h_i$. Define:
	\begin{equation} \label{traj}
		\mathcal M(\mathcal R(T_i),\mathcal R(T_i)):=\{(z_1,z_2) \in \mathcal R(T_i)\times \mathcal R(T_i) \mid 
		\exists t \in \rr^{\geq 0}\hspace{5pt} 
		\phi_i^t(z_1)=\phi_i^{-t}(z_2)\}.
	\end{equation}
	According to \cite{W:m-trajectory}, $\mathcal M(\mathcal R(T_i),\mathcal R(T_i))$ can be compactified to a cornered manifold and the set 
	$U_\gamma \times S_\gamma$ is a codimension 1 face of this compactification. Roughly speaking, this face can be thought as the case that $t=\infty$ and 
	$\phi_i^t(z_1)=\phi_i^{-t}(z_2)=\gamma$. Let $\mathcal M_\gamma$ be the union of 
	$\mathcal M(\mathcal R(T_i),\mathcal R(T_i))$ and $U_\gamma \times S_\gamma$. This space as a manifold with boundary can be embedded in 
	$\mathcal R(T_i)\times \mathcal R(T_i)$ in a natural way. 
	By Hypothesis \ref{eta-hypo-1}, the restriction map $r:M_\eta(\bW^i_1)\times M_\eta(\bW^i_2)\to
	\mathcal R(Y_0)\times \mathcal R(T_i)\times \mathcal R(T_i)\times \mathcal R(Y_1)$ is transverse to the sub-manifolds
	$\mathcal R(Y_0)\times \mathcal M_\gamma \times \mathcal \mathcal R(Y_1)$. Note that 
	$r^{-1}(\mathcal M(\mathcal R(T_i),\mathcal R(T_i)))$ is an open subset of $\tmeGab$. Therefore, $ r^{-1}(\mathcal M_\gamma)$ is a 1-dimensional manifold 
	that its interior is a subset of $\tmeGab$ and its boundary is equal to $\widetilde {M}_\eta(\bW^i_1,\alpha,\gamma) \times \widetilde {M}_\eta(\bW^i_2,\gamma,\beta)$.
	 This verifies that the first term in (\ref{traj}) parametrizes part of the open ends of $\tmeGab$. 
	 A similar argument shows that the other two terms parametrize other open ends and these three sets of the open ends are mutually disjoint.
\end{proof}
 
\begin{lemma}\label{boundary-1}
	For the family of metrics $G$ as above, we have:
	\begin{equation}
		\sum_i   f^{\eta}_{\bW^{i}_1 \times \bW^{i}_2}=d\circ \feW+\feW \circ d
	\end{equation}
	If $W\backslash T_i$ is the composition of the cobordisms $W^i_1: Y_0 \to T_i$ and 
	$W^i_2: T_i \to Y_1$, then $f^{\eta}_{\bW^{i}_1 \times \bW^{i}_2}=  f^{\eta}_{\bW^{i}_1} \circ f^{\eta}_{\bW^{i}_2}$. 
	If $W\backslash T_i$ is the union of the cobordisms $W^{i}_1: Y_0 \to T_i\coprod Y_1$ and $W^{i}_2: T_i \to \emptyset$, then:
	$$f^{\eta}_{\bW^{i}_1}:\PFC(Y_0) \to \PFC(T_i \coprod Y_1)=\PFC(T_i)\otimes \PFC(Y_1)\hspace{1cm} f^{\eta}_{\bW^{i}_2}:\PFC(T_i) \to \Lambda_0$$
	In this case, $f^{\eta}_{\bW^{i}_1 \times \bW^{i}_2}$ is equal to $(f^{\eta}_{\bW^{i}_2} \otimes id) \circ f^{\eta}_{\bW^{i}_1}$
\end{lemma}
\begin{proof}
	By Lemma \ref{ends-bar-M-1} elements of the set (\ref{bdry-1}) form the boundary of a 1-dimensional manifold. 
	In particular, they can be divided into pairs where elements of each pair consists of open ends of an inerval. 
	Elements of each pair have the same energy, because they are connected by a path in $M_\eta(\bW)$. 
	Therefore, the weighted sum of the points is zero.
	Weighted number of the points in the first, second, and third sets of (\ref{bdry-1}), respectively, 
	correspond to the coefficients of $\langle\beta\rangle$ in
	$\sum_{i=1}^k f^{\eta}_{\bW^{i}_1\times \bW^{i}_2}(\langle\alpha\rangle)$, 
	$\feW\circ d(\langle\alpha\rangle)$, and $d \circ \feW(\langle\alpha\rangle)$. 
	Thus vanishing of the summation of these numbers verifies the above relation.
\end{proof}

In the case that $G$ has only one metric $g$, Lemma \ref{boundary-1} asserts that $\feW$ is a chain map.  We will write $\PFC(W)$ for $\feW$ in this special case, and the resulting map at the level of homology will be denoted by $\PFH(W)$. Standard arguments show that:
\begin{proposition} \label{functoriality}
	The map $\PFH(W):\PFH(Y_0) \to \PFH(Y_1)$ is independent of the perturbation and the metric on $W$. 
	If $W=W_0\circ W_1$ then $\PFH(W)=\PFH(W_1)\circ \PFH(W_0)$. 	
\end{proposition}

\subsection{Oriented Moduli Spaces and a Local Coefficient System} \label{ori-loc}

In the previous subsections, plane Floer homology is defined as a functor from the category of 3-manifolds and cobordisms between them to the category of $\zz$-graded $\Lambda_0$-modules. In this section, this functor will be lifted to the category of of $\zz$-graded modules over the following ring: 
\begin{equation} \label{tLz}
	\tLz:=\{\sum_{i=1}^{\infty} a_iu^{q_i} \mid a_i \in \zz, q_i \in \qq^{\geq 0},  \lim_{i \to \infty} q_i=\infty\}
\end{equation}
In the above definition, if we allow the exponents $q_i$ to be negative, then the resulting Novikov ring is denoted by $\tL$. To define such lift of plane Floer homology, it is necessary to change slightly the definition of morphisms in our topological category. In the new topological category, a morphism between two 3-manifolds is a cobordism with a choice of a homology orientation. Composition of two morphisms is also given by the composition of the underlying cobordisms equipped with the composition of homology orientations. As before, an important feature of the cobordism maps is the possibility of extending their definition to the more general case of family of metrics.

Firstly we set our conventions for orienting the elements of the real $K$-group of a topological space. For a topological space $G$, a line bundle can be associated with each element of $KO(G)$ that is called the {\it orientation bundle}. This line bundle for $O=[E_1]-[E_2]\in KO(G)$ is defined as $l_{O}=\Lambda^{\max}(E_1)\otimes (\Lambda^{\max}E_2)^*$. Different representations of this element of $KO(G)$ gives rise to naturally isomorphic line bundles (more precisely, $\zz/2\zz$-bundles). As an example, the determinant bundle of a family of metrics is the orientation bundle of the index bundle. If $O=[E_1]-[E_2]$ and $P=[F_1]-[F_2]$ are two elements of $KO(G)$ then $l_O\otimes l_P$ and $l_{O+P}$ are isomorphic and we fix this isomorphism to be:
$$(e_1\wedge e_2^*)\otimes (f_1 \wedge f_2^*) \in l_O\otimes l_P \to (-1)^{\dim(E_2)\cdot \dim(F_1)}(e_1\wedge f_1) \wedge (e^*_2 \wedge f^*_2) \in l_{O+P}$$
where $e_1\in \Lambda^{\max} E_1$, and so on. In what follows, this convention is used to trivialize $l_{O+P}$ when the line bundles $l_{O}$ and $l_{P}$ are trivialized. Note that, the isomorphism depends on the order of $O$ and $P$, and reversing this order changes the isomorphism by $(-1)^{{\rm rk}(O)\cdot {\rm rk}(P)}$.

For a 3-manifold $Y$, fix a Morse-Smale function $f$ on $\mathcal R(Y)$ and define $(\tPFC(Y),\td_p)$ to be the Morse complex of $f$ with coefficients in $\tLz$(cf., for example, \cite{KM:monopoles-3-man}). Because $\tLz$ is a characteristic zero ring, to define this chain complex, one needs to work with oriented moduli spaces of gradient flow lines between the critical points. For the convenience of the reader, we briefly recall the definition of the Morse  complex in this case. For a critical point $\alpha$ of the Morse function $f$, let $\cU^+_\alpha$ and $\cU^-_\alpha$ be the tangent bundles of the stable and the unstable manifolds of $\alpha$. Because these manifolds are contractible, the tangent bundles are trivial. Define $o(\alpha)$ to be the abelian group generated by the two different orientations of $\cU^+_\alpha$ with the relation that these two generators are reverse of each other. Clearly, this group has rank 1. However, it does not have a distinguished generator. With this convention, $\tPFC=\bigoplus_{\alpha \in Crit(f)} o(\alpha)\otimes_\zz \tLz$. In order to define the differential, we orient the moduli spaces of flow lines in the following way. For two distinct critical points $\alpha$ and $\beta$, let $\check{\mathcal M}(\alpha,\beta)$ be the set of parametrized flow lines from $\alpha$ to $\beta$. By the aid of inclusion maps of $\check{\mathcal M}(\alpha,\beta)$ in the stable and the unstable manifolds, $[\cU^+_{\beta}]-[\cU^+_{\alpha}]$ can be regarded as an element of $KO(\check{\mathcal M}(\alpha,\beta))$, and is canonically isomorphic to $[T\check{\mathcal M}(\alpha,\beta)]$. The translation action of $\rr$ on $\check{\mathcal M}(\alpha,\beta)$ lets us to view this manifolds as a fiber bundle over $\mathcal M(\alpha,\beta)$ with the fiber $\rr$. Thus, as an identity in $KO(\check{\mathcal M}(\alpha,\beta))$: 
$$[\underline \rr]+[q^*(T\mathcal M(\alpha,\beta))]=[\cU^+_{\beta}]-[\cU^+_{\alpha}]$$ 
where $q: \check{\mathcal M}(\alpha,\beta) \to \mathcal M(\alpha,\beta)$ is the quotient map. Therefore, given elements of $o(\alpha)$ and $o(\beta)$, we can orient $T\mathcal M(\alpha,\beta)$. In the case that $\mathcal M(\alpha,\beta)$ is 0-dimensional, the orientation assigns a sign to each gradient flow line and the signed count of the points of $\mathcal M(\alpha,\beta)$ is used to define the differential of the Morse complex. The homology of the complex $(\tPFC(Y),\td_P)$, denoted by $\tPFH(Y)$, is the oriented plane Floer homology of $Y$.

\begin{remark}
	The standard convention to orient the moduli space $\mathcal M(\alpha,\beta)$ uses orientations of $\cU_{\alpha}^-$ 
	and $\cU_{\beta}^-$ rather than $\cU_{\alpha}^+$ and $\cU_{\beta}^+$. 
	In fact, the standard orientation gives rise to a Morse homology group that is canonically 
	isomorphic to singular homology groups. However, note that if $M$ is a compact manifold and $f$ is a 
	Morse-Smale function, then the stable manifolds (respectively, unstable manifolds) of $f$ are the same as unstable manifolds 
	(respectively, stable manifolds) of $-f$. Therefore, the Morse complex for the homology of $M$, defined by $f$ 
	and the orientations of the stable manifolds of $f$, is the same as the Morse complex for the cohomology of $M$, 
	defined by $-f$ and the orientations of the unstable manifolds of $-f$. This is a manifestation of Poincar\'e
	duality at the level of Morse complexes. Thus $\tPFH(Y)$ is in fact naturally isomorphic to the cohomology of $\mathcal R(Y)$, 
	and in order to identify it with the homology of $\mathcal R(Y)$, 
	one needs to fix an orientation of the representation variety. 
\end{remark}

Let $\bW$ be a family of metrics on a cobordism $W:Y_0 \to Y_1$ parametrized by a cornered manifold $G$. Let also $G^1$, $\dots$, $G^m$ be the set of open faces of codimension 1 in $G$ such that $G^i=G^i_1\times G^i_2$ and the family of metrics $\bW|_{G^i}$ has the form $\bW^i_1 \times \bW^i_2$. Here $\bW^i_j$ is a family of metrics parametrized by $G^i_j$. We assume that for each $i$, $j$, an orientation of $\ind(\bW^i_j;{ G^i_j})$ and also an orientation of $\ind(\bW;G)$ are fixed such that the following holds. The orientation of $\ind(\bW; G)$ induces an orientation of $\ind(\bW|_{G^i};G^i)\in KO(G^i)$ which is required to be the same as the orientation given by the sum $\ind(\bW^i_1;{ G^i_1}) + \ind(\bW^i_2;{ G^i_2})$. Furthermore, we assume that $G$ and all $G^i_j$'s are oriented. Thus $G^i$ can be either oriented as a codimension 1 face of $G$ by the outward-normal-first convention or as the product $G^i_1\times G^i_2$ by the first-summand-first convention. We demand these two orientations to be the same. We will write $\ind(W;{\bf G})$ for the pull-back of the index bundle $\ind(W;G)$ to the fattening ${\bf G}$ by the contraction map.

The following lemma explains how to orient $\tmeGab$:

\begin{lemma} \label{ori-conv}
	Suppose $\bW$ is a family of metrics as above and $\eta$ is an admissible perturbation for this family that satisfies Hypothesis \ref{eta-hypo-1}. Let also 
	$\alpha$, $\beta$ be critical points of the Morse functions $f_0$, $f_1$ such that $\meGab$ is either 0- or 1-dimensional. Then the orientation 
	double cover of $\tmeGab$ is canonically isomorphic to the pull-back of the orientation bundle of: 
 	\begin{equation} \label{orien-KO}
		-[\cU^+_{\alpha}]+\ind(\bW;{\bf G})+[\cU^+_\beta]+[T{\bf G}]
	\end{equation}	
	 by the projection map $\pi:\tmeGab \to {\bf G}$.
\end{lemma}
In (\ref{orien-KO}), $[\cU^+_{\alpha}]$  and $[\cU^+_{\beta}]$ are given by the pull-back of arbitrary constant maps from $G$ to the stable manifolds of $\alpha$ and $\beta$. These elements of $KO({\bf G})$ are well-defined because $\cU^+_{\alpha}$ and $\cU^+_{\beta}$ are bundles over contractible spaces. 
\begin{proof}
	Firstly consider the case that $\meGab$ is 0-dimensional. Fix $p\in\tmeGab$, 
	and without loss of generality assume that $p \in \widetilde M_\eta(\bW,\alpha,\beta)|_H$ for the same choice of $H$ 
	that produces (\ref{sub-tmGabe}). Recall that this set is described as:
	\begin{equation*} 
		\widetilde M_\eta(\bW,\alpha,\beta)|_H=\{([\underline A],g,t_{k+1},\dots,t_i) \mid 
		([\underline A],g) \in M_{\eta'}(\bW',	\alpha,\beta),
		\hspace{1pt}  \varphi_j^{t_j}(r^{out}_{j}([\underline A]))=\varphi_j^{-t_j}(r^{in}_{j}([\underline A]))\}
	\end{equation*}
	where the definitions of $W'$, $\bW'$ and $M_{\eta'}(\bW',\alpha,\beta)$ are given in the discussion before Diagram (\ref{fib-prod-1}). 
	Schematically, the set $\widetilde M_\eta(\bW,\alpha,\beta)|_H$ is equal to $F^{-1}(C)$ for a Fredholm map of Banach manifolds $F:B_1 \to B_2$ 
	and a finite dimensional smooth sub-manifold $C \subset B_2$. 
	Hypothesis \ref{eta-hypo-1} asserts that $F$ is transversal to $C$ and the tangent space of $\widetilde M_\eta(\bW,\alpha,\beta)|_H$ 
	at $p=([\underline A],g,t_{k+1},\dots,t_i)$ is given by the kernel of the composition of $D_pF$ 
	and the projection map to the the normal bundle of $C$ at the point $F(p)$. 
	More explicitly, the tangent space $T_p\widetilde M_\eta(\bW,\alpha,\beta)|_H$ is given 
	by the kernel of a map $D$ that has the following form:
	$$D: L^2_{k,\ext}(W',i\Lambda^1) \oplus T_g G' \oplus \rr^{i-k}\to \hspace{7cm}$$
	$$ \hspace{3cm} \to L^2_{k-1,\delta}(W',i\Lambda^0)_0 \oplus L^2_{k-1,\delta}(W', i\Lambda^+)\oplus 
	\cS^+_\alpha \oplus C(T_{k+1})\oplus \dots \oplus C(T_{i})\oplus \cU^-_{\beta}$$
	\begin{equation*}
		D=\left(
		\begin{array}{ccc}
		-d^*\oplus d^+ & *&*\\
		*& *&*\\
		\vdots&\vdots&\vdots\\
		*&*&*\\
		\end{array}
		\right)
	\end{equation*}
	Here $C(T_j)$ denotes the space of harmonic 1-forms on $T_j$. 
	Although the terms which are represented by $*$ above can be spelled out in more detail, 
	the only point that we need about them is that for each of them either the domain or the codomain is finite dimensional. 
	The direct sum $D \oplus id_{\cU_\beta^+}$ of $D$ and the identity map of $\cU_{\beta^+}$ has the same kernel as $D$, 
	and  there is a homotopy from this operator through the set of 
	Fredholm operators to the following one:
	\begin{equation*}
		D'=A_\alpha \oplus \mathcal D_{W'} \oplus B_\beta \oplus C_G 
	\end{equation*}
	where $A_\alpha:\{0\} \to \mathcal \cU_{\alpha}^+$, $B_\beta: \mathcal \cU_{\beta}^+ \to \{0\} $ and $C_G:T_{(g,t_{k+1},\dots,t_i)}{\bf G} \to \{0\}$.
	Orientation of $\widetilde M_\eta(\bW,\alpha,\beta)|_H$ at $p$ is determined by an orientation of the kernel (and hence the index) of $D$. 
	The homotopy to $D'$ shows the set of
	orientations of the index of $D$ is canonically isomorphic to the set of orientations of the index of $D'$. The latter can be identified with the orientation of:
	\begin{equation*} 
		-[\cU^+_\alpha]+\ind(\bW;{\bf G})+[\cU^+_{\beta}] +[T{\bf G}]
	\end{equation*}	
	at $(g,t_{k+1},\dots,t_i)$. It is also easy to check that these identifications for different choices of $H$ are compatible.
	
	In the case that $\meGab$ is 1-dimensional, the elements of $\tmeGab$ that lie over the interior of $H$ can be oriented analogously, and this orientation can be 
	extended to the boundary of $H$ in the obvious way. It is again straightforward to check that these orientations are compatible and define an orientation of $\tmeGab$.
\end{proof}

\begin{remark} \label{tensor-line-bndle}
	In the context of Lemma \ref{ori-conv} consider the special case that $G$ has only one element $g$. Let $L$ be a complex line bundle whose restriction to the ends of 
	$W:Y_0 \to Y_1$ is trivialized. Let also $A$ be a connection on $L$ such that the induced connections on $Y_0$ and $Y_1$ are trivial and $F^{+_g}(A)=0$. 
	Tensor product with the connection $A$, defines a map from $\widetilde M_\eta(W^g,\alpha,\beta)$ to itself. 
	The proof of Lemma \ref{ori-conv} shows that this map is orientation preserving. 
\end{remark}

Suppose that $\meGab$ is 0-dimensional. In the previous subsection, the matrix entry $\feW (\alpha,\beta)$ was defined by the weighted count of the points of the moduli space $\tmeGab$. To lift this matrix entry element to a $\tLz$-module homomorphism $\tfeW(\alpha,\beta):o(\alpha)\otimes \tLz \to o(\beta)\otimes \tLz$, fix a generator for each of $o(\alpha)$ and $o(\beta)$ (i.e. an orientation for each of $\cU^+_\alpha$ and $\cU^+_\beta$). We can orient (\ref{orien-KO}) using our convention to orient the direct sums. According to Lemma \ref{ori-conv}, this determines an orientation of $\tmeGab$, i.e., a sign assignment to the points of $\tmeGab$. For $p \in \tmeGab$, define $\epsilon_p:o(\alpha) \to o(\beta)$ to be the map that sends the generator of $o(\alpha)$ to the generator of $o(\beta)$ if the sign of $p$ is positive. Otherwise $\epsilon_p$ maps the generator of $o(\alpha)$ to reverse of the generator of $o(\beta)$. Define:
\begin{equation} \label{mult}
	\tfeW(\alpha,\beta):=\sum_{p\in\tmeGab}\epsilon_p\otimes u^{e(p)}.
\end{equation}

\begin{lemma} \label{signed-bdry-form-lem}
	With our conventions as above:
	\begin{equation} \label{signed-bdry-form}
		\sum_i \tilde f^{\eta}_{\bW^i_1\times \bW^i_2}=\tilde d_p\circ \tfeW-(-1)^{\dim(G)}\tfeW \circ \tilde d_p
	\end{equation}	
\end{lemma}

\begin{proof}
	This lemma is the oriented version of Lemma \ref{boundary-1} and the proof is very similar. Lemma \ref{ori-conv} explains how to orient $\tmeGab$ when it is 
	1-dimensional. The terms in (\ref{signed-bdry-form}) are in correspondence with the oriented boundary of such 1-dimensional manifolds. 
	We leave it to the reader to check that our conventions produce the signs in (\ref{signed-bdry-form}).
\end{proof}

Now let $W:Y_0 \to Y_1$ be a cobordism with a homology orientation, and $g$ be a metric on $W$. By Lemma \ref{signed-bdry-form-lem}, a $\tLz$-module homomorphism $\tPFH(W):\tPFH(Y_0)\to \tPFH(Y_1)$ can be defined by $\tilde f_{W^g}^\eta$. An analogue of Lemma \ref{functoriality} holds for $\tPFH$, namely, $\tPFH$ defines a functor from the category of 3-manifolds and cobordisms with homology orientations to the category of $\tLz$-modules. 

To characterize the map $\tPFH(W)$, we shall recall the definition of {\it pull-up-push-down} map. In the following diagram, suppose $M$, $R_0$, $R_1$ are smooth closed manifolds and $r_0$, $r_1$ are smooth maps:
\begin{equation} \label{cor-mor}
	\xymatrix{
		&M\ar[dl]_{r_0} \ar[dr]^{r_1}\\
		R_0&&R_1\\
	}
\end{equation}
The pull-up-push-down map from $H^*(R_0,\zz)$ to $H^*(R_1,\zz)$ is given by:
$$F_M:={\rm PD}^{-1}_{R_1} \circ (r_1)_*\circ {\rm PD}_M \circ (r_0)^*$$
where ${\rm PD}_X$ stands for the Poincar\'e duality isomorphism from $H^*(X,\zz)$ to $H_*(X,\zz)$. If $f_0$, $f_1$ are respectively Morse functions on $R_0$, $R_1$, then the pull-up-push-down map can be lifted to the level of Morse complexes by mapping a critical point $\alpha$ of $f_0$ to:
$$\sum_{\beta \in Cirt(f_1)} f_M(\alpha,\beta) \beta$$
Here $f_M(\alpha,\beta)$ is the signed count of the points of:
$$\{z\in M \mid (r_0(z),r_1(z))\in U_{\alpha}\times S_\beta\}$$
where the signs of the points of this set are determined by an orientation of:
$$-[\cU^+_\alpha]+[TM]-[r_1^*TR_1]+[\cU^+_\beta]$$
The involved moduli space in the definition of $\tPFH(W)$ is $M_\nu(W)$ for a generic choice of a 2-form $\nu$. Lemma \ref{cob-moduli-1} implies that $\tPFC(W)$ (and hence $\tPFH(W)$) vanishes if $b^+(W)>0$. On the other hand, if $b^+(W)=0$ and $\s$ is a $\spinc$ structure on $W$, with $\s|_{Y_i}=\st_i$, then the diagram: 
\begin{equation} \label{cor-mor}
	\xymatrix{
		&M_\nu(W,\s) \ar[dl]_{r_0} \ar[dr]^{r_1}\\
		\mathcal R(Y_0,\st_0)&&\mathcal R(Y_1,\st_1)\\
	}
\end{equation}
can be identified with:
\begin{equation} 
	\xymatrix{
		&J(W)\ar[dl]_{i_0} \ar[dr]^{i_1}&\\
		J(Y_0)&&J(Y_1)\\
	}
\end{equation}
where $i_0$ and $i_1$ are induced by the inclusion maps of $Y_0$ and $Y_1$ in $W$ (Remark \ref{res-recast-1}). Although the identification of $\mathcal R(Y_i,\st_i)$ with $J(Y_i)$ is canonical only up to translation, the identification of the homology groups of these spaces is canonical because the action of translation on the homology of a torus is trivial. Therefore, the part of $\tPFH(W)$, that is induced by the moduli space $M_\nu(W,\s)$, maps the summand $H^*(\mathcal R(Y_0,\st_0);\tLz)$ of $\tPFH(Y_0)$ by:
$$F_{J(W)}\otimes id_{\tLz}: H^*(\mathcal R(Y_0,\st_0);\tLz)=H^*(J(Y_0))\otimes \tLz \to H^*(\mathcal R(Y_1,\st_1);\tLz)=H^*(J(Y_1))\otimes \tLz $$
and maps the summands, corresponding to the other torsion $\spinc$ structures of $Y_0$, to zero. We will write $D(W;Y_0,Y_1)(\mathfrak t_0,\mathfrak t_1)$ for this map. In summary we have:
\begin{proposition} \label{mor-b^+}
	Suppose $W:Y_1 \to Y_2$ is a cobordism: 
	\vspace{-8pt}
	\begin{enumerate}
		\item[i)] If $b^+(W)>0$, then $\tPFH(W)=0$.
		\item[ii)] If $b^+(W)=0$, then: 
		$$\tPFH(W)=\sum_{\s:\mathfrak t_0 \to \mathfrak t_1} 
		u^{-c_1(\mathfrak s).c_1(\mathfrak s)}D(W;Y_0,Y_1)(\mathfrak t_0,\mathfrak t_1)$$
	\end{enumerate}
	where $\mathfrak s:\mathfrak t_0 \to \mathfrak t_1$ denotes a $\spinc$ structure $\s$ on $W$, with 
	$\mathfrak s|_{Y_0}=\mathfrak t_0$ and $\mathfrak s|_{Y_1}=\mathfrak t_1$ being torsion $\spinc$ structures.
\end{proposition}

For the purpose of proving the existence of exact triangles in the next section, we need to extend the definition of plane Floer homology again. Suppose $Y$ is a connected and closed 3-manifold and $\zeta$ is an embedded closed 1-manifold. We shall assign $\tPFH(Y,\zeta)$ to this pair that is equal to $\tPFH(Y)$ in the case that $\zeta$ is empty. This extension is defined with the aid of a local coefficient system $\Gz$ defined on $\mathcal R(Y)$. For a cobordism of the pairs $(W,Z):(Y_0,\zeta_0)\to (Y_1,\zeta_1)$, we also define a map $\tPFH(W;Z):\tPFH(Y_0;\zeta_0)\to \tPFH(Y_1,\zeta_1)$. More generally, cobordism maps can be defined for the {\it family of metrics over the pair $(W,Z)$} (Definition \ref{fam-met-pair}).

Fix a pair $(Y,\zeta)$ and let $\alpha$ be the equivalence class of a $\spinc$ connection on $Y$ under the action of the gauge group. A triple $(W,Z,[\underline A])$ {\it fills} $(Y,\zeta,\alpha)$, if $W$ is a compact oriented 4-manifold such that $\partial W=Y$, $Z$ is a compact, oriented and embedded surface in $W$ with $\partial Z=\zeta$, and $[\underline A]$ is the equivalence class of a $\spinc$ connection $\underline A$ on $W$ with $[\underline A]|_{\partial W} = \alpha$. For technical convenience, we assume that the restriction of $\underline A$ to a collar neighborhood of $\partial W$ is given by the pull-back of a representative of $\alpha$. 

\begin{lemma}
	Let $(W,Z,[\underline A])$ fills $(Y,\zeta,\alpha)$ and fix a framing for the 1-manifold $\zeta$. Use the framing to define:
	\begin{equation} \label{iota}
		\iota(W,Z,[\underline A]):= \frac{1}{2}(Z\cdot Z- \frac{i}{2\pi}\int_Z F(A))
	\end{equation}	
	This function is constant under the homotopies of the connection $[\underline A]$ that fix $[\underline A]|_{Y}$.
	If $(W',Z',[\underline A'])$ is another triple that fills $(Y,\zeta,\alpha)$, then
	 $\iota(W,Z,[\underline A])-\iota(W',Z',[\underline A'])$ is an integer independent of the choice of the framing.
\end{lemma}

\begin{proof}
	Glue the reverse of $(W',Z',[\underline A'])$ to $(W,Z,[\underline A])$ in order to form 
	$(W\#W',Z\#Z',[\underline A\# \underline A'])$:
	$$\iota(W,Z,[\underline A])-\iota(W',Z',[\underline A'])=\frac{1}{2}(Z\#Z'\cdot Z\#Z' - \frac{i}{2\pi}\int_Z F(A\#A'))=
	\frac{1}{2}(Z\#Z'\cdot Z\#Z' - c_1(\s\#\s')[Z\#Z'])$$
	where $\s\#\s'$ is the $\spinc$ structure that carries $[\underline A\#\underline A']$. 
	Since $c_1(\s\#\s') \equiv \omega_1(W\#W')$ mod 2, Wu's formula asserts that 
	$\iota(W,Z,[\underline A])-\iota(W',Z',[\underline A'])$ is an integer independent of the framing. Also, this topological interpretation implies that $\iota$ is constant under
	a homotopy that fixes the restriction of $[\underline A]$ to the boundary.
\end{proof}

Let $\Gz(\alpha)$ be the abelian group generated by all the triples that fill $(Y,\zeta,\alpha)$ with the relation that for any two triples $(W,Z,[\underline A])$, $(W',Z',[\underline A'])$ we have $(W,Z,[\underline A])=\pm (W',Z',[\underline A'])$ where the sign is positive or negative depending on whether $\iota(W,Z,[\underline A])-\iota(W',Z',[\underline A'])$ is even or odd. Take a cobordism of the pair $(W,Z):(Y_0,\zeta_0)\to (Y_1,\zeta_1)$ and let $[\underline A]$ be the equivalence class of a $\spinc$ connection on $W$ such that its restrictions to $Y_0$, $Y_1$ are equal to $\alpha_0$, $\alpha_1$, respectively. Then we can define a map of the local coefficient systems:
$$\GZ ([\underline A]): \Gzz(\alpha_0) \to \Gzo(\alpha_1)$$
$$\GZ([\underline A])([W_0,Z_0,[\underline A_0]])=(-1)^{\iota(W_0\#W\#W_1,Z_0\#Z\#Z_1,[\underline A_0\# \underline A\#\underline A_1])}[W_1,Z_1,[\underline A_1]]$$
where $(W_0,Z_0,[\underline A_0])$, $(W_1,Z_1,[\underline A_1])$ are fillings of $(Y_0,\zeta_0,\alpha_0)$, $(Y_1,\zeta_1,\alpha_1)$, respectively. Here, $\iota$ is defined as in (\ref{iota}). This definition is well-defined and does not depend on the choice of the generators of $\Gzz(\alpha_0)$ and $\Gzo(\alpha_1)$.

Suppose $\{\alpha_t\}$ is a path of $\spinc$ connections. This path determines a $\spinc$ connection (in temporal gauge) on $I \times Y$. Thus for any 1-manifold $\zeta \subset Y$, we can define an isomorphism $\Gamma_{I \times \zeta}(\{\alpha_t\}):\Gz(\alpha_0)\to \Gz(\alpha_1)$. This map depends only on the homotopy class of the path.  Thus $\Gz$ defines a local coefficient system on $\mathcal R(Y)$. The plane Floer homology of the pair $(Y,\zeta)$, denoted by $\tPFH(Y,\zeta)$, is the homology of $\mathcal R(Y)$ with coefficients in $\Gz$. The Morse complex that realizes this homology group is denoted by $\tPFC(Y,\zeta)$(cf. \cite[Section 2]{KM:monopoles-3-man}). In particular, for each critical point $\alpha$, this complex has a generator of the form $\Gz(\alpha)\otimes o(\alpha)\otimes \tLz$.

\begin{lemma} \label{mod-2-dep-loc-coef}
	If $\zeta$ and $\zeta'$ are 1-dimensional embedded sub-manifolds in a 3-manifold $Y$, inducing the same elements of
	$H_1(Y;\zz/2\zz)$, then the local coefficient systems $\Gamma_\zeta$ and $\Gamma_{\zeta'}$ are isomorphic. 
	In particular, $\tPFH(Y,\zeta)\cong \tPFH(Y,\zeta')$.
\end{lemma}

\begin{proof}
	Let $\gamma$ be a closed loop in $\mathcal R(Y)$. This loop induces a $\spinc$ connection $[\underline A]$ on $S^1 \times Y$. Since $S^1 \times \zeta$, 
	$S^1 \times \zeta$ have zero self-intersection and $S^1 \times Y$ is spin, the difference between the following numbers is even:
	$$\frac{1}{2}((S^1\times \zeta) \cdot (S^1\times \zeta)- \frac{i}{2\pi}\int_{S^1 \times \zeta} F(A)) \hspace{1cm}
	\frac{1}{2}((S^1\times \zeta') \cdot (S^1\times \zeta')- \frac{i}{2\pi}\int_{S^1 \times \zeta'} F(A))$$
	Therefore, $\Gz$ and $\Gamma_{\zeta'}$ are isomorphic local coefficient systems.
\end{proof}

\begin{definition} \label{fam-met-pair}
	A family of non-broken metrics $(\bW,\bZ)$ on the pair $(W,Z)$ and parametrized by $G$ 
	consists of a family of non-broken metrics $\bW$ over $W$ 
	and an embedded manifold $\bZ\subset \bW$ such that $\bZ$ is a fiber bundle over $G$ with fiber $Z$. 
\end{definition}

Given a pair $(W,Z)$, let $T$ be a cut in $W$ transversal to $Z$ and $\xi$ be the embedded 1-manifold $Z \cap T$.
Let also $\bigcup_{j=1}^i T_j$ be the decomposition of $T$ to the connected components, and $\xi_j=\xi\cap T_j$. Removing
$T$ decomposes $W$ and $Z$ to $W_0 \coprod \dots \coprod W_i$ and $Z_0 \coprod \dots \coprod Z_i$ such that 
$Z_j \subset W_j$. Fix the labeling such that the pair $(T_j,\xi_j)$ is equal to the incoming end of $(W_j,Z_j)$ and an outgoing
end of $(W_{o(j)},Z_{o(j)})$:
\begin{definition} \label{pair-fam-met}
	A broken metric on $(W,Z)$ with a cut along $(T,\xi)$ is given by a metric on $W_j$ for each $j$ 
	such that the following holds. After an appropriate identification of the incoming end of $(W_j,Z_j)$ 
	(respectively the outgoing end of $(W_{o(j)},Z_{o(j)})$ in correspondence with $(T_j,\xi_j)$) 
	with $\rr^{\leq0}\times (T_j,\xi_j)$ (respectively with $\rr^{\geq0}\times (T_j,\xi_j)$), the metric on 
	$\rr^{\leq 0} \times T_j\subset W_j$ (respectively $\rr^{\geq 0} \times T_j\subset W_{o(j)}$) is given by the product metric.
	Furthermore, we require that these metrics are given by the same metrics on $T_j$.
\end{definition}

A broken metric as above can be used to define a family of (possibly broken) metrics on $(W,Z)$ that is parametrized with $[0,\infty]^i$. As in the case of the family of metrics on a cobordism $W$, we can combine the definitions \ref{fam-met-pair} and \ref{pair-fam-met} to define a family of (possibly broken) metrics on a pair $(W,Z)$.

A family of metrics $(\bW,\bZ)$ on the pair $(W,Z):(Y_0,\zeta_0) \to (Y_1,\zeta_1)$ gives rise to a map between the chain complexes $\tPFC(Y_0,\zeta_0)$, $\tPFC(Y_1,\zeta_1)$. To that end, we need to extend the definition of $\GZ([\underline A])$ to the case that $[\underline A]$ is an element of $\tmeGab$. Fix $p\in \tmeGab$ and without loss of generality assume that  $p=([\underline A],g,t_{k+1},\dots,t_i)$ is an element of $\widetilde M_\eta(\bW,\alpha,\beta)|_H$, defined in (\ref{sub-tmGabe}). Then $\underline A$ is a $\spinc$ connection on $W'=W\backslash \bigcup_{j=k+1}^n T_j$. For each $j$, the flow $\varphi_j$ determines a path from  $r^{out}_{j}([\underline A])$ to $r^{in}_{j}([\underline A])$. Use the same methods as in subsection \ref{ab-asd-eq} to glue these paths and $\underline A$ to form a $\spinc$ connection on $W$. This connection can be used to define a homomorphism $\GZ([A]):\Gzz(\alpha) \to \Gzo(\beta)$. This map does not depend on the choices involved in gluing and hence is well-defined. We can define a map $\tfeWZ:\tPFC(Y_0,\zeta_0)\to \tPFC(Y_1,\zeta_1)$ with the following matrix entry:
\begin{equation*}
	\tfeWZ(\alpha,\beta):\Gzz(\alpha)\otimes o(\alpha)\otimes \tLz \to \Gzo(\beta)\otimes o(\beta)\otimes \tLz
\end{equation*}	
\begin{equation} \label{mult}
	\tfeWZ(\alpha,\beta):=\sum_{p\in\tmeGab} \GZ(p)\otimes \epsilon_p\otimes u^{e(p)}
\end{equation}	
This map has degree:
\begin{equation} \label{ori-morphism-deg}
	\deg_p(\tfeWZ)=2\dim(G)-\sigma(W)-\chi(W)
\end{equation}	
Because the map $\GZ$ is locally constant, an analogue of (\ref{signed-bdry-form}) holds for the map $\tfeWZ$:
\begin{equation} \label{pair-signed-bdry-form-loc-sys}
		\sum_i \tilde f^{\eta}_{(\bW^i_1,\bZ^i_1) \times (\bW^i_2,\bZ^i_2)}=
		\widetilde d_p\circ \tfeWZ-(-1)^{\dim(G)}\tfeWZ \circ \widetilde d_p
\end{equation}	
As a result of (\ref{pair-signed-bdry-form-loc-sys}), we have a well-defined map $\tPFH(W,Z):\tPFH(Y_0,\zeta_0)\to \tPFH(Y_1,\zeta_1)$.
\begin{remark} \label{bdry-terms}
	As in the unoriented case, it is more convenient to write down $\tilde f^{\eta}_{(\bW^i_1,\bZ^i_1) \times (\bW^i_2,\bZ_2^i)}$
	 in the left hand side of (\ref{signed-bdry-form}) in terms
	of $\tilde f^{\eta}_{(\bW^i_1,\bZ_1^i)}$ and $\tilde f^{\eta}_{(\bW^i_2,\bZ_2^i)}$. However, due to possible incompatibilities in orienting moduli spaces, 
	more care is needed.
	If $W\backslash Z_i$ is the composition of the cobordisms $W^i_1: Y_0 \to Z_i$ and 
	$W^i_2: Z_i \to Y_1$, then: 
	$$f^{\eta}_{(\bW^i_1,\bZ_1^i) \times (\bW^i_2,\bZ_2^i)}= (-1)^{\dim(G^i_1)\times \dim(G^i_2)}  f^{\eta}_{(\bW^i_1,\bZ_1^i)} \circ f^{\beta}_{(\bW^i_2,\bZ_2^i)}.$$
	Next, let $W\backslash Z_i$ be the union of the cobordisms $W^i_1: Y_0 \to Z_i\coprod Y_1$ and $W^i_2: Z_i \to \emptyset$. If $\dim(G^i_2)$ is an even number
	then a similar formula holds:
	$$f^{\eta}_{(\bW^i_1,\bZ_1^i) \times (\bW^i_2,\bZ_2^i)}=(-1)^{\dim(G^i_1)\times \dim(G^i_2)}(f^{\eta}_{(\bW^i_2,\bZ_2^i)} \otimes id) \circ f^{\eta}_{(\bW^i_1,\bZ_1^i)}.$$
	In the case that $\dim(G^i_2)$ is an odd number, we need to introduce $s:\tPFC(Y,\zeta) \to \tPFC(Y,\zeta)$ in the following way:
	$$s(\alpha):=(-1)^{b_1(Y)-\ind(\alpha)}\alpha=(-1)^{\frac{\deg_p(\alpha)-b_1(Y)}{2}}\alpha$$
	where $\alpha$ is a generator of $\tPFC(Y,\zeta)$. The map $s$ commutes with the cobordism maps up to a sign. For example, if $(\bW,\bZ)$ is a family of
	metrics parametrized with $G$ on the cobordism $(W,Z):(Y_0,\Gamma_0)\to (Y_1,\Gamma_1)$, with $Y_0$ and $Y_1$ being connected, then:
	\begin{equation} \label{com-s}
		s\circ \tfeWZ=(-1)^{\ind(\mathcal D_W)+\dim(G)}\tfeWZ\circ s.
	\end{equation}	
	For an odd dimensional family of metrics $G^i_2$, we have:
	$$f^{\eta}_{(\bW^i_1,\bZ_1^i) \times (\bW^i_2,\bZ_2^i)}=(-1)^{\dim(G^i_1)\times \dim(G^i_2)}(f^{\eta}_{(\bW^i_2,\bZ_2^i)} \otimes s) \circ f^{\eta}_{(\bW^i_1,\bZ_1^i)}.$$
\end{remark}

In the following, we study the cobordism map $\tfeWZ$ in two special cases. Suppose $Z$ is an embedded sphere in $\overline {\cc P}^2$with self-intersection -1. Let $D^4$ be a disc neighborhood of a point away from $Z$ and consider $R:= \overline {\cc P}^2\backslash D^4$. Equip the pair $(R,Z)$, as a cobordism from $(S^3,\emptyset)$ to the empty pair, with a metric that is cylindrical on the end. Let $\tPFC(R,Z):\tPFC(S^3) \to \tLz$ be the cobordism map induced by an arbitrary homology orientation of $R$ and the zero perturbation. The representation variety $\mathcal R(S^3)$ consists of only one point $\alpha$ and $b^+(R)=b^1(R)=0$. Therefore, the zero perturbation satisfies Hypothesis \ref{eta-hypo-1}, and the moduli space $M(R)$ consists of one $\spinc$ connection for each choice of $\spinc$ structure. Any $\spinc$ structure $\s$ on $R$ is uniquely determined by the pairing of $c_1(\s)$ and $Z$ which is an odd number. Let $\s_{2k+1}$ be the $\spinc$ structure that this pairing is equal to $2k+1$ and $\underline A_{2k+1}$ is the corresponding element of $M(R)$. The $\spinc$ connections $\underline A_{2k+1}$, $\underline A_{-2k-1}$ have the same energy and  $\GZ(\underline A_{2k+1})=-\GZ(\underline A_{-2k-1})$. Furthermore, by Remark \ref{tensor-line-bndle} the orientation of the moduli space at these two points match with each other. Consequently, we have the following result, which is our main motivation to introduce plane Floer homology for cobordism pairs:
\begin{lemma} \label{special-fam}
	The cobordism map $\tPFC(R,Z)$ is equal to zero.
\end{lemma}

\begin{remark} \label{special-fam-2}
	Removing a ball neighborhood of a point in $Z \subset \overline {\cc P}^2$ produces a cobordism of the 
	pairs $(R,Z\backslash D^2)$ from $(S^3,S^1)$ to the empty pair. 
	A similar argument as above shows that $\tPFC(R,Z\backslash D^2)=0$. 
\end{remark}

Next, consider the cobordism $U$ in Example \ref{example-family-metrics-1} from $Q$ to the empty 3-manifold. Extend $U$ to a cobordism of pairs $(U,Z): (Q,\emptyset) \to (\emptyset,\emptyset)$ by defining $Z:=S_1$. The family of metrics $\bU$ in Example \ref{example-family-metrics-1} determines a family of metrics $(\bU,\bZ)$ on the pair $(U,Z)$. Fix a Morse-Smale function on $\mathcal R(Q)$ with two critical points $\alpha$ of index 0 and $\alpha'$ of index 1 such that neither of them is equal to $h_0$ (defined in Example \ref{example-family-metrics-1}) or $-h_0$. Also fix a homology orientation for $U$ and a generator for $o(\alpha)$. By the characterization of moduli spaces in Example \ref{example-family-metrics-1}, the cobordism map $f_{\bU,\bZ}:\tPFC(Q) \to \tLz$, induced by the zero perturbation, sends the corresponding generator of $o(\alpha)$ to:
\begin{equation} \label{constant-c}
	c=\sum_{i=0}^\infty \pm (2k+1)u^{(2k+1)^2}\in \tLz
\end{equation}	
for an appropriate choices of the signs. The ring element $c$ can be regarded as an element of $\tL$ and independent of what the signs are it is invertible in $\tL$. On the other hand, $f_{\bU,\bZ}$ maps the summand $o(\alpha')$ to zero.

\begin{lemma} \label{quasi-iso}
	Consider $(V:=S^1 \times D^3,\emptyset)$ as a cobordism form $(Q,\emptyset)$ to the empty pair. Then there exists a homology orientation of 
	$V$ such that: 
	$$\tilde f_{\bU,\bZ}=c\cdot\tPFC(V)$$ 
\end{lemma}

\begin{proof}
	Since $M(V)=\mathcal R(Q)$ and the outgoing end of $V$ is empty, the cobordism map $\tPFC(V)$, for appropriate 
	choices of homology orientations of $V$ and generators of $o(\alpha)$, maps the generator of $o(\alpha)$ to $1\in \tLz$
	 and  the summand $o(\alpha')$ to zero. This verifies the claim of the lemma.
\end{proof}

\begin{remark} \label{fam-met2}
	Consider two 2-spheres $S_1$ and $S_2$. Then the cobordism $U$ of Example \ref{example-family-metrics-1} is diffeomorphic to the plumbing of two disc bundles 
	with Euler number -1 over these spheres. This diffeomorphism can be chosen such that $S_1$ is in correspondence with
	the previously defined $S_1$, and $S_2$ is sent to $S_1' $ with an orientation reversing map. Let $\widehat{Z}$ be the 
	embedded surface in $U$ that is union of a fiber of the disc bundle over $S_1$ and a fiber of the disc bundle over $S_2$.
	The family of metrics $\bU$ can be lifted to a family of metrics $(\bU,\widehat \bZ)$ over the pair $(U,\widehat{Z})$.
	The intersection $\widecheck{\zeta}=\widehat{Z}\cap Q$ is equal to 
	$S^1 \times \{p,q\} \subset Q \cong S^1 \times S^2$ where $p,q \in S^2$
	and the induced orientations on these two circles are reverse of each other. Since the orientations of these two circles
	are different, there is an embedded annulus $\widecheck Z$ in $V$ such that its intersection with the boundary of $V$, 
	identified with $Q$, is equal to $\widecheck \zeta$. Therefore, $(V,\widecheck{Z})$ is another cobordism from 
	$(Q,\widecheck{\zeta})$ to the empty pair. The proof of the following relation is analogous to that of Lemma \ref{quasi-iso}:
	$$\tilde f_{\bU,\widehat{\bZ}}=c\cdot\tPFC(V,\widecheck{\zeta})$$ 
\end{remark}

\section{Link Surgery Cube Complex} \label{ex-tri}
Suppose $(C_1,d_1)$ and $(C_2,d_2)$ are two (not necessarily $\zz$-graded) chain complexes. A map $f:C_1 \to C_2$ is an anti-chain map if:
$$f\circ d_1+d_2 \circ f=0$$
Composition of two maps which are either chain map or anti-chain map is again either a chain map or an anti-chain map. Two anti-chain maps  $f_1, f_2:C_1\to C_2$ are chain homotopy equivalent if there is a map $h:C_1 \to C_2$ such that:
$$f_2-f_1=h\circ d_1-d_2 \circ h$$
A chain homotopy of chain maps is also defined in the standard way. Define $\mc$ to be the category whose objects are the collection of chain complexes over a fixed ring $R$. A morphism from $(C_1,d_1)$ to $(C_2,d_2)$ is represented by either a chain map or an anti-chain map and two such morphisms represent the same elements if they are chain homotopy equivalent.

Given a general category $\mathcal C$, we can define the {\it isomorphism equivalence relation} on the objects of $\mathcal C$. That is to say, two objects are equivalent if there is an isomorphism between them. The equivalence class of an object with respect to this relation is called the {\it isomorphism class} of that object. In the case of the category $\mc$, the isomorphism class of an object is also called the chain homotopy type of the corresponding chain complex. 

Given a pair $(Y,\zeta)$, the chain complex $\tPFC(Y,\zeta)$ was associated with this pair in the previous section. We use $\tL$ as the coefficient ring in this section and hence $\tPFC(Y,\zeta)$ is an object of $\mctL$. We shall show that, for any framed link $L$ in $Y\backslash \zeta$, one can construct a new chain complex that has the same chain homotopy type as $\tPFC(Y,\zeta)$. In subsection \ref{3-man-link}, starting with the pair $(Y,\zeta)$ and $L$, various families of metrics are introduced. These families of metrics are used in subsection \ref{triangles} to construct the new chain complex and to show that this chain complex has the same chain homotopy type as $\tPFC(Y,\zeta)$.

\subsection{Framed Links and Families of Metrics} \label{3-man-link}

Fix the lexicographical ordering on the lattice $\zz^n$ which is induced by the standard ordering of integer numbers. For $\bm=(m_1\dots,m_n)\in\zz^n$, define: 
$$|\bm|_1:=\sum_{i=1}^n m_i \hspace{2cm} |\bm|_\infty=\max\{m_1,\dots,m_n\}$$
\begin{definition}
	A {\it cube} of dimension $d$ in $\zz^n$ consists of a pair $(\bm,\bn)$ such that $\bm > \bn$, $|\bm-\bn|_\infty=1$ and $|\bm-\bn|_1=d$. 
	A 1-dimensional (respectively, 2-dimensional) cube is called an edge (respectively, a face) of $\zz^n$. The set of edges and faces of $\zz^n$ are denoted by 
	$\mathcal E(n)$ and $\mathcal F(n)$ .
\end{definition}
\begin{definition}
	 A {\it double-cube} in $\zz^n$ is a pair $(\bm,\bn)$ such that $\bm > \bn$, $|\bm-\bn|_\infty=2$, 
	 and the entries of $\bm$, $\bn$ differ by 2 in exactly one instance. 
	 A {\it triple-cube} in $\zz^n$ is a pair $(\bm,\bn)$ such that $\bm > \bn$, $|\bm-\bn|_\infty=3$. 
	 Furthermore, we require that the entries of $\bm$, $\bn$ differ by 3 in exactly one instance and otherwise they differ by at most 1. 
	 Dimension of a double-cube (respectively, triple-cube) is equal to $|\bm-\bn|_1-1$ (respectively, $|\bm-\bn|_1-2$).
\end{definition}

Suppose $(Y,\zeta)$ is a pair of a 3-manifold and an embedded closed 1-manifold. Let $L=\bigcup_{1\leq i \leq n} K_i$ be an oriented framed link with $n$ connected components in $Y\backslash \zeta$. The complement of a regular neighborhood of $L$ in $Y$, denoted by $Y^L$, is a 3-manifold with boundary $\bigcup_{1\leq i \leq n} \tau_i$ where $\tau_i$ is the torus boundary of a neighborhood of $K_i$. The complement of a regular neighborhood of only the $i^{\rm th}$ component is denoted by $Y^{K_i}$. Since $K_i$ is framed, $\tau_i$ has well-defined oriented longitude and meridian and we will denote them by $\lambda_i$ and $\mu_i$. In particular, the orientation of $\tau_i=\partial Y^{K_i}$ determined by the outward-normal-first convention agrees with the orientation determined by the ordered pair $\{\mu_i,\lambda_i\}$. Let $\gamma_i^{0}$, $\gamma_i^{1}$, and $\gamma_i^{2}$ be oriented curves in $\tau_i$ in the homology class of $\lambda_i$, $-(\lambda_i+\mu_i)$, and $\mu_i$. We extend these curves to the family $\{\gamma_i^{j}\}_{j \in \mathbb Z}$ by requiring that this family is 3-periodic in $j$. Given ${\bf m}=(m_1,\dots,m_n)\in\zz^n$, the closed 3-manifold $Y_{\bf m}$ is defined to be the result of the Dehn filling of $Y^L$ along the curves $\gamma_i^{m_i}$. We will write $D_{\bm,i}$ for the solid torus which is glued to $Y^L$ along $\tau_i$. Note that $\gamma_i^{m_i-1}$ defines a framing of this solid torus. The 1-manifold $\zeta_\bm$ is equal to the union of $\zeta$ and $\bigcup_i \zeta_{\bm,i}$, where $\zeta_{\bm,i}$ is the core of $D_{\bm,,i}$ if $m_i \equiv 0$ mod 3, and is empty otherwise. Fix a metric on $Y^L$ which is the standard product metric in the union $\bigcup_{1\leq i \leq n} \tau_i$, and extend this to $Y_\bm$ by the standard metric on $D_{\bm, i}$.

For each pair $(\bm,\bn)$ that ${\bf m} \geq {\bf n}$, a cobordism of the pairs $(\Wmn,\Zmn):(\ym,\zm) \to (\yn,\zn)$ can be defined as follows. Firstly let $(\bm,\bn)$ be a 1-dimensional cube and the $i^{\rm th}$ entry be the only coordinate that $\bm$ and $\bn$ differs. In this case, define $\Wmn$ to be the elementary cobordism from $\ym$ to $\yn$ given by attaching a 2-handle along $D_{\bm,i}$. More concretely, this cobordism can be constructed as the union of $[0,1]\times Y^{K_i}$ and the space:
$$H=\{(z_1,z_2)\in \cc^2\mid |z_1|, |z_2|\leq 1, |z_1|\cdot |z_2|\leq \frac{1}{2}\}$$
Boundary of $H$ is the union of the following sets (Figure \ref{2-handle}):
$$\partial_0 H=\{(z_1,z_2)\in \cc^2 \mid |z_1|= 1, |z_2|\leq \frac{1}{2}\} 
\hspace{1cm}\partial_1 H=\{(z_1,z_2)\in \cc^2 \mid |z_1|\leq \frac{1}{2}, |z_2|= 1\}$$
$$\partial_2 H=\{(z_1,z_2)\in \cc^2 \mid |z_1|, |z_2|\leq 1, |z_1|\cdot |z_2|= \frac{1}{2}\}$$
\begin{figure}
	\centering
		\includegraphics[width=.5\textwidth]{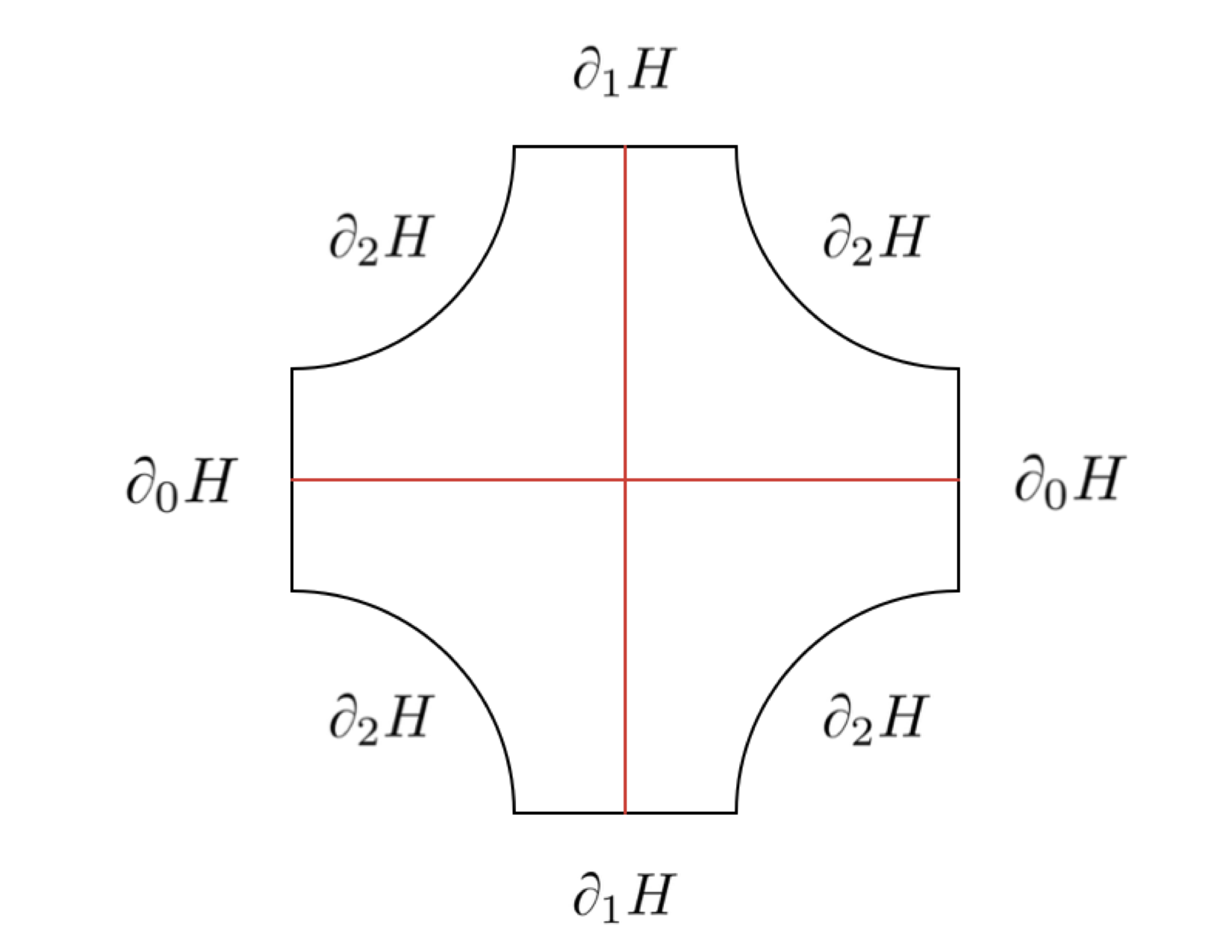}
		\caption{A schematic picture of a 2-handle; the horizontal red line represents the core of the 2-handle and the vertical red line is the representative of the co-core}
		\label{2-handle}
\end{figure}
With the aid of the negative gradient flow of the Morse function $|z_1|^2-|z_2|^2$, the set $\partial_2 H $ can be identified with the product of the interval $[0,1]$ and the 2-dimensional torus. The space $[0,1]\times \partial Y^{K_i}$ can be identified with the same space using the closed curves $\gamma_i^{m_i}$ and $\gamma_i^{m_i-1}$ in $\partial Y^{K_i}$. The cobordism $\Wmn$ is given by gluing $[0,1] \times Y^{K_i}$ to $H$ along these parts of their boundaries and then adding the cylindrical ends. Thus $\Wmn$ is equal to the union of $\rr \times Y^{K_i}$ and $\bar H$  where $\bar H:= H\cup \rr^{\leq 0} \times \partial_0 H \cup \rr^{\geq 0} \times \partial_1 H$. The 2-dimensional cobordism $\Zmn:\zm \to \zn$ is determined by the value of $m_i=n_i+1$ mod 3. If $m_i \equiv 0$ mod 3, then $\zeta_{\bn,i}$ is empty and $\Zmn$ is the union of $\rr \times \zn $ and the following cobordism from $\zeta_{\bm,i}$ to the empty set:
\begin{equation} \label{S-1}
	\{(z_1,0) \mid |z_1|\leq 1\} \cup \rr^{\leq 0} \times \{z_1,0) \mid |z_1|= 1\} \subset H \cup \rr^{\leq 0} \times \partial_0 H
\end{equation}	
If $m_i \equiv 1$ mod 3, then $\zm$ is empty and $\Zmn$ is the union of $\rr \times \zm$ and the following cobordism from the empty set to $\zeta_{\bn,i}$:
\begin{equation} \label{S-2}
	\{(0,z_2) \mid |z_2|\leq 1\} \cup \rr^{\geq 0} \times \{0,z_2) \mid |z_2|= 1\} \subset H \cup \rr^{\geq 0} \times \partial_1 H.
\end{equation}	. Finally, if $m_i \equiv 2$ mod 3, then $\zeta_{\bm,i}$ and $\zeta_{\bn,i}$ are both empty and $\Zmn$ is defined to be $\rr \times \zm$

For a general pair $(\bm,\bn)$, we can find a sequence ${\bf m}={\bf k_0} > {\bf k_1} > \dots {\bf k_{l-1}} > {\bf k_l}={\bf n}$ such that $|{\bf k_i} -{\bf k_{i+1}}|_1=1$. In this case we define $\Wmn=W_{\bf k_0}^{\bf k_1}\circ W_{\bf k_1}^{\bf k_2}\circ\dots \circ W_{\bf k_{l-1}}^{\bf k_l}$, $\Zmn=Z_{\bf k_0}^{\bf k_1}\circ Z_{\bf k_1}^{\bf k_2}\circ\dots \circ Z_{\bf k_{l-1}}^{\bf k_l}$. It is clear that the diffeomorphism type of $(\Wmn,\Zmn)$ is independent of how we choose ${\bf k_i}$'s. 
\begin{lemma} \label{fix-hom-ori}
	For each pair $(\bm,\bn)$ with $\bm\geq \bn$, there exists a homology orientation $o(\bm,\bn)$ of the cobordism $\Wmn$ such that if $\bm\geq\bk\geq \bn$, then
	the composition of $o(\bm,\bk)$ and $o(\bk,\bn)$ is equal to $o(\bm,\bn)$. Furthermore, if $o'(\bm,\bn)\in o(\Wmn)$ is another series
	of homology orientations that  has the same property, then there exists $\mathfrak z:\zz^n \to \{\pm 1\}$, such that:
	\begin{equation} \label{additiviity-2}
		o'(\bm,\bn)=\mathfrak z(\bm)\mathfrak z(\bn)o(\bm,\bn).
	\end{equation}	
	The function $\mathfrak z$ is unique up to an overall sign.
\end{lemma}
\begin{proof}
	The composition of homology orientations is associative and hence the first part of this lemma can be proved with the same argument as in 
	\cite[Lemma 6.1]{KM:Kh-unknot}. To prove the second part of the lemma, note that $\mathfrak u$, defined as the difference of $o$ and $o'$, is a map from the set of 
	$(\bm,\bn)\in \zz^n\times \zz^n$, with $\bm\geq \bn$, to $\zz/2\zz$ that satisfies:
	$$\mathfrak u (\bm,\bn)=\mathfrak u (\bm,\bk)+\mathfrak u (\bk,\bn)$$
	for $\bm\geq \bk \geq \bn$. This identity implies that $\mathfrak u$ induces a simplicial co-cycle on $\rr^n$ with respect to the simplicial structure spelled out in 
	\cite{KM:Kh-unknot}. Since $\rr^n$ is contractible, there is a map 
	$\mathfrak z:\zz^n \to \zz/2\zz$ such that:
	$$\mathfrak u(\bm,\bn)=\mathfrak z(\bm)-\mathfrak z(\bn)$$
	that is to say $\mathfrak z$ satisfies (\ref{additiviity-2}).
\end{proof}

The pairs $(Y_\bm,\zeta_\bm)$ and $(\Wmn,\Zmn)$ are 3-periodic. That is to say, for any $\bk \in \zz^n$ we have the natural identifications $(Y_{\bm+3\bk},\zeta_{\bm+3\bk})=(Y_\bm,\zeta_\bm)$ and $(W_{\bm+3\bk}^{\bn+3\bk},Z_{\bm+3\bk}^{\bn+3\bk})=(\Wmn,\Zmn)$. However, Lemma \ref{fix-hom-ori} does not guarantee that the homology orientations for the cobordisms $\Wmn$ and $W_{\bm+3\bk}^{\bn+\bk}$ are the same. Nevertheless, as a consequence of Lemma \ref{fix-hom-ori}, there exists $s_{\bk}:\zz^n \to \{\pm 1\}$ such that $o(\bm+3\bk,\bn+3\bk)=s_{\bk}(\bm)s_{\bk}(\bn)o(\bm,\bn)$. 

The rest of this subsection is devoted to define families of metrics on $(\Wmn,\Zmn)$, in the case that $(\bm,\bn)$ represents a cube, double-cube, or a triple-cube . An important ingredient of the proof of exact triangles in the next subsection is the existence of such families of metrics. Almost the same families of metrics are described in \cite{Bl:HM-branch} which in turn are based on those of \cite{KMOS:mon-lens} (cf. also \cite{CS:iCube}). Families of orbifold metrics with similar formal properties play a similar role in \cite{KM:Kh-unknot}. 

We start with the case that $(\bm,\bn)$ is a $d$-dimensional cube in $\zz^n$ and construct a family of metrics of dimension $d-1$. For the sake of simplicity of exposition, we assume that $\bm=(1,\dots,1)+\bn$. The cobordism $\Wmn$ in this case is produced by gluing 2-handles along disjoint curves in $\ym$. This cobordism can be decomposed to the union $\rr\times Y^L \cup \bar H_1 \cup \dots \cup \bar H_n$. Here $\bar H_i$ is a copy of $\bar H$ corresponding to the $i^{\rm th}$ component of $L$. The cobordism $\Zmn$ also decomposes to the union $\rr\times \zeta \cup Z_1 \cup \dots \cup Z_n$ where $\rr\times \zeta \subset \rr \times Y^L$ and $ Z_i \subset \bar H_i$. The product metric on $\rr\times Y^L$ and the standard metrics on $\bar H_i$ determine a metric on $(\Wmn,\Zmn)$. 
Given $t_i \in \rr$ for each $1\leq i \leq n$, we can compose the gluing map of $\bar H_i$ to $\rr \times Y^L$ with a translation in the $\rr$ direction and produce a new metric on $(\Wmn,\Zmn)$. This family of metrics is parameterized by $\rr^n$. The fibers of this family over $(t_1,\dots,t_n)$ and $(t_1 + t,\dots,t_n+t)$ can be naturally identified with each other. Thus by considering only the metrics which are parametrized by $\rr^{n-1}=\{(t_1,\dots,t_n)\in \rr^n\mid \sum t_i=0\}$, we can come up with a more effective family of metrics. This family can be compactified to a family $(\bWmn,\bZmn)$ by allowing the difference of some of the coordinates to be equal to $\infty$. More precisely, $(\bWmn,\bZmn)$ is parametrized by a polyhedron $\Gmn$ such that a face of codimension $i$ of $\Gmn$ is determined by a partition of $\{1,\dots,n\}$ to $i+1$ (labeled) sets $A_0=\{a^0_1,\dots,a^0_{j_0}\}$, $\dots$, $A_i=\{a^i_1,\dots,a^i_{j_i}\}$. Roughly speaking, this face parametrizes the set of points $(t_1,\dots,t_n)$ that the entries with indices in $A_{l+1}$ have larger values than those with indices in $A_{l}$ and in fact the difference is $\infty$. In order to give a more detailed definition and for the sake of simplicity of exposition, assume that:
$$A_0=\{1,\dots,j_0\}\hspace{.5cm} A_1=\{j_0+1,\dots,j_0+j_1\} \hspace{.5cm} \dots \hspace{.5cm} A_i=\{(\sum_{l=0}^{i-1} j_l)+1,\dots,n=\sum_{l=0}^{i} j_l\}$$ 
and for $1\leq l \leq i$ define $\bk_l$ to be:
$$\bk_l:=(\underbrace{1,\dots,1}_{j_0+\dots+ j_{l-1}},\underbrace{0,\dots,0}_{j_{l}+\dots+j_i})+\bn$$
Then closure of the corresponding face of $\Gmn$ can be identified with $G_{\bm}^{\bk_{i}}\times G_{\bk_{i}}^{\bk_{i-1}}\times \dots \times G_{\bk_1}^{\bn}$. A point in this face has the form $((t_1,\dots,t_{j_0}),(t_{j_0+1},\dots,t_{j_0+j_1}),\dots,(t_{(\sum_{l=0}^{i-1} j_l)+1},\dots,t_n))$ such that $t_{j_0+\dots+j_{l-1}+1}+\dots+t_{j_0+\dots+j_l}=0$ for each $l$. Intersection of a neighborhood of this point and the interior of $\Gmn$ parametrizes the normalized version of the points of the following form:
$$(t'_1,\dots,t'_{j_0},t'_{j_0+1}+M_1,\dots, t'_{j_0+j_1}+M_1,t'_{j_0+j_1+1}+M_1+M_2,\dots,\hspace{2cm}$$
$$\hspace{4cm} t'_{(\sum_{l=0}^{i-1} j_l)+1}+M_1+\dots+M_{i},\dots,t'_n+M_1+\dots+M_{i} )$$
where $t'_k$ is a real number close to $t_k$ and $M_j$'s are large real numbers. Here the normalization of the above element of $\rr^n$ is produced by subtracting a constant real number from all the entries such that the summation of the entries is equal to zero. A point in the interior of this face parametrizes a broken metric on $(\Wmn,\Zmn)$ with a cut along $(Y_{\bk_1},\zeta_{\bk_1})\cup \dots \cup (Y_{\bk_{i}},\zeta_{\bk_i)}$. In particular, a face of co-dimension 1 of $\Gmn$ has the form $G_{\bm}^{\bk}\times G_{\bk}^{\bn}$ where $\bk=\bI+\bn$, $\bI\in \{0,1\}^n$, and $\bk\neq \bm,\bn$. The restriction of the family of metrics to this face is $(\bWmk,\bZmk)\times (\bWkn,\bZkn)$.

\begin{lemma}\label{G-ori}
	The polyhedron $\Gmn$ can be oriented for all cubes $(\bm,\bn)$ such that for each co-dimension 1 face $G_\bm^\bk\times G_\bk^\bn$ of $\Gmn$ the 
	orientation given by the outward-normal-first convention differs from the product orientation by $(-1)^{\dim(G_\bm^\bk)\cdot\dim(G_\bk^\bn)+|\bn|_1}$.
	Moreover, in the case that $(\bm,\bn)$ is an edge, the orientation of the point $\Gmn$ is given by the following sign:
	\begin{equation}\label{eta}
		{\sgn}(\bm,\bn):=(-1)^{\sum_{i=1}^{i_0-1}m_i}
	\end{equation}	
where $i_0$ is the unique place that $\bm$ and $\bn$ differ.
\end{lemma}
\begin{proof}
	In \cite{KM:Kh-unknot}, the polyhedrons $\Gmn$ are oriented such that the boundary orientation on the codimension 1 face $G_{\bm}^\bk\times G_\bk^\bn$ of 
	$\Gmn$ differs from the product orientation by $(-1)^{|\bm-\bk|_1-1}$. As it is shown there, this orientation can be fixed such that the orientation of $\Gmn$ for 
	an edge is given by the following sign:
	\begin{equation*}
		(-1)^{\sum_{i=i_0}^{n}m_i}
	\end{equation*}	
	where $i_0$ is defined above.
	In order to produce the desired orientation, multiply this orientation of $\Gmn$ by 
	$(-1)^{{\rm sgn}'(\bm,\bn)}$, where:
	$${\rm sgn}'(\bm,\bn):=\frac{|\bm-\bn|_1(|\bm-\bn|_1-1)}{2}+|\bm|_1$$
\end{proof}

If $(\bm,\bn)$ is a $d$-dimensional double-cube, then a family of metrics on $(\Wmn,\Zmn)$, parametrized by a polyhedron $\Nmn$, can be constructed in a similar way. We explain the construction of this family in the special case that $\bm=(2,1,\dots,1)+\bn$. The general case can be treated similarly. The cobordism $\Wmn$ can be decomposed as $\rr \times Y^L \cup D \cup \bar H_2 \cup \dots \cup \bar H_n$, where $D$ is produced by gluing two 2-handles $H_1^1$, $H_1^2$ along $\partial_1 H_1^1=\partial_0 H_1^2$ and then adding the cylindrical ends $\rr^{\leq 0}\times \partial_0 H_1^1$ and $\rr^{\geq 0}\times \partial_1 H_1^2$. Note that $\partial D$ is the same as $\rr \times \widebar{\partial Y^{K_1}} $. An elementtary but important observation is that our choices of framings produce a 2-sphere $S$ with self-intersection -1 which is the union of the co-core disc of $H_1$ (i.e. $\{(z_1,z_2)\in H_1^1 \mid z_1 =0\}$) and the core of $H_2$ (i.e. $\{(z_1,z_2)\in H_1^2 \mid z_2 =0\}$). Consequently, a regular neighborhood of $S$ denoted by $R$ is diffeomorphic to $\widebar{\cc P}^2\backslash D^4$. We will write $T$ for the boundary of $R$. The space $D$ can be equipped with a family of metrics that is parametrized by $\rr$. For $\tau \in\rr^{\geq 0}$, insert the cylinder $[-\tau,\tau]\times \partial_1 H_1^1=[-\tau,\tau]\times \partial_0 H_1^2$ between $H_1^1$ and $H_1^2$ and then glue the cylindrical ends. For $\tau \in\rr^{\leq 0}$, insert the cylinder $[\tau,-\tau]\times T$ between $R$ and its complement. In summary, for positive values of $\tau$, we stretch along the solid torus $\partial_1 H_1^1=\partial_1 H_1^2$ and for negative values of $\tau$, we stretch along the boundary of $R$. For this family of metrics, pin down an isomorphism between the boundary of $D$ with $\rr \times \partial \widebar{Y^{K_1}}$ in the following way. For $\tau \in \rr^{\geq 0}$, let $[-\tau,\tau]\times \partial(\partial_1 H_1^1)$, inside the inserted cylinder, be identified with $[-\tau,\tau]\times \partial  \widebar{Y^{K_1}}$ and extend this identification in an obvious way. For $\tau \in \rr^{\leq 0}$, the corresponding metrics form a constant family on the boundary of $D$ and let the isomorphism be the same as the isomorphism for $\tau=0$. For a given $(t_1,\dots,t_n)\in \rr^n$, glue the handle $H_i$ to $\rr\times Y^{L}$  such that the gluing map is determined by $t_i$. Furthermore, glue the metric corresponding to $t_1$ on $D$ to $\rr \times Y^L$ with the above identification. As in the case of cubes, this family of metrics can be compactified to a family of (possibly broken) metrics $\bWmn$ parametrized by a polyhedron $\Nmn$. 

The 2-dimensional cobordism $\Zmn:\zm \to \zn$ can be decomposed as the union $\rr \times \zeta \cup Z\cup Z_2 \cup  \dots\cup Z_n$ where $\rr \times \zeta \subset \rr \times Y^L$, $Z \subset D$, $Z_2 \subset \bar H_2$, and so on. For $i \geq 2$, $Z_i \subset \bar H_i$ has the same form as in the case of cubes; it is either empty, or of the form (\ref{S-1}), or of the form (\ref{S-2}). There are three possibilities for $Z$ depending on $m_1$ mod 3. If $m_i \equiv 0$ mod 3, then $Z$ is a disc cobordism from $\zeta_{\bm,1}$ to the empty set. The disc intersects $S$ in exactly one point. Therefore, $Z$ intersects boundary of $R$ in a circle. On the other hand, the intersection with $\partial_1 H_1^1=\partial_0 H_1^2$ is empty. A similar discussion applies to the case that $m_i \equiv 2$ mod 3. If $m_i \equiv 1$ mod 3, then $Z=S$ and intersection of $Z$ and $\partial_1 H_1^1=\partial_0 H_1^2$ is a circle while $S \cap \partial R=\emptyset$. In each of these cases, extend the family of metrics on $D$ in the previous paragraph to a family of metrics for the pair $(D,Z)$ by inserting $[-\tau,\tau] \times (Z \cap \partial_1 H_1^1)$ (respectively, $[\tau,-\tau] \times (Z \cap \partial R)$) in the inserted $[-\tau,\tau] \times \partial_1 H_1^1$ (respectively, $[\tau,-\tau] \times \partial R$) for $\tau \geq 0$ (respectively, $\tau \leq 0$). Use this family of metrics to extend $\bWmn$ to a family of metrics $(\bWmn,\bZmn)$ on the pair $(\Wmn,\Zmn)$, parametrized by $\Nmn$. Codimension 1 faces of this polyhedron parametrized the following broken metrics: 
\begin{itemize}
	\item broken metrics along the cut $(Y_\bk,\zeta_\bk)$ where $\bm>\bk>\bn$. 
	This case can be divided to three different cases:
	\begin{itemize}
		\item[(a)] $\bk=(0,k_2,\dots k_n)$; this face can be identified with the family parametrized by 
		$N_{\bm}^{\bk} \times G_{\bk}^{\bn}$
		\item[(b)] $\bk=(1,k_2,\dots k_n)$; this face can be identified with the family parametrized by 
		$G_{\bm}^{\bk} \times G_{\bk}^{\bn}$				
		\item[(c)] $\bk=(2,k_2,\dots k_n)$; this face can be identified with the family parametrized by 
		$G_{\bm}^{\bk} \times N_{\bk}^{\bn}$				
	\end{itemize}
	\item broken metrics along the cut $(T,Z \cap T)$. This face will be denoted by $I_{\bm}^{\bn}$
\end{itemize}
The family of broken metrics in the second bullet is determined by families of metrics on the two connected components of $(\Wmn \backslash T,Z\backslash T)$. The family of metrics on the component $(R,Z \cap R)$ is either the same as the one used in Lemma \ref{special-fam} or that of Remark \ref{special-fam-2}.

Suppose an orientation of the polyhedrons $\Gmn$, satisfying the assumption of Lemma \ref{G-ori} is fixed. Then the following lemma sets our convention for the orientation of $\Nmn$:
\begin{lemma} \label{ori-N}
	The polyhedron $\Nmn$ can be oriented for all double-cubes $(\bm,\bn)$ such that for each co-dimension 1 face $A_\bm^\bk\times B_\bk^\bn$ of $\Nmn$,
	in which $A,B \in \{G,N\}$, the 
	orientation given by the outward-normal-first convention differs from the product orientation by $(-1)^{\dim(A_\bm^\bk)\cdot\dim(B_\bk^\bn)+|\bn|_1}$.
\end{lemma}
\begin{proof}
	As in Lemma \ref{G-ori}, we can change the orientation of $\Nmn$, defined in \cite{KM:Kh-unknot}, to produce the desired orientation of the polyhedron $\Nmn$.
\end{proof}

For a $d$-dimensional triple-cube $(\bm,\bn)$, a $(d+1)$-dimensional family of metrics parametrized by a polyhedron $\Kmn$ is defined in an analogous way. For the simplicity of the exposition, we explain the construction in the special case that $\bm=(3,1,\dots,1)+\bn$. There is a decomposition $\rr \times Y^L \cup F \cup \bar H_2 \cup \dots \cup \bar H_n$ of the cobordism $\Wmn$ with $F$ being the result of composing $\rr^{\leq 0} \times \partial_0 H_1^1$, the handles $H_1^1$, $H_1^2$, $H_1^3$, and finally $\rr^{\geq 0} \times \partial_1 H_1^3$. There is also a decomposition of $\Zmn$ as the union $\rr \times \zeta \cup B \cup Z_2 \cup \dots \cup Z_n$ where $\rr \times \zeta \subset \rr \times Y^L$, $B \subset F$, $Z_2 \subset \bar H_2$, and so on. 

The co-core of the handle $H_1^i$ and the core of $H_1^{i+1}$ produces a 2-sphere $S_i$ for $i=1,2$ ($S_1$ is what we called $S$ before). Regular neighborhoods of $S_1$, $S_2$, denoted by $R_1$, $R_2$, are diffeomorphic to $\widebar{\cc P}^2\backslash D^4$.  We will also write $T_1$, $T_2$ for the boundaries of $R_1$, $R_2$. The spheres $S_1$ and $S_2$ intersect each other in exactly one point, and hence a neighborhood of their union, denoted by $U$, is diffeomorphic to $D^2\times S^2 \# \overline{\cc P^2}$. The boundary of $U$ is also denoted by $Q$. We assume that $U$ contains $R_1$ and $R_2$. In the list:
$$ E_1=\partial_1 H_1^2=\partial_0 H_1^3\hspace{1cm} E_2=\partial_1 H_1^1=\partial_0 H_1^2 \hspace{1cm}  E_3=T_2  \hspace{1cm} E_4=Q \hspace{1cm}  E_5=T_1 $$
each 3-manifold $E_i$ is disjoint from $E_{i-1}$ and $E_{i+1}$ where $E_0=E_5$ and $E_6=E_1$. In fact removing $E_i$ disconnects $F$, and $E_{i+1}$, $E_{i-1}$ are both in the same connected component of $F\backslash E_i$. For each $i$, construct a family of metrics on $F$ that are broken along $E_i$, and are parametrized by $[-\infty,\infty]$, in the following way. Firstly remove $E_i$ and then for $\tau \in [0,\infty]$ (respectively, $\tau \in [-\infty,0]$) remove $E_{i+1}$ (respectively, $E_{i-1}$) and insert $[-\tau,\tau]\times E_{i+1}$ (respectively, $[\tau,-\tau]\times E_{i-1}$). The metrics corresponding to $\infty$ and $-\infty$ are broken along $E_{i+1}$ and $E_{i-1}$, respectively. Furthermore, this family of metrics is constant on the connected component of $F\backslash E_{i}$ that does not contain $E_{i-1}$ and $E_{i+1}$. Example \ref{example-family-metrics-1} determines one such family for $i=4$. As in the case of double-cubes, these families can be lifted to families of metrics on the pair $(F,B)$.

These five families of metrics consist of broken metrics parametrized by the edges of a pentagon, denoted by $\mathfrak p$. Extend these metrics into the interior of $\mathfrak p$ in an arbitrary way. This extended family on $(F,B)$, and the metrics on $(\rr \times Y^L, \rr \times \zeta)$, $(\bar H_i, Z_i)$ generate a family of metrics on $(\Wmn,\Zmn)$ parametrized by $\rr^{n-1} \times \mathfrak p$. This family can be compactified by adding broken metrics. The resulting family, $(\bWmn,\bZmn)$, is parametrized by a polyhedron $\Kmn$. Codimension 1 faces of this polyhedron parametrize the following families of metrics:

\begin{itemize}
	\item broken metrics along the cut $(Y_\bk,\zeta_\bk)$ where $\bm>\bk>\bn$. This case can be divided to four different sub-cases:
	\begin{itemize}
		\item[(a)] $\bk=(0,k_2,\dots k_n)$; this face can be identified with the family parametrized by 
		$K_{\bm}^{\bk} \times G_{\bk}^{\bn}$
		\item[(b)] $\bk=(1,k_2,\dots k_n)$; this face can be identified with the family parametrized by 
		$N_{\bm}^{\bk} \times G_{\bk}^{\bn}$				
		\item[(c)] $\bk=(2,k_2,\dots k_n)$; this face can be identified with the family parametrized by 
		$G_{\bm}^{\bk} \times N_{\bk}^{\bn}$
		\item[(d)] $\bk=(3,k_2,\dots k_n)$; this face can be identified with the family parametrized by 
		$G_{\bm}^{\bk} \times K_{\bk}^{\bn}$				
	\end{itemize}
	\item broken metrics along the cut $(Q,B\cap Q)$. This face will be denoted by $Q_{\bm}^{\bn}$
	\item broken metrics along the cut $(T_1,B \cap T_1)$. This face will be denoted by $(I_1)_{\bm}^{\bn}$
	\item broken metrics along the cut $(T_2, B \cap T_2)$. This face will be denoted by $(I_2)_{\bm}^{\bn}$
\end{itemize}

Orientation of the polyhedrons $\Kmn$ can be fixed such that these orientations are compatible with the orientations of $\Gmn$ and $\Nmn$ from Lemmas \ref{G-ori} and \ref{ori-N} in the following way:

\begin{lemma}\label{ori-K}
	The polyhedron $\Kmn$ can be oriented for all triple-cubes $(\bm,\bn)$ such that for each co-dimension 1 face 
	$A_\bm^\bk\times B_\bk^\bn$ of $\Kmn$,
	in which $A,B \in \{G,N,K\}$, the 
	orientation given by the outward-normal-first convention differs from the product orientation by 
	$(-1)^{\dim(A_\bm^\bk)\cdot\dim(B_\bk^\bn)+|\bn|_1}$.
\end{lemma}

The families broken along $T_1$ and $T_2$ are similar to the family parametrized by $I_{\bm}^\bn$ in the case of double-cubes. The family of metrics broken along $Q$ is determined by  an $(|\bm-\bn|_1-3)$-dimensional family of metrics on $(\Wmn \backslash U,\Zmn\backslash U)$, denoted by $\tQmn$ and a family on $(U,B\cap U)$ parametrized by a closed interval $\tilde I$. The former family is the compactification of a family parametrized by $\rr^{n-1}$ such that a typical element in the family is given by gluing $\rr \times Y^L$, with the product metric, to the translations of $H_i$'s, with the standard metrics, and $F \backslash U$, with a fixed metric. The family of metrics on $(U,B\cap U)$ is one of the families introduced in Example \ref{example-family-metrics-1} and Remark \ref{fam-met2}. Note that an orientation of $\tQmn$ can be fixed such that the product orientation on $ \tQmn \times \tilde I$ agrees with the boundary orientation, induced by the orientation of $\Kmn$ from Lemma \ref{ori-K}. 

Let $V$ be a copy of $S^1 \times D^3$ and $C$ is empty or the surface introduced in Remark \ref{fam-met2}, depending on whether $(U,B\cap U)$ is the pair of Example \ref{example-family-metrics-1} or Remark \ref{fam-met2}. Glue a copy of $(V,C)$ to $(\Wmn \backslash U,\Zmn \backslash U)$, along the boundary component $Q$, in order to form a pair $(\cWmn,\cZmn)$. Note that Gluing $(V,C)$ to $(F\backslash U,B \backslash U)\subset (\Wmn \backslash U,\Zmn \backslash U)$ gives the product pair $(\rr \times \partial_0 H_1^1,\rr \times (\partial_0 H_1^1\cap \zeta_\bm))$. Therefore, $(\cWmn,\cZmn)$ is diffeomorphic to $(W_{\bm'}^{\bn},Z_{\bm'}^{\bn})$ where $\bm'=(m_1-3,m_2,\dots,m_n)$. Fix homology orientations of $U$ and $V$ such that the assertion of Lemma \ref{fam-met2} holds. Then these homology orientations and the homology orientation of $\Wmn$ can be utilized to define a homology orientation for $\cWmn$ by requiring:
$$o(\cWmn)=o(\cWmn\backslash V,V)\hspace{1cm} o(\Wmn)=o(\Wmn\backslash U,U)\hspace{1cm}o(\cWmn\backslash V)=o(\Wmn\backslash U)$$

\begin{lemma} \label{hom-ori-checkWmn}
	Suppose $(\bm,\bn)$ is a triple-cube and $\bm \geq \bk,\bk' \geq \bn$ such that  $(\bm,\bk)$ and $(\bk',\bn)$ are triple-cubes, i.e., $(\bk,\bn)$ and $(\bm,\bk')$ 
	are cubes.
	Then the composition $\widecheck W_\bm^\bk \circ W_\bk^\bn$ and $W_\bm^{\bk'} \circ \widecheck W_{\bk'}^\bn$ are diffeomorphic to $\cWmn$. 
	Furthermore, the homology orientations of these cobordisms are related via the following relations:
	\begin{equation} \label{hom-cWmn}
		o(\cWmn)=(-1)^{\ind(\mathcal D_{W_\bk^\bn})} o(\widecheck W_\bm^\bk,W_\bk^\bn)=o(W_\bm^\bk,\widecheck W_\bk^\bn)
	\end{equation}
\end{lemma}

\begin{proof}
	The claim about the diffeomorphisms follows from the fact that $\Wmn$ and the compositions $W_\bm^\bk\circ W_\bk^\bn$ and 
	$W_\bm^{\bk'}\circ W_{\bk'}^\bn$ are diffeomorphic to each other. The relationship between the homology orientations of $\Wmn$ and 
	$W_\bm^\bk\circ W_\bk^\bn$ can be verified as follows:
	\begin{align}
		o(\cWmn,U,U)&=o(\cWmn\backslash V,V,U,U)=o(\Wmn\backslash U,U,V,U)\nonumber\\
		&=o(\Wmn,V,U)=o(W_\bm^\bk,W_\bk^\bn,V,U)\nonumber\\
		&=o(W_\bm^\bk\backslash U,U,W_\bk^\bn,V,U)\nonumber\\
		&=(-1)^{\ind(\mathcal D_{W_\bk^\bn})}
		o(\widecheck W_\bm^\bk\backslash V,V,W_\bk^\bn,U,U)\nonumber \\
		&=(-1)^{\ind(\mathcal D_{W_\bk^\bn})}
		o(\widecheck W_\bm^\bk,W_\bk^\bn,U,U)\implies\nonumber \\
		o(\cWmn)&=(-1)^{\ind(\mathcal D_{W_\bk^\bn})}o(\widecheck W_\bm^\bk,W_\bk^\bn)\nonumber
	\end{align}	
	Here we use $\ind (\mathcal D_U)=0$ and $\ind (\mathcal D_V)=1$. The other identity in (\ref{hom-cWmn}) can be checked similarly. 
\end{proof}

The family of metrics on $(\Wmn \backslash U, \Zmn \backslash U)$, parametrized by $\tQmn$, and a fixed metric with a cylindrical end on $(V,C)$ determine a family of broken metrics on $(\cWmn,\cZmn)$ parametrized by $\tQmn$. However, we will write $\widebar Q_{\bm}^{\bn}$ for the parametrizing set of this family in order to distinguish between the families.  Assume that again $\bm=\bn+(3,1,\dots,1)$. Then an element of this family is a metric on: 
$$(\cWmn=\rr \times Y^L \cup \rr \times \partial_0 H_1^1 \cup \bar H_2 \cup \dots \cup \bar H_n,\cZmn=\rr\times \zeta \cup \rr\times (\partial_0 H_1^1\cap \zeta_\bm)\cup Z_2 \cup \dots \cup Z_n)$$ 
where the metric on $(\rr \times \partial_0 H_1^1,\rr\times (\partial_0 H_1^1\cap \zeta_\bm))$ is broken along $Q$. Replacing this broken metric with the standard product metric constructs another family of metrics on $(\cWmn,\cZmn)$, parametrized with the set $\widebar Q_{\bm}^{\bn}$. To emphasize the difference, we will denote the parametrizing set by $\widehat Q_{\bm}^{\bn}$. By construction, it is clear that one can interpolate between the families parametrized by $\widebar Q_{\bm}^{\bn}$ and $\widehat Q_{\bm}^{\bn}$; that is to say, there exists a family parametrized by $I \times \widebar Q_{\bm}^{\bn}$ such that $\{0\}\times \widebar Q_{\bm}^{\bn}$ agrees with $\widehat Q_{\bm}^{\bn}$ and $\{1\}\times \widebar Q_{\bm}^{\bn}$ is equal to $\widebar Q_{\bm}^{\bn}$. The orientation of $\tQmn$ determines an orientation of $\bQmn$ which can be used to equip $I \times \bQmn$ with the product orientation. Also, the codimension 1 faces of $I \times \bQmn$ are given as follows:

\begin{itemize}
	\item broken metrics along the cut $(Y_\bk,\zeta_\bk)$ where $\bm>\bk>\bn$ and $\bk=(0,k_2,\dots k_n)$; 
	this face can be identified with  $(I \times\widebar Q_\bm^\bk) \times G_{\bk}^{\bn}$
	\item broken metrics along the cut $(Y_\bk,\zeta_\bk)$ where $\bm>\bk>\bn$ and $\bk=(3,k_2,\dots k_n)$; 
	this face can be identified with $G_{\bm}^{\bk}\times (I \times \widebar Q_{\bk}^{\bn}) $
	\item the family of metrics $\widebar Q_\bm^\bn$
	\item the family of metrics $\widehat Q_{\bm}^{\bn}$
\end{itemize}

\begin{lemma} \label{Q-ori}
	For any triple-cube $(\bm,\bn)$, the fixed orientation of $I \times \widebar Q_{\bm}^{\bn}$ produces the following orientation of its faces:
	\begin{itemize}
		\item the boundary orientation on $(I \times\widebar Q_\bm^\bk) \times G_{\bk}^{\bn}$ differs from the product orientation by 
		$$(-1)^{\dim(K_\bm^\bk)\cdot\dim(G_\bk^\bn)+|\bn|_1+1+|\bk-\bn|_1}=(-1)^{\dim(I \times\widebar Q_\bm^\bk)\cdot\dim(G_\bk^\bn)+|\bn|_1}$$
		\item the boundary orientation on $G_{\bm}^{\bk}\times (I \times \widebar Q_{\bk}^{\bn})$ differs from the product orientation by 
		$$(-1)^{\dim(G_\bm^\bk)\cdot(\dim(K_\bk^\bn)+1)+|\bn|_1+1+|\bm-\bk|_1}=(-1)^{\dim(G_\bm^\bk)\cdot\dim(I \times\widebar Q_\bm^\bk)+|\bn|_1+1+|\bm-\bk|_1}$$
		\item the boundary orientation on $\widebar Q_\bm^\bn=\{1\}\times \widebar Q_\bm^\bn$ is the same as the fixed orientation of $\widebar Q_\bm^\bn$;
		\item the boundary orientation on $\widehat Q_{\bm}^{\bn}=\{0\}\times \widebar Q_\bm^\bn$ is the reverse of the fixed orientation of $\widebar Q_\bm^\bn$.
	\end{itemize}
\end{lemma}
\begin{proof}
	It suffices to prove the same claim when $I \times \bQmn$ in the statement of the lemma is replaced with $\tilde I \times \tQmn$. 
	The polyhedrons $\tQmn\times \tilde I $ and $K_\bm^\bk\times G_\bk^\bn$, with $\bk=(0,k_2,\dots k_n)$, are codimension 1 faces of $\Kmn$. 
	Intersection of these two polyhedrons is $(\widetilde Q_\bm^\bk \times \tilde I ) \times G_\bk^\bn$ which
	has codimension 1 with respect to $\tQmn	 \times \tilde I$ and $K_\bm^\bk\times G_\bk^\bn$. 
	The polyhedrons $\tQmn \times \tilde I$ and $K_\bm^\bk\times G_\bk^\bn$ inherit boundary orientations form $K_\bm^\bk$ which in turn induce two orientations on
	$(\widetilde Q_\bm^\bk \times \tilde I) \times G_\bk^\bn$. These two orientations differ by a sign.
	Another way to orient $(\widetilde Q_\bm^\bk \times \tilde I)\times G_\bk^\bn$ is to consider it as the product of $\widetilde Q_\bm^\bk \times\tilde I$ and $G_\bk^\bn$. 
	This orientation is the same as the boundary orientation by considering $(\widetilde Q_\bm^\bk \times \tilde I)\times G_\bk^\bn$ as a codimension 1 face of 
	$K_\bm^\bk\times G_\bk^\bn$ where the latter is equipped with the product orientation. 
	Therefore, Lemma \ref{ori-N} states that the product and the boundary orientations of $(\widetilde Q_\bm^\bk \times \tilde I) \times G_\bk^\bn$, 
	as a face of $\tQmn \times \tilde I$ differ by $(-1)^{\dim(K_\bm^\bk)\cdot\dim(G_\bk^\bn)+|\bn|_1+1}$. 
	This proves the claim about the orientation of 
	the faces in the first bullet. A similar argument verifies the claim for the faces in the second bullet. 
	The assertions about the orientations of the faces in the last two bullets are trivial.
\end{proof}

\subsection{Exact Triangles} \label{triangles}

Suppose a pair $(Y,\zeta)$ and a framed link $L$ is fixed as in the previous subsection. Then we can form the pairs $(Y_\bm,\zm)$ and the cobordisms of the pairs $(\Wmn,\Zmn)$. Use Lemma \ref{fix-hom-ori} to fix homology orientations for the cobordisms $\Wmn$. For each representation variety $\mathcal R(Y_{\bm})$, fix a Morse-Smale function $f_\bm$ such that $f_{\bm+3\bk}=f_{\bm}$ for any $\bk\in \zz^n$. For a cube, double-cube, or a triple-cube $(\bm,\bn)$, we constructed a family of metrics $(\bWmn,\bZmn)$ on the pair $(\Wmn,\Zmn)$. Since the parametrizing sets for these families are contractible, the homology orientation for $\Wmn$ determines an orientation of $\ind(\bWmn, \Gmn)$, $\ind(\bWmn, \Nmn)$ or $\ind(\bWmn, \Kmn)$ depending on whether $(\bm,\bn)$ is a cube, double-cube, or a triple-cube.  

Fix a perturbation $\emn$ for the family $(\bWmn,\bZmn)$ such that Hypothesis \ref{eta-hypo-1} holds. Thus the map $\feWmn$ is well-defined. To simplify our notation, we denote this cobordism map by $\fGmn$, $\fNmn$ or $\fKmn$ when $(\bm,\bn)$ is a cube, double-cube, or triple-cube, respectively. 
We also require that the induced perturbation for the families $I_{\bm}^\bn$, $(I_1)_{\bm}^\bn$, $(I_2)_{\bm}^\bn$ and $Q_{\bm}^{\bn}$, are equal to zero on $R$, $R_1$, $R_2$ and $U$, respectively. We also assume that $\emn$ is equal to $\eta_{\bm+3\bk}^{\bn+3\bk}$ for any $\bk \in \zz^n$.

Given a cube $(\bm,\bn)$, an application of (\ref{pair-signed-bdry-form-loc-sys}) with the aid of Lemmas \ref{fix-hom-ori} and \ref{G-ori} to the family of metrics parametrized by $\Gmn$ gives:

\begin{equation*}
	\sum_{\bm>\bk>\bn} (-1)^{|\bn|_1}\fGkn\circ \fGmk=
	d_\bn\circ \fGmn +(-1)^{|\bm-\bn|_1} \fGmn \circ d_{\bm}
\end{equation*}

where $d_{\bk}$ is the differential on the complex  $\tPFC(Y_\bk,\zeta_{\bk})$. If we set the convention that $\tilde f_{G_\bk^\bk}$ is equal to $(-1)^{|\bk|_1+1}d_{\bk}$, then the above identity can be rewritten as:
\begin{equation} \label{cube-diff}
	\sum_{\bm\geq \bk\geq \bn} \fGkn\circ \fGmk=0
\end{equation}
This relation states that the following is a chain complex:
\begin{equation*}
	\left(\Cmn := \bigoplus_{\bm\geq \bk\geq \bn}\tPFC(Y_\bk,{\zeta_\bk}) ,\dmn:= (\tilde f_{G_{\bf k}^{\bf k'}})_{\bm\geq \bk\geq \bk'\geq \bn}\right)
\end{equation*}

For $\bm=(m_1,\dots, m_{n-1}) \in \zz^{n-1}$ and $i\in \zz$, let ${\bf m}(i):=(i,m_1,\dots,m_{n-1}) \in \zz^{n}$ be given as the juxtaposition of $i$ and $\bm$. Given a cube $(\bm, \bn)$ in $\zz^{n-1}$ and $i\in \zz$, define $\gmni:\Cmni \to \Cmnim$ by the following matrix presentation:
$$\gmni:= (\tilde f_{G_{{\bf k}(i)}^{{\bf k'}(i-1)}})_{\bm\geq \bk\geq \bk'\geq \bn}$$
By (\ref{cube-diff}), $\gmni$ is an anti-chain map, namely, we have the following identity:
$$\dmnim\circ \gmni+\gmni\circ \dmni=0$$
In general, given an anti-chain map $g:(C,d) \to (C',d')$ we can form the mapping cone chain complex $\Cone(g)$ defined in he following way:
$$\Cone(g):=(C\oplus C'\hspace{2pt},\hspace{2pt}
\left( \begin{array}{cc}
	d&0\\
	g&d'
\end{array}
\right)
)$$
For the anti-chain map $\gmni$, the chain complex $\Cone(\gmni)$ is clearly the same as $(C_{{\bm}(i)}^{{\bn}(i-1)},d_{{\bm}(i)}^{{\bn}(i-1)})$. The heart of this paper lies in the following theorem:
\begin{theorem} \label{PFH-tri-1}
	The chain complexes $(\Cmni,\dmni)$ and $\Cone(\gmnim)$ have the same chain homotopy type. 
\end{theorem}	
An immediate corollary of Theorem \ref{PFH-tri-1} is that the chain complex $(\Cmn,\dmn)$ is a new chain complex for the plane Floer homology of an appropriate pair:
\begin{corollary} \label{new-complex-1}
	Let $(\bm,\bn)$ be a cube and $\bp=2\bm-\bn$. Then the chain complexes $(\tPFC(Y_{\bp},\zeta_{\bp}),d_\bp)$ and $(\Cmn ,\dmn)$ have the same 
	chain homotopy type.
\end{corollary}
\begin{proof}
	For the simplicity of the exposition, assume that ${\bp}=(2,\dots,2)$, $\bm=(1,\dots,1)$, $\bn=(0,\dots,0)$.
	 By Theorem \ref{PFH-tri-1}, $(\Cmn,\dmn)$ and $(C_{\bm'}^{\bn'},d_{\bm'}^{\bn'})$ have the same chain homotopy type when
	 $\bm'=(2,1,\dots,1)$, $\bn'=(2,0,\dots,0)$. 
	 By repeating this application of Theorem \ref{PFH-tri-1}, we can show that $(\tPFC(Y_\bp,\zeta_{\bp}),d_{\bp})$ and $(\Cmn,\dmn)$ have 
	 the same chain homotopy type.
\end{proof}

Consider the grading on $(\Cmn,\dmn)$ that is defined to be equal to $|{\bf k}|_1$ on the summand $\tPFC(Y_{\bk},\zeta_{\bk})$ of $\Cmn$. The differential $\dmn$ does not increase the grading. Therefore, we can use this filtration to produce a spectral sequence: 

\begin{corollary} \label{surgery-spectral-sequence}
	There is a spectral sequence with the first page $\bigoplus_{\bk\in\{0,1\}^n}\tPFH(Y_{\bf k},\zeta_{\bk})$ that abuts to $\tPFH(Y,\zeta)$. 
	The differential in the first page can be identified with $\bigoplus (-1)^{{\rm sgn}(\bk,\bk')}\tPFH(W_{\bf k}^{\bf k'},Z_{\bf k}^{\bf k'})$. 
	Here the sum is over all the edges of the cube 
	$\{0,1\}^n$ and ${\rm sgn}(\bk,\bk')$ is defined in (\ref{eta}). 
\end{corollary}
\begin{proof}
	The definition of the filtration implies the first part. The second part is the result of Lemma  \ref{G-ori}.
\end{proof}	

The following algebraic lemma underlies our strategy of proving Theorem \ref{PFH-tri-1} (cf. \cite{OzSz:HF-branch, KM:Kh-unknot}):
\begin{lemma} \label{trianlge-detection-lemma}
	For $i\in \zz$, let $(C_i,d_i)$ be a chain complex and $g_i:C_i \to C_{i-1}$ be an anti-chain map.
	For each $i$, suppose there are $n_i:C_i \to C_{i-2}$ and $k_i:C_i \to C_{i-3}$ such that:
	\begin{equation} \label{null-htpy}
		g_{i-1}\circ g_i + d_{i-2}\circ n_{i}+n_{i}\circ d_{i}=0
	\end{equation}
	and
	\begin{equation} \label{q-iso}
		n_{i-1}\circ g_i+g_{i-2}\circ n_i+d_{i-3}\circ k_{i}+k_{i}\circ d_{i}=q_i
	\end{equation}
	where $q_i:C_i \to C_{i-3}$ is an isomorphism. Then $C_i$ and $\Cone(g_{i-1})$ have the same chain homotopy type. 
\end{lemma}

\begin{proof}
	The following maps:
	\begin{equation} \label{chain-map-iso}
		\phi_i:=(g_i,n_i):C_i \to \Cone(g_{i-1})\hspace{2cm} \psi_i:=\left(
		\begin{array}{c}
			n_{i-1}\\
			g_{i-2}
		\end{array}\right):\Cone(g_{i-1}) \to C_{i-3}
	\end{equation}	
	are anti-chain maps. We also have:
	\begin{equation}\label{htpy-equiv-1}
		\psi_i\circ \phi_i \sim_h q_i\hspace{2cm}\phi_{i}\circ \psi_{i+3} \sim_h\left(
		\begin{array}{cc} 
			q_{i+2}&0\\
			n_{i}\circ n_{i+2}+k_{i+1}\circ g_{i+2}+g_{i-1} \circ k_{i+2}&q_{i+1}
		\end{array}
		\right) 
	\end{equation}
	where the homotopies are given by the following maps:
	\begin{equation}\label{htpies}
		k_i \hspace{2cm} \left(
		\begin{array}{cc}
			k_{i+2}&n_{i+1}\\
			0&k_{i+1}
		\end{array}
		\right)
	\end{equation}
	The homotopies in (\ref{htpy-equiv-1}) show that $\psi_i\circ \phi_i $ and $\phi_{i}\circ \psi_{i+3}$ are homotopic to isomorphisms. Consequently, 
	$\phi_i:C_i \to \Cone(g_{i-1})$ is a homotopy equivalence.
\end{proof}

\begin{proof}[Proof of Theorem \ref{PFH-tri-1}] 
	In order to utilize Lemma \ref{trianlge-detection-lemma}, we need to construct maps $\nmni: \Cmni \to \Cmnimm$ and $\kmni: \Cmni \to \Cmnimmm$
	such that the constructed maps satisfy the relations (\ref{null-htpy}) and (\ref{q-iso}). 
	For each $i\in \zz$, the pair $(\bm(i),\bn(i-2))$ determines a double-cube in $\zz^{n}$, and hence we can consider the 
	family of metrics parametrized by $N_{\bm(i)}^{\bn(i-2)}$ on the pair cobordism $(W_{\bm(i)}^{\bn(i-2)},Z_{\bm(i)}^{\bn(i-2)})$. 
	This family defines the map: 
	$$f_{N_{\bm(i)}^{\bn(i-2)}}:\tPFC(Y_{\bm(i)},\zeta_{\bm(i)}) \to \tPFC(Y_{\bn(i-2)},\zeta_{\bn(i-2)})$$ 
	Identity (\ref{pair-signed-bdry-form-loc-sys}), Lemma \ref{ori-N} and the characterization of codimension 1 faces of $N_{\bm(i)}^{\bn(i-2)}$ imply that:
	\begin{equation*}
		\tilde f_{I_{\bm(i)}^{\bn(i-2)}}+\sum_{\bm\geq \bk > \bn} (-1)^{|\bn|_1+i} \tilde f_{G^{\bn(i-2)}_{\bk(i-2)}} \circ \tilde f_{N^{\bk(i-2)}_{\bm(i)}}+
		\sum_{\bm\geq \bk \geq \bn} (-1)^{|\bn|_1+i}\tilde f_{G^{\bn(i-2)}_{\bk(i-1)}} \circ \tilde f_{G^{\bk(i-1)}_{\bm(i)}}+\hspace{1.5cm}
	\end{equation*}
	\begin{equation*}
		\hspace{1.5cm}+\sum_{\bm> \bk\geq \bn} (-1)^{|\bn|_1+i}\tilde f_{N^{\bn(i-2)}_{\bk(i)}} \circ \tilde f_{G^{\bk(i)}_{\bm(i)}}
		=d_{\bn(i-2)} \circ \tilde f_{N_{\bm(i)}^{\bn(i-2)}}-(-1)^{|\bm-\bn|_1+1}\tilde f^{N_{\bm(i)}^{\bn(i-2)}}_\eta\circ d_{\bm(i)}
	\end{equation*}
	According to Lemma \ref{special-fam}, the cobordism map $\tilde f_{I_{\bm(i)}^{\bn(i-2)}}$ vanishes. 
	We can use this and the convention that $\tilde f_{G_\bk^\bk}=(-1)^{|\bk|_1+1}d_{\bk}$ to simplify
	slightly:
		\begin{equation} \label{nmn}
		\sum_{\bm\geq \bk \geq \bn}  \tilde f_{G^{\bn(i-2)}_{\bk(i-2)}} \circ \tilde f_{N^{\bk(i-2)}_{\bm(i)}}+
		\sum_{\bm\geq \bk \geq \bn} \tilde f_{G^{\bn(i-2)}_{\bk(i-1)}} \circ \tilde f_{G^{\bk(i-1)}_{\bm(i)}}
		+\sum_{\bm\geq \bk\geq \bn} \tilde f_{N^{\bn(i-2)}_{\bk(i)}} \circ \tilde f_{G^{\bk(i)}_{\bm(i)}}
		=0
	\end{equation}
	If we define a map $\nmni: \Cmni \to \Cmnimm$ as:
	$$\nmni:= (\tilde f_{N_{\bf k(i)}^{\bf k'(i-2)}})_{\bm\geq \bk\geq \bk'\geq \bn}$$
	then \eqref{nmn} implies:
	\begin{equation} \label{null-htpy}
		\gmnim\circ \gmni + \dmnimm\circ \nmni+\nmni \circ \dmni=0
	\end{equation}	
	
	We can repeat the above construction for the family of metrics parametrized by $K_{\bm(i)}^{\bn(i-3)}$ and define the map:
	$$\tilde f_{K_{\bm(i)}^{\bn(i-3)}}:\tPFC(Y_{\bm(i)},\zeta_{\bm(i)}) \to \tPFC(Y_{\bn(i-3)},\zeta_{\bn(i-3)})$$ 	
	The description of the boundary of $K_{\bm(i)}^{\bn(i-3)}$ provides us with the following relation:
	\begin{equation*}
		f_{(I_1)_{\bm(i)}^{\bn(i-3)}} + f_{(I_2)_{\bm(i)}^{\bn(i-3)}}+f_{ Q_{\bm(i)}^{\bn(i-3)}}+
		(-1)^{|\bn|_1+i-3} (\sum_{\bm\geq \bk > \bn} f_{G^{\bn(i-3)}_{\bk(i-3)}} \circ f_{K^{\bk(i-3)}_{\bm(i)}}
		+\sum_{\bm\geq \bk \geq \bn}f_{G^{\bn(i-3)}_{\bk(i-2)}} \circ f_{N^{\bk(i-2)}_{\bm(i)}}+
	\end{equation*}
	\begin{equation*}
		+\sum_{\bm\geq \bk \geq \bn} f_{N^{\bn(i-3)}_{\bk(i-1)}} \circ f_{G^{\bk(i-1)}_{\bm(i)}}
		+\sum_{\bm> \bk \geq \bn} f_{K^{\bn(i-3)}_{\bk(i)}} \circ f_{G^{\bk(i)}_{\bm(i)}})
		=d_{\bn(i-3)} \circ f_{K_{\bm(i)}^{\bn(i-3)}}-(-1)^{|\bm-\bn|_1+2}f_{K_{\bm(i)}^{\bn(i-3)}}\circ d_{\bm(i)}
	\end{equation*}
	Again, the first two terms in the left hand side are zero by Lemma \ref{special-fam}. 
	Therefore, the above relation can be rewritten as follows:
	\begin{equation*}
		\sum_{\bm\geq \bk \geq \bn} 
		f_{G^{\bn(i-3)}_{\bk(i-3)}} \circ f_{K^{\bk(i-3)}_{\bm(i)}}
		+\sum_{\bm\geq \bk \geq \bn} f_{G^{\bn(i-3)}_{\bk(i-2)}} \circ f_{N^{\bk(i-2)}_{\bm(i)}}
		+\sum_{\bm\geq \bk \geq \bn} f_{N^{\bn(i-3)}_{\bk(i-1)}} \circ f_{G^{\bk(i-1)}_{\bm(i)}}
		+\sum_{\bm\geq \bk \geq \bn} f_{K^{\bn(i-3)}_{\bk(i)}} \circ f_{G^{\bk(i)}_{\bm(i)}}
	\end{equation*}
	\begin{equation} \label{null-hom}		
		\hspace{4cm}=(-1)^{|\bn|_1+i} f_{ Q_{\bm(i)}^{\bn(i-3)}}
	\end{equation}
	Define $\hkmni, \hqmni:\Cmni \to \Cmnimmm$ as follows:
	$$\hkmni:= (\tilde f_{K_{\bf k(i)}^{\bf k'(i-3)}})_{\bm\geq \bk\geq \bk'\geq \bn}\hspace{1cm}
	\hqmni:= ((-1)^{|\bk'|_1+i}\tilde f_{Q_{\bf k(i)}^{\bf k'(i-3)}})_{\bm\geq \bk\geq \bk'\geq \bn}$$
	then (\ref{null-hom}) gives rise to the following identity:
	$$\dmnimmm\circ \hkmni+\gmnidm\circ \nmni+\nmnim\circ \gmni+\hkmni\circ \dmni =\hqmni$$
	As the final step of the proof, we show that $\hqmni$ is homotopy equivalent to an isomorphism from $\Cmni$ to $\Cmnimmm$. Using Lemma \ref{quasi-iso} and 
	Remark \ref{bdry-terms}, it can be shown:
	\begin{equation} \label{rel-hat-ckeck}
		\tilde f_{Q_{\bm(i)}^{\bn(i-3)}}=(-1)^{|\bm-\bn|_1}c \cdot s \circ \tilde f_{\widebar Q_{\bm(i)}^{\bn(i-3)}}
	\end{equation}	
	where $\tilde f_{\widebar Q_{\bm(i)}^{\bn(i-3)}}$ is the map for the pair $(\widecheck W_{\bm(i)}^{\bn(i-3)},\widecheck Z_{\bm(i)}^{\bn(i-3)})$.
	Furthermore, the family of metrics $I \times \widebar Q_{\bm(i)}^{\bn(i-3)}$ defines a cobordism map and with the aid of (\ref{pair-signed-bdry-form-loc-sys}) 
	and Lemmas \ref{hom-ori-checkWmn} and \ref{Q-ori}, we can derive the following identity:
	\begin{equation*}
		\sum_{\bm\geq \bk\geq \bn} (-1)^{|\bn|_1+i+1+\ind(\mathcal D(W_{\bk(i-3)}^{\bn(i-3)})}f_{G_{\bk(i-3)}^{\bn(i-3)}}\circ f_{I \times \widebar Q_{\bm(i)}^{\bk(i-3)}}+
		(-1)^{|\bn|_1+|\bm-\bk|_1+i} f_{I \times  \widebar Q_{\bk(i)}^{\bn(i-3)}}\circ f_{G_{\bm(i)}^{\bk(i)}}
	\end{equation*}	
	\begin{equation*}
		=f_{\widehat Q_{\bm(i)}^{\bn(i-3)}}-f_{\widebar Q_{\bm(i)}^{\bn(i-3)}}
	\end{equation*}	
	Post-composing with $s$ and using (\ref{com-s}) let us to rewrite above identity as:
		\begin{equation*}
		\sum_{\bm\geq \bk\geq \bn} f_{G_{\bk(i-3)}^{\bn(i-3)}}\circ((-1)^{|\bm-\bk|_1} c \cdot s \circ f_{I \times \widebar Q_{\bm(i)}^{\bk(i-3)}})+
		((-1)^{|\bk-\bn|_1} c \cdot s\circ f_{I \times  \widebar Q_{\bk(i)}^{\bn(i-3)}})\circ f_{G_{\bm(i)}^{\bk(i)}}=
	\end{equation*}
	\begin{equation}\label{2-chain-htpy}
		\hspace{4cm}=(-1)^{|\bm|_1+i}c \cdot s\circ f_{\widehat Q_{\bm(i)}^{\bn(i-3)}}-(-1)^{|\bm|_1+i}c \cdot s\circ f_{\widebar Q_{\bm(i)}^{\bn(i-3)}}
	\end{equation}
	Define $\tkmni:\Cmni \to \Cmnimmm$ and $\qmni:\Cmni \to \Cmnimmm$ as follows:
	$$\tkmni:= ((-1)^{|\bk-\bk'|_1} c \cdot s \circ f_{I \times \widebar Q_{\bk(i)}^{\bk'(i-3)}})_{\bm\geq \bk\geq \bk'\geq \bn}$$
	$$\qmni:= ((-1)^{|\bk|_1+i}c \cdot s\circ f_{\widehat Q_{\bk(i)}^{\bk'(i-3)}})_{\bm\geq \bk\geq \bk'\geq \bn}$$
	Identity (\ref{2-chain-htpy}) gives rise to the following chain homotopy:
	$$\dmnimmm\circ \tkmni+\tkmni\circ \dmni=\qmni-\hqmni$$
	Finally, with $\kmni=\hkmni+\tkmni$ we have:
	$$\dmnimmm\circ \kmni+\gmnidm\circ \nmni+\nmnim\circ \gmni+\kmni\circ \dmni =\qmni$$
	which is the analogue of (\ref{q-iso}), if we can show that $\qmni$ is an isomorphism.
	If $\bm\neq \bn$, the family of metrics parametrized by $\widehat Q_{\bm(i)}^{\bn(i-3)}$ admits a faithful $\rr$-action. Therefore, in this case we can arrange the involved
	perturbation terms in the definition of $\widehat Q_{\bm(i)}^{\bn(i-3)}$ such that the dimension of the non-empty moduli spaces is at least one. 
	Thus if $\bm\neq\bn$, then with these perturbation terms $\tilde f_{\widehat Q_{\bm(i)}^{\bn(i-3)}}=0$. If $\bm=\bn$, the 
	family of metrics $\widehat Q_{\bm(i)}^{\bn(i-3)}$ consists of the product metric on 
	$(\widecheck W_{\bm(i)}^{\bn(i-3)},\widecheck Z_{\bm(i)}^{\bn(i-3)})\cong (I\times Y_{\bm},I\times \zeta_{\bm})$. However,
	the fixed homology orientation of this cobordism and the orientation of the point $\widehat Q_{\bm(i)}^{\bm(i-3)}$ might not be the trivial ones.
	Consequetly, $\tilde f_{\widehat Q_{\bm(i)}^{\bn(i-3)}}$ is equal to $\pm c\cdot s$.
	Therefore, $\qmni$ is indeed an isomorphism of the complexes $\Cmni$ and $\Cmnimmm$.
\end{proof}

\section{Classical Links} \label{classical-links}
A classical link is a link embedded in $S^3$ (or equivalently $\rr^3$). Let $K$ be an oriented classical link. Then we can form $\Sigma(K)$, the branched double cover of $S^3$ branched along $K$. Clearly, $\tPFH(\Sigma(K))$ is an invariant of $K$ that is functorial in the following sense. For classical links $K_0$ and $K_1$, let $S:K_0 \to K_1$ be a (not necessarily orientable) link cobordism embedded in $I \times S^3$ such that $S \cap (\{i\}\times S^3)=K_i$ for $i\in \{0,1\}$. Then $\Sigma(S)$, the branched double cover of $I \times S^3$ branched along $S$, is a cobordism from $\Sigma(K_0)$ to $\Sigma(K_1)$. Therefore, we can consider the cobordism map $\tPFH(\Sigma(S)):\tPFH(\Sigma(K_0)) \to \tPFH(\Sigma(K_1))$. In subsection \ref{PFH-dbl-cover}, we will explore the basic properties of this functor defined on the category of classical links and cobordisms between them. In particular, we shall explain how the exact triangles in subsection \ref{triangles} give rise to a generalization of unoriented skein exact triangle for this link invariant. We shall use these exact triangles in subsection \ref{PKH} to define {\it plane knot homology}, a stronger version of the  plane Floer homology of the branched double covers.

\subsection{Plane Floer Homology of Branched Double Covers} \label{PFH-dbl-cover}

Suppose $S:K_0 \to K_1$ is a cobordism of classical links, $\zeta_0\subset \Sigma(K_0)$ and $\zeta_1\subset \Sigma(K_1)$ are embedded 1-manifolds, and $Z:\zeta_0 \to \zeta_1$ is an embedded cobordism in $W:=\Sigma(S)$. Given a family of metrics $(\bW,\bZ)$ on $(W,Z)$, consider the cobordism map $\tfeWZ$. The following lemma shows how the degree of this map can be computed in terms of the topological invariants of $S$, $K_0$, and $K_1$:

\begin{lemma} \label{top-invts-bra-dble-1}
	The degree of the map $\tfeWZ: \tPFC(\Sigma(K_0),\zeta_0) \to \tPFC(\Sigma(K_1),\zeta_1)$ with respect to $\deg_p$ is given by:
		\begin{equation}\label{deg-form}
			\deg_p(\tfeWZ)=\chi(S)+\frac{1}{2}S\cdot S-(\sigma(K_1)-\sigma(K_0))+2\dim(G)
		\end{equation}
	In particular, the degree of the map $\tPFH(\Sigma(S),Z)$ is equal to:
	$$\chi(S)+\frac{1}{2}S.S-(\sigma(K_1)-\sigma(K_0)).$$
	Here $\sigma(K_i)$ denotes the signature of the link $K_i$ and the self-intersection $S \cdot S$ is computed with respect to the Seifert framing on $K_0$ and $K_1$.
\end{lemma}
Note that the self-intersection $S \cdot S$ is defined even when $S$ is not orientable. To that end, firstly fix a Seifert surface for the link $K_i$ and move this link on this surface to produce $\widetilde K_i$, a parallel copy of $K_i$. Then perturb $S$ to a cobordism $\widetilde S: \widetilde K_0 \to \widetilde K_1$ that is transversal to $S$. The self-intersection number $S \cdot S$ is given by the signed count of the intersection points of $S$ and $\widetilde S$. To determine the sign of each intersection point, fix an arbitrary orientation of $S$ in a neighborhood of the intersection point. This induces a local orientation of $\widetilde S$ and the sign of the intersection point is determined by comparing the sum of these local orientations on $S$, $\widetilde S$ and the orientation of the ambient space $I \times S^3$.
\begin{proof}
	For the cobordism $S:K_0\to K_1$, the Euler characteristic and the signature of  the 4-dimensional cobordism $\Sigma(S)$ are given by the following identities:
	$$\chi(\Sigma(S))=-\chi(S) \hspace{2cm} \sigma(\Sigma(S))=-\frac{1}{2}S.S+\sigma(K_1)-\sigma(K_0)$$
	The identity for the Euler characteristic is obvious. To obtain the other identity, let $S_i$ be a Seifert surface for $K_i$ and $\widetilde S_i$ is given by pushing $S_i$ 
	into a 4-dimensional ball that fills the sphere containing $K_i$. Then $\sigma(K)$ is equal to the signature of $\Sigma(\widetilde S_i)$, the branched double cover of
	the 4-ball, branched along $\widetilde S_i$. Gluing $\widetilde S_0$, $S$, and $\widetilde S_1$ (with the reversed orientation) determines a closed surface 
	$\widebar S$ in $S^4$. The $G$-signature theorem implies that:
	\begin{align}
		\sigma(\Sigma(\widebar S))&=\sigma(S^4)-\frac{1}{2}S\cdot S\nonumber \\
		&=-\frac{1}{2}S\cdot S\nonumber
	\end{align}
	 This identity and the additivity of the signature of 4-manifolds gives the desired identity for $\sigma(\Sigma(S)$. Using these topological formulas, the equality 
	 (\ref{ori-morphism-deg}) asserts that:
	$$\deg_p(\tfeWZ)=2\dim(G)-\sigma(\Sigma(S))-\chi(\Sigma(S)))=\chi(S)+\frac{1}{2}S.S-(\sigma(K_1)-\sigma(K_0))+2\dim(G)$$
\end{proof}

Following \cite{KM:q-grading}, a crossing $c$ for a link $K\subset S^3$ is an orientation preserving embedding $c:D^3 \to S^3$ such that $c(T)=c(D^3) \cap K$ where the 1-manifold $T\subset D^3$ is demonstrated in Figure \ref{crossing}. Let $T$ be oriented as in Figure \ref{crossing}. A crossing $c$ is {\it positive}, if the the restriction of $c$ to $T$ on both components is either orientation-preserving or orientation-reversing. If $c$ is orientation preserving on exactly one connected component, then $c$ is a {\it negative} crossing. Figure \ref{resolutions} shows how to {\it resolve} a crossing in order to produce the 0{\it -resolution} and the 1{\it -resolution}. The orientation of $K$ on the complement of the crossing $c$ can be extended to exactly one of these resolutions. This resolution is called the {\it oriented} resolution of $c$. If $c$ is a positive crossing, then the 0-resolution is the oriented resolution. Otherwise, the 1-resolution is the oriented one.

\begin{figure}
	\centering
	\begin{subfigure}[b]{0.3 \textwidth}
		\includegraphics[width=\textwidth]{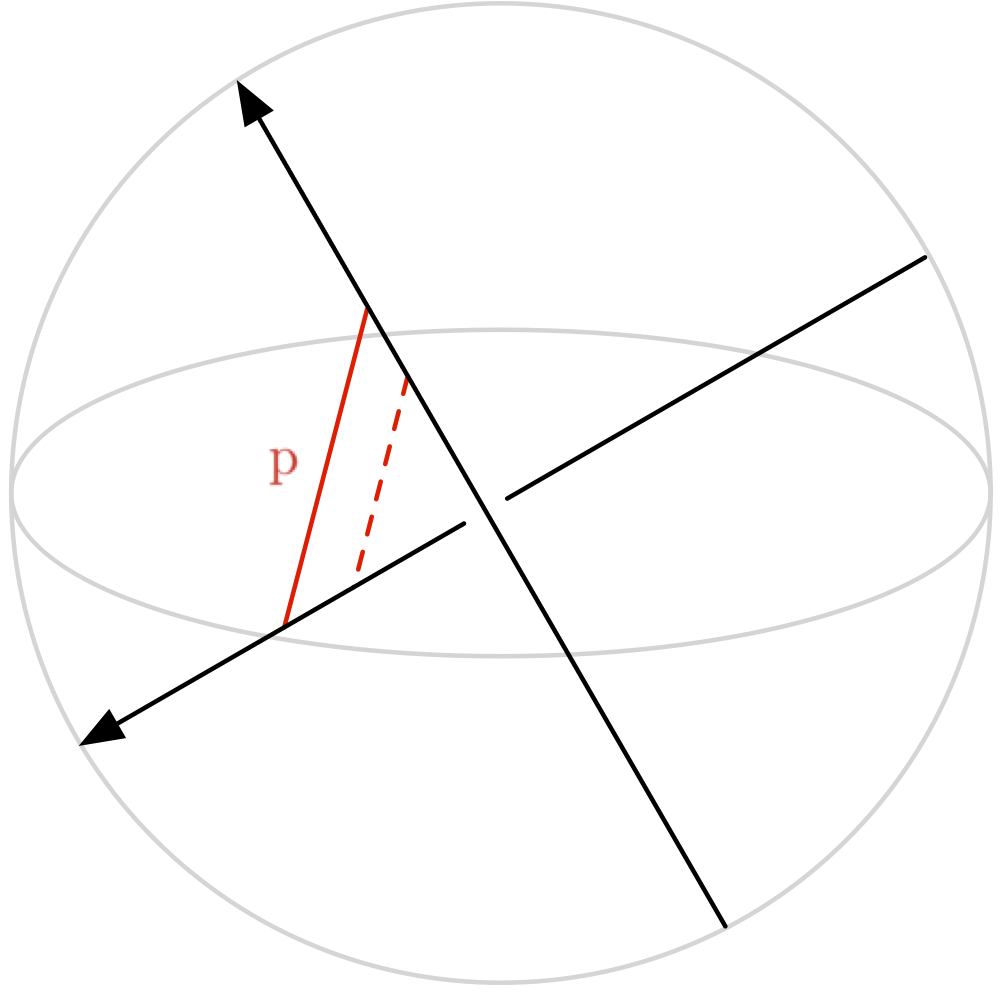}
		\caption{A crossing of a link is an embedded ball in $S^3$ that 
		its intersection with the link is equal to the union of the two black arcs.}
		\label{crossing}
	\end{subfigure}
	\hspace{7pt}
	\begin{subfigure}[b]{0.65 \textwidth}
		\includegraphics[width=\textwidth]{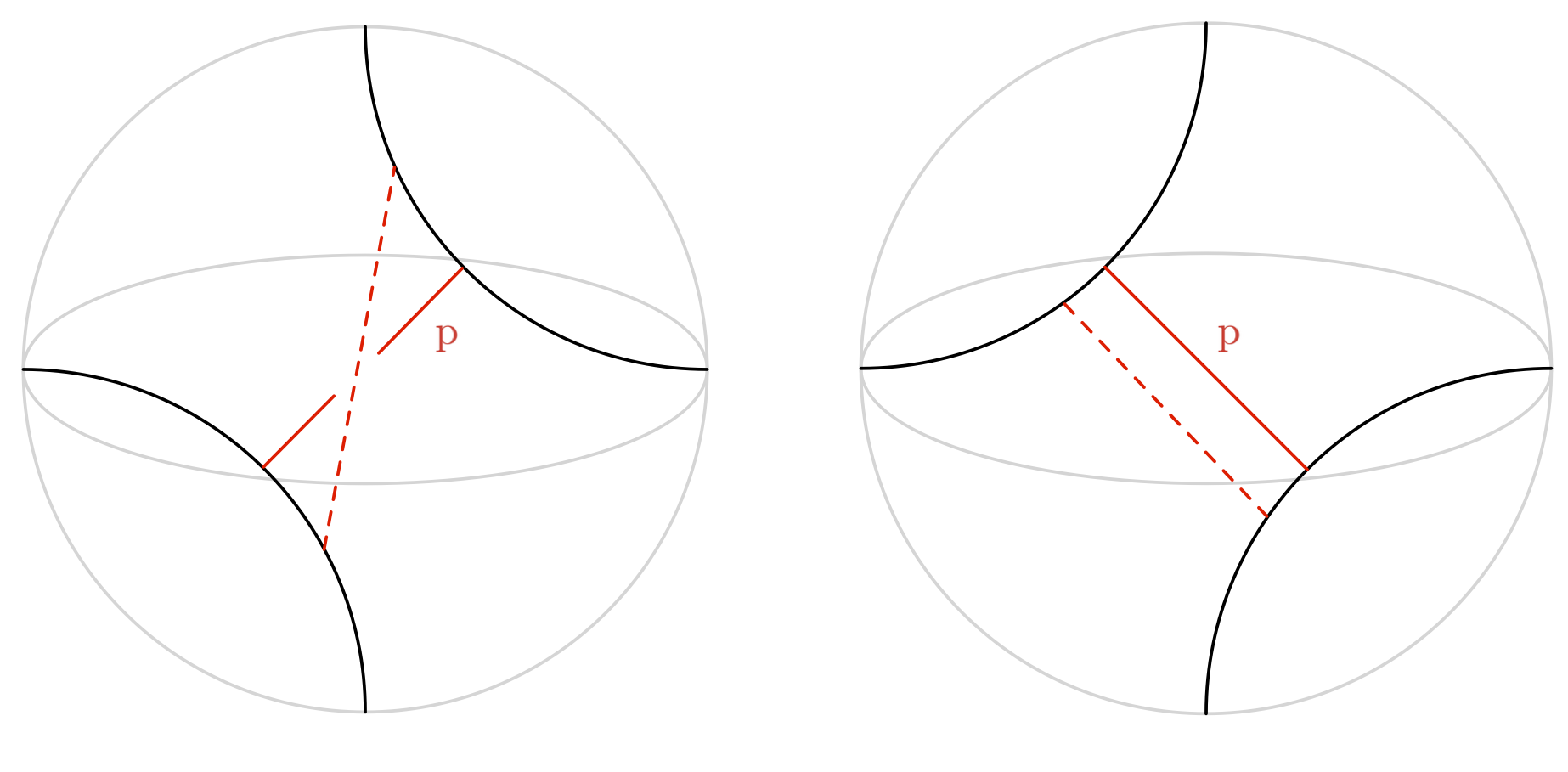}
		\caption{0-resolution (1-resolution) of the crossing in Figure \ref{crossing} is the link given by replacing the crossing with the left (right)
		picture.\\ \\}
		\label{resolutions}
	\end{subfigure}
	\caption{A crossing and its resolutions}
\end{figure}

For a crossing $c$, the branched double cover of $c(D^3)$ branched along $c(T)$ is a solid torus which is denoted by $\Sigma(c(T))$. Let $l_p$ be the lift of the path $p$ in Figure \ref{crossing}. Then $l_p$ is the core of $\Sigma(K)$. The lift of the dashed path provides a parallel copy of $l_p$ which in turn determines a framing for $l_p$. In fact, 0- and 1-surgery on $\sk$ along the framed knot $l_p$ produce 3-manifolds that are respectively branched double covers of 0- and 1-resolutions of $K$. The 2-resolution of a crossing is given by leaving the crossing unchanged. For a general $i \in \zz$, the $i$-resolution, denoted by $K_{(i)}$, is the same as the $j$-resolution where $1\leq j\leq 2$ and $i\equiv j$$\mod3$.

In the previous section, we define the cobordism $W_{(i)}^{(j)}:\Sigma(K)_{(i)} \to \Sigma(K)_{(j)} $ when $i,j \in \zz$ and $i\geq j$. This cobordism is equal to $\Sigma(S_{(i)}^{(j)})$ where $S_{(i)}^{(j)}:K_{(i)}\to K_{(j)}$ can be described as follows. The cobordism $S_{(i)}^{(i-1)}$ is the result of gluing a 1-handle along the endpoints of the path $p$ in $K_{(i)}$ with respect to the framing given by the endpoints of the dashed red path. In particular, the Euler characteristic of $S_{(i)}^{(i-1)}$ is equal to -1. The non-trivial connected component of  $S_{(i)}^{(i-1)}$, as a topological surface, is either a pair of pants or a twice punctured $\rr P^2$. The latter case happens if and only if $K_{(i)}$ and $K_{(i-1)}$ have the same number of connected components. In general, the cobordism $S_{(i)}^{(j)}$ is equal to the composition $S_{(i)}^{(i-1)} \circ S_{(i-1)}^{(i-2)}\circ \dots \circ S_{(j+1)}^{(j)}$. Let also ${\bar S}_{(i)}^{(j)}:K_{(i)} \to K_{(j)}$ be the reverse of the cobordism $S_{(i)}^{(j)}$. Then $S_{(i)}^{(i-2)}$, as a cobordism embedded in $I \times S^3$ is diffeomorphic to the connected sum of ${\bar S}_{(i+1)}^{(i)}$ and $S_0$, the standard embedding of $\rr P^2$ in $S^4$ with self intersection 2 (cf. \cite[Lemma 7.2]{KM:Kh-unknot}). The cobordism $S_{(i)}^{(i-3)}$ has a similar description. Suppose $\widecheck S_{(i)}^{(i-3)}:K_{(i)} \to K_{(i-3)}$ is the product cobordism. Then there is a 4-dimensional ball in $I \times S^3$ (respectively, $S^4$) that intersects $\widecheck S_{(i)}^{(i-3)}$ (respectively, $S_0$) in two 2-dimensional discs. Removing these 4-balls from $I \times S^3$ and $S^4$ produces two pairs of 4-manifolds and embedded surfaces such that each of them has a boundary component consists of the pair of $S^3$ and an embedded link with two connected components which are  unlinked and unknotted. Gluing these two pairs along the common boundary determines a 2-dimensional cobordism in $I \times S^3$ which is equal to $S_{(i)}^{(i-3)}$ (\cite[Lemma 7.4]{KM:Kh-unknot}).

Starting with a crossing for an oriented link $K$, the oriented resolution inherits an orientation form $K$. However, there is not a canonical way to fix an orientation for the other resolution. Therefore, in general there might be an ambiguity in the definition of the self-intersection numbers of the cobordism $S_{(i)}^{(j)}$. There is an important case for our purposes  that we can avoid this ambiguity. A link $L$ is an {\it unlink} with $n$ connected components if there is an embedding of a union of $n$ 2-discs in $S^3$ that fills the link $L$. Different orientations of an unlink produce the same Seifert framing. Therefore, if both of the resolutions $K_{(i)}$ and $K_{(j)}$ are unlink, then the self-intersection of the cobordism $S_{(i)}^{(j)}$ is well-defined. For example, if $K_{(1)}$ and $K_{(0)}$ are unlink, then possible self-intersections of the cobordism $S_{(1)}^{(0)}$ for this case can be described as follows (cf. \cite[Lemma 4.7]{KM:q-grading}). If the non-trivial component of this cobordism is a pair of pants, then the self-intersection is zero. Otherwise, this self-intersection is equal to 2 or -2. Note that if the crossing is positive (respectively, negative), then the cobordism $S_{(0)}^{(-1)}$ (resp $S_{(2)}^{(1)}$) has zero self-intersection. Another constraint is that the sum of the self-intersections of the cobordisms $S_{(2)}^{(1)}$, $S_{(1)}^{(0)}$, and $S_{(0)}^{(-1)}$ is equal to 2 because of the special form of the cobordism $S_{(2)}^{(-1)}$, described in the previous paragraph. As a consequence, if $K$, $K_{(0)}$ and $K_{(1)}$ are all unlink, then out of these three cobordisms, the self-intersection of one of them is 2 and the remaining two cobordisms have zero self-intersection.

Next, let $N$ be a set of $n$ crossings for $K$. Given $\bm=(m_1,\dots,m_n)\in \zz^n$, the link $K_\bm$ is the result of replacing the $i^{\rm th}$ crossing with the $m_i$-resolution. As before, consider the path $p_i$ connecting the two strands in the $i^{\rm th}$ crossing. Inverse image of this path in $\Sigma(K)$ is a framed knot and the union of these knots for the crossings in $N$ determines a framed link in $\Sigma(K)$. It is still true that $\Sigma(K_\bm)$ is the same as the 3-manifold $\Sigma(K)_\bm$, given by surgery on $\Sigma(K)$ along this framed link and corresponding to $\bm \in \zz^n$. For a pair $(\bm,\bn)$ with $\bm\geq \bn$, the cobordism $\Wmn:\Sigma(K)_\bm \to \Sigma(K)_\bn$ is the branched double cover of a cobordism $\Smn: K_\bm \to K_\bn$. In the case $|\bm-\bn|_1=1$, the topological properties of this cobordism is explained above. In general, $\Smn$ can be decomposed as the composition of $|\bm-\bn|_1$ such simpler cobordisms. In particular, the Euler characteristic of $\Smn$ is equal to $-|\bm-\bn|_1$. If $(\bm,\bn)$ is a triple-cube, then $\cWmn$ is also branched double cover of a cobordism that is denoted by $\cSmn$. For example, if $\bm=\bm_0(2)$ and $\bn=\bn_0(-1)$, then $\cSmn$ is the composition of the product cobordism from $K_{\bm_0(2)}$ to $K_{\bm_0(-1)}$ and the cobordism $S_{\bm_0(-1)}^{\bn_0(-1)}$.

In the case that $K_\bm$ and $K_\bn$ are both unlink, the self-intersection of the cobordism $\Smn$ is a well-defined even integer  number. As in \cite{KM:q-grading}, we will write $\sigma(\bm,\bn)$ for this self-intersection number and relax the condition $\bm\geq \bn$ by requiring that:
\begin{equation}\label{additivity}
	\sigma(\bm,\bn)=\sigma(\bm,\bk)+\sigma(\bk,\bn)
\end{equation}	
when $\bm,\bk,\bn \in \zz^n$ and $K_\bm$, $K_\bk$, and $K_\bn$ are unlink. As a consequence of (\ref{additivity}), we have:
\begin{align}\label{3-per}
	\sigma(\bm+3\bk,\bn)&=\sigma(\bm+3\bk,\bm)+\sigma(\bm,\bn)\nonumber\\
	&=\sigma(\bm,\bn)+2|\bk|_1
\end{align}

Since the set of crossings $N$ determines a framed link with $n$ connected components, we have the following corollary of Theorem \ref{PFH-tri-1}:
\begin{theorem} \label{crossing-set-complex}
	For the set of crossings $N$, there exists the object $(\bigoplus_{{\bf m}\in \{0,1\}^n}\tPFC(\Sigma(K_{{\bf m}}),\zeta_\bm),\widetilde d_p)$
	 of the category $\mathfrak C(\tilde \Lambda)$
	that has the same chain homotopy type as $\tPFC(\Sigma(K))$. 
	The differential $\widetilde d_p$ can be decomposed as $\bigoplus_{(\bm,\bn) \in \{0,1\}^n} d_{\bf m}^{\bf n}$ where 
	$d_{\bf m}^{\bf n}:\tPFC(\Sigma(K_{\bf m}),\zeta_\bm) \to \tPFC(\Sigma(K_{\bf n}),\zeta_\bn)$ is a map induced by a family of metrics of dimension 
	$|{\bf m}-{\bf n}|_1-1$ on 
	$\Sigma(S_{\bf m}^{\bf n})$. This map is nonzero only when $\bm\geq \bn$.
\end{theorem}
Recall that to define the maps $\dmn$, we need to fix an appropriate metric on $\Sigma(K)$, Morse-Smale functions on $\mathcal R(\Sigma(K_\bm))$, perturbations of the ASD equation and homology orientations of the cobordism $\Sigma(\Smn)$. The set of all these data together is called {\it auxiliary choices} to form the chain complex $(\bigoplus_{{\bf m}\in \{0,1\}^n}\tPFC(\Sigma(K_{{\bf m}}),\zeta_\bm),\widetilde d_p)$.

\subsection{Plane Knot Homology} \label{PKH}
A {\it filtered chain complex} $(C,d)$ is a chain complex with a filtration of the module $C$:
$$C\supseteq \dots \supseteq \mathcal F_{i-1} C\supseteq \mathcal F_i C \supseteq \mathcal F_{i+1} C \supseteq \dots \supseteq \{0\}$$
such that $d$ maps each {\it filtration level} $ \mathcal F_i C$ to itself. Suppose there exists $r \in \rr$ such that each of the chain complexes $(\mathcal F_i C,d)$ is graded with a set of the form $\{r\}+\zz$ and the differential $d$ increases this degree by 1. If these gradings are compatible with respect to the filtration, then we say that $(C,d)$ is a {\it filtered $\zz$-graded chain complex}. For our purposes here, we can assume that $r$ is equal to $0$ or $\frac{1}{2}$. If $(C,d)$ and $(C',d')$ are two filtered chain complexes, then a map $f:(C,d) \to (C',d')$ of degree $(a,b) \in \zz \times \rr$ is a module homomorphism of degree $b$ that sends $\mathcal F_i C$ to $\mathcal F_{i+a} C'$. Note that a map of degree $(a,b)$ has degree $(a',b)$ for any $a'\leq a$. 

Let $\cgf$ be the category whose objects are filtered $\zz$-graded chain complexes. A morphism in this category is represented by either a chain map or an anti-chain map of degree $(0,0)$. Two such maps represent the same morphism if there is a chain homotopy between them that is of degree $(-1,-1)$. The isomorphism class of an objet of $\cgf$ is called its {\it filtered chain homotopy type}. Clearly, there is also a forgetful functor from $\cgf$ to $\mc$.

The goal of this subsection is to assign an object of the category $\cgfl$ to a link $K$, equipped with appropriate additional data. We shall show that the the filtered chain homotopy type of this chain complex is independent of the additional data and hence is a link invariant. In order to define this filtered chain complex, we need to work with a variant of the notion of pseudo-diagrams introduced in \cite{KM:q-grading}:

\begin{definition}
	A {\it pseudo-diagram} for a link $K$ is a set of crossings $N$ such that $K_{\bm}$ is an unlink for each $\bm \in \{0,1\}^{|N|}$. 
\end{definition}

A planar diagram $D$ for a link $K$ determines a set of crossings $N_D$ which is clearly a pseudo-diagram. This justifies the terminology of \cite{KM:q-grading} for pseudo-diagrams. 

\begin{definition}\label{pl-pseudo-diagram}
	A {\it planar pseudo-diagram} $N$ is a pseudo-diagram that is given by a planar diagram $D$ in one of the following ways:
	\vspace{-10 pt}
	\begin{itemize}
		\item $N$ is equal to $N_D$, the set of crossings of $D$;
		\item $N$ is given by removing an element of $N_D$;
		\item $N$ is given by removing two consecutive under-crossings or two consecutive over-crossings in $N_D$.
	\end{itemize}
\end{definition}

If adding a crossing to the pseudo-diagram $N$ produces another pseudo-diagram $N'$, then we say $N'$ is related to $N$ by adding a crossing. Alternatively, $N$ is related to $N'$ by dropping a crossing. We also say that the pseudo-diagrams $N$ and $N'$ of the links $K$ and $K'$ are related by a pseudo-diagram isotopy, if there is an isotopy $\varphi_t: S^3 \to S^3$ such that $\varphi_0=id$ and $\varphi_1$ maps $K$ (respectively, $N$) to $K'$ (respectively, $N'$). Note that $K$ and $K'$ represents the same classical link and hence $N$ and $N'$ can be regarded as two pseudo-diagrams for a classical link. The following observation is stated in \cite{KM:q-grading}:

\begin{lemma} \label{moves}
	Any two planar pseudo-diagrams can be related to each other by a series of pseudo-diagram isotopies, adding and dropping crossings.
\end{lemma}

\begin{remark}
	In the upcoming definition of planar knot homology, part of the additional data is the choice of a planar pseudo-diagram. 
	However, the construction can be applied for an arbitrary pseudo-diagram. It is not clear to the author whether the choice of planar pseudo-diagrams is essential to 
	the definition of plane knot homology or an arbitrary pseudo-diagram determines a filtered chain complex with the same filtered chain homotopy type. 
	A purely topological question that probably should be addressed first is the following. Is there any pseudo-diagram that cannot be related to a planar
	diagrams by adding or dropping crossings?
\end{remark}

Fix a pseudo-diagram $N$ with $n$ crossings for a link $K$. The set of crossings of $N$ can be used to form the resolutions $K_{\bm}$ and the cobordisms $\Smn$. The following topological lemma studies the self-intersection of the cobordisms $\Smn$ in the case of pseudo-diagrams:

\begin{lemma} \label{self-int-cob}
	Suppose $(\bm,\bn)$ is a cube in $\{0,1\}^n$ and $\Smn$ is the cobordism associated with a pseudo-diagram $N$.
	Then the self-intersection $\Smn\cdot\Smn$ is not greater than $2|\bm-\bn|_1$. If the equality holds, then for $\bm \geq \bk \geq \bn$ the number of connected 
	components of $K_\bk$ is independent of $\bk$.~
\end{lemma}

\begin{proof}
	Firstly let $|\bm-\bn|_1=1$. As it is shown in \cite[Lemma 4.7]{KM:q-grading}, the set of possible values for the self-intersection of $\Smn$ is equal to $\{-2,0,2\}$.
	If this value is equal to 2, then $\Smn$ is a a union of a product cobordism and a twice punctured $\rr P^2$ (as a cobordism from the unknot to itself).  In particular,
	$K_\bm$ and $K_\bn$ have the same number of the connected components. This special case can be used to verify the general case.
\end{proof}

The pseudo-diagram $N$ can be used to define an integer-valued function on the cube $\{0,1\}^n$:
$$h:\{0,1\}^n \to \zz$$
$$h(\bm):=-|{\bm}|_1+\frac{1}{2}\sigma({\bm},{\bo})$$
where $\bo$ is the vertex of the cube $\{0,1\}^n$ such that $K_{\bo}$ is the oriented resolution of $K$. Extend this function in the obvious way to the elements $\bm \in \zz^n$ such that $K_\bm$ is an unlink. The following corollary is an immediate consequence of Lemma \ref{self-int-cob}:
\begin{corollary} \label{h-filtered}
	Let $\bm,\bn \in \{0,1\}^n$ and $\bm\geq \bn$. Then $h(\bn)\geq h(\bm)$.
\end{corollary}

Suppose that auxiliary choices for $N$ is fixed such that we can appeal to Theorem \ref{crossing-set-complex} to construct the chain complex:
$$\PKC(N):=\oplus_{\bm \in \{0,1\}^n}\PFC(\Sigma(K_\bm))$$
with the differential $d_p$. Here we use $\Lambda$ as the coefficient ring. In particular, there is no need to work with local coefficient systems. Furthermore, we require that the Morse-Smale function on $f_\bm$ is chosen such that the differential on the complex $\tPFC(\Sigma(K_\bm))$ vanishes. This is obviously possible because the representation variety $\mathcal R(\Sigma(K_\bm))$ is a torus. The chain complex $(\PKC(N),d_p)$ has the same chain homotopy type as $\PFC(\Sigma(K))$. The map $h$ can be used to define a grading on the chain complex $\PKC(N)$ in the following way:
$$\deg_h|_{\PFC(\Sigma(K_\bm))}=h(\bm)+n_-(N).$$ 
where $n_-(N)$ denotes the number of negative crossings. We will also write $n_+(N)$ for the number of positive crossings of $N$.
The grading $\deg_h$ is called the {\it homological grading}. Corollary \ref{h-filtered} states that $\deg_h$ defines a filtration on $\PKC(N)$ in which the $i^{\rm th}$ level of the filtration is generated by those elements that $\deg_h\geq i$. We also define another grading on $(\PKC(N),d_p)$ in the following way:
$$\delta|_{\PFC(\Sigma(K_{\bm}))}:=-\frac{1}{2}\deg_p-\frac{1}{2}|{\bm}|_1-\frac{1}{4}\sigma({\bm},{\bo})+\frac{1}{2}n_+(N)$$
This grading is called the {\it $\delta$-grading}. Since $\sigma(\bm,\bo)$ is always an even integer, $\deg_\delta$ is either an integer or a half-integer.  In fact, the same argument as in \cite[Lemma 4.4]{KM:q-grading} shows that parity of $2\deg_\delta$ is equal to $b_0(K)-1$. In particular, if $K$ is a knot, then the $\delta$-grading is always an integer number. 

The definition of the homological grading and the $\delta$-grading can be extended to $\PFC(\Sigma(K_{\bm}))$ when $K_\bm$ is an unlink. Recall that if $\bm' \equiv\bm$ mod 3, then $\PFC(\Sigma(K_{\bm'}))=\PFC(\Sigma(K_{\bm}))$. The following lemma shows how the gradings on these two vector spaces compare to each other:
\begin{lemma}\label{period-deg}
	Suppose $\bm'=\bm+3\bk$ for $\bm,\bk\in \zz^n$. Then:
	$$\deg_h|_{\PFC(\Sigma(K_{\bm'}))}=\deg_{h}|_{\PFC(\Sigma(K_{\bm}))}-2\bk$$
	$$\deg_\delta|_{\PFC(\Sigma(K_{\bm'}))}=\deg_{\delta}|_{\PFC(\Sigma(K_{\bm}))}-2\bk$$
\end{lemma}

\begin{proof}
	These claims can be easily verified using (\ref{3-per}).
\end{proof}

\begin{lemma} \label{delta-deg-family}
	For $\bm\geq \bn$, suppose that $K_\bm$ and $K_\bn$ are unlink. Let $f:\PFC(\Sigma(K_\bm)) \to \PFC(\Sigma(K_\bn))$ be a map induced by a family of metrics on
	$\Sigma(\Smn)$ (respectively, $\Sigma(\cSmn)$). If $d$ denotes the dimension of the parametrizing set for the family of metrics, 
	then the degree of $f$ with respect to the $\delta$-grading is equal to $|\bm-\bn|_1-d$ (respectively, $|\bm-\bn|_1-d-1$).
\end{lemma}
\begin{proof}
	Let $f$ be induced by a family of metrics on $\Sigma(\Smn)$. According to Lemma \ref{top-invts-bra-dble-1}:
	$$\deg_p(f)=-|\bm-\bn|_1+\frac{1}{2}\Smn\cdot\Smn+2d$$
	As a result:
	$$\deg_\delta(f)=-\frac{1}{2}(-|\bm-\bn|_1+\frac{1}{2}\Smn\cdot\Smn+2d)-\frac{1}{2}(|\bn|_1-|\bm|_1)
		-\frac{1}{4}(\sigma(\bn,\bo)-\sigma(\bm,\bo))=|\bm-\bn|_1-d$$
	A similar computation verifies the case that $f$ is induced by a family of metrics on $\Sigma(\cSmn)$.
\end{proof}

The differential of the chain complex $(\PKC(N),d_p)$ is constructed via the cobordism maps associated with the family of metrics $\Gmn$ of dimension $|\bm-\bn|_1-1$ on the cobordisms $\Sigma(\Smn)$. Therefore, Lemma \ref{delta-deg-family} asserts that $d_p$ has degree 1 with respect to the $\delta$-grading. In summary, the homological and $\delta$-gradings let us regard $(\PKC(N),d_p)$ as an element of $\cgfl$. 

\begin{proposition}
	Suppose that a planar pseudo-diagram $N$ of a link $K$ is fixed. Then the filtered chain homotopy type of $(\PKC(N),d_p)$ does not depend on the auxiliary choices.
\end{proposition} 
\begin{proof}
	Suppose that $(C_0,d_0)$ and $(C_1,d_1)$ are two filtered chain complexes that are constructed with the aid of the pseudo-diagram $N$ 
	and two different auxiliary choices. We will write $g_0$, $(f_\bm)_0$ and $\eta_0$ respectively for the metric on $\Sigma(K)$, 
	the Morse-Smale function on $\mathcal R(\Sigma(K_\bm))$, and the perturbation term. Note that we do not need to fix a homology orientation 
	for $\Sigma(\Smn)$ because we are working with the coefficient ring $\Lambda$. The relevant auxiliary choices for $(C_1,d_1)$ are also denoted by 
	$g_1$, $(f_\bm)_1$, and $\eta_1$. 
	Because the metrics $g_0$ and $g_1$ are compatible with the pseudo-diagram $N$, they agree in the branched double covers of the crossings of $N$. 
	Fix a smooth family of metrics $\{g_t\}_{t\in I}$, connecting the merits $g_0$ and $g_1$ on $\Sigma(K)$ 
	such that each metric $g_t$ is the standard metric in the branched double covers of crossings of $N$. 
	From now on, in order to distinguish the 3-manifolds and the cobordisms that are involved in the construction of the two 
	chain complexes, we use the notations $(K_{\bm})_0$, $(K_{\bm})_1$, and so on. 
	
	Let $(\Smn)_0^1$ be another copy of $\Smn$ that is regarded as a cobordism from 
	$(K_\bm)_0$ to $(K_\bm)_1$. As it was explained in subsection \ref{3-man-link}, $g_i$, for $i\in\{0,1\}$, determines a family of metrics 
	of dimension $|\bm-\bn|_1-1$ on $\Sigma(\Smn)_i$ when $(\bm,\bn)$ is a cube, double-cube, or a triple-cube. In the case that $(\bm,\bn)$ is a cube, 
	the parametrizing sets are denoted by $(\Gmn)_0$ and $(\Gmn)_1$. These families are constructed by decomposing $\Sigma(\Smn)$ to a union of a product
	manifold $\rr \times Y^L$ and $|\bm-\bn|_1$ handles. Here $Y^L$ is the branched double cover of the complement of the set of crossings. 
	On each of these pieces we fixed a metric, and different ways to glue them produce the families of metrics. 
	For $(\Gmn)_i$, we fixed the product metric on $\rr \times Y^L$ 
	that is given by the restriction of $g_i$ to the complement of the set of crossings. 
	Alternatively, we can use the metric on $\rr \times Y^L$ that is determined by $\{g_{t}\}$ 
	to construct another family of metrics on $\Sigma(\Smn)$. This family can be compactified to a family of metrics parametrized with a 
	polyhedron $\Jmn$ which has dimension $|\bm-\bn|_1$. Note that $\dim(\Jmn)=\dim(\Gmn)+1$ because the new metric on $\rr\times Y^L$ lacks the the translational 
	symmetry. The codimension 1 faces of $\Jmn$ are equal to:
	$$J_{\bm}^{\bk}\times (G_\bk^\bn)_1\hspace{1cm} (G_{\bm}^{\bk})_0\times J_\bk^\bn$$
	where $\bm>\bk>\bn$. 
	
	We can also arrange for a perturbation $\eta_0^1$ for the families of metrics $\Jmn$ such that the cobordism maps are well-defined and the perturbation term on 
	each codimension 1 face is determined by $\eta_0$, $\eta_1$, and $\eta_0^1$ on a smaller polyhedron. 
	With a similar argument as in the proof of Theorem \ref{PFH-tri-1}, the cobordism maps associated with the family of metrics $\Jmn$ can 
	be used to construct a chain map 
	$\jmn:(C_0,d_0) \to (C_1,d_1)$. Corollary \ref{h-filtered} and Lemma \ref{delta-deg-family} show that this chain map has degree $(0,0)$.
	
	The chain map $\jmn$ is analogous to the continuation maps in Floer theories and a standard argument shows that its 
	chain homotopy type does not depend on the choice of the connecting family of the metrics $\{g_t\}_{t \in I}$ and the perturbation $\eta_0^1$. 
	This argument involves defining a chain homotopy with the aid of cobordism maps associated with the product family of metrics $I \times \Jmn$. Another application
	of Corollary \ref{h-filtered} and Lemma \ref{delta-deg-family} shows that the chain homotopy has degree $(0,-1)$ (and hence degree $(-1,-1)$). Since the filtered 
	homotopy type of the chain map $\jmn$ is independent of the choices of metrics and perturbation, it defines a filtered chain homotopy equivalence.
\end{proof}	
	
Next, we want to show that if $N$ and $N'$ are two planar pseudo-diagrams for $K$, then the filtered chain homotopy types of the corresponding complexes are the same.  This is clear if $N$ and $N'$ are related to each other by a pseudo-diagram isotopy. Suppose that $N'$ is related to $N$ by adding a crossing. Order the set of crossings $N'$ such that the crossing in $N'\backslash N$ is the first one. The fact that $N$ and $N'$ are both pseudo-diagram implies that the resolution $Y_{\bm(i)}$ is an unlink for $\bm \in \{0,1\}^{n}$ and $i\in \zz$. With appropriate auxiliary data, the chain complexes $(\PKC(N),d_p)$ and $(\PKC(N'),d_p)$ are respectively equal to $(C_{\bm(2)}^{\bn(2)},d_{\bm(2)}^{\bn(2)})$ and $(C_{\bm(1)}^{\bn(0)},d_{\bm(1)}^{\bn(0)})$ where $\bm=(1,\dots,1)$ and $\bn=(0,\dots,0)$. In particular, $(\PKC(N),d_p)$ inherits two other gradings form the homological and the $\delta$-gradings associated with the pseudo-diagram $N'$. These two gradings are different from the original homological and $\delta$-gradings on $(\PKC(N),d_p)$ and we denote them by $\deg_{h'}$ and $\deg_{\delta'}$.
	
\begin{proposition}\label{rel-two-gradings-pro}
	The above gradings on $(\PKC(N),d_p)$ are related to each other in the following way:
	\begin{equation}\label{rel-two-gradings}
		\deg_{\delta'}=\deg_{\delta}-1\hspace{2cm} \deg_{h'}=\deg_{h}-1
	\end{equation}
\end{proposition}

Given an object $(C,d)$ of $\cgf$, the filtered chain complex $(C[a,b],d)$ denotes the same chain complex where the homological and the $\delta$-grading are shifted down by $a$ and $b$ respectively. For example, this means that $\mathcal F^i C[a,b]=\mathcal F^{i+a} C$. Therefore, this proposition states that $C_{\bm(2)}^{\bn(2)}$ is equal to $\PKC(N)[1,1]$.   
\begin{proof}[Proof of Proposition \ref{rel-two-gradings-pro}]
	Let $\bo'$ be the oriented resolution of the set of crossings $N$. If the first crossing of $N'$ is positive, then $\bo'=\bo(0)$ and:
	\begin{equation} \label{pos-case}
		n_+(N')=n_+(N)+1 \hspace{1cm} n_-(N')=n_-(N)\hspace{1cm}\sigma(\bk(2),\bo')=\sigma(\bk,\bo)+2
	\end{equation}	
	for any $\bk\in \{0,1\}^n$.
	A straightforward computation using (\ref{pos-case}) proves the claim in this case. 
	If the first crossing is negative, then the fact that $\bo'=\bo(1)$ and the 
	following identities verifiy (\ref{rel-two-gradings}):
	\begin{equation} \label{neg-case} 
		n_+(N')=n_+(N) \hspace{1cm} n_-(N')=n_-(N)+1\hspace{1cm}\sigma(\bk(2),\bo')=\sigma(\bk,\bo)
	\end{equation}	
\end{proof}

Theorem \ref{PFH-tri-1} sates that the chain complexes $(C_{\bm(1)}^{\bn(0)},d_{\bm(1)}^{\bn(0)})$ and $(C_{\bm(2)}^{\bn(2)},d_{\bm(2)}^{\bn(2)})$ are chain homotopy equivalent. As it can be seen from the proof of Lemma \ref{trianlge-detection-lemma}, this chain homotopy equivalence is given by the following map:
$$\Phi:(C_{\bm(2)}^{\bn(2)},d_{\bm(2)}^{\bn(2)}) \to (C_{\bm(1)}^{\bn(0)},d_{\bm(1)}^{\bn(0)})=\Cone(g_{\bm(1)}^{\bn(1)})$$
$$\Phi:=(g_{\bm(2)}^{\bn(2)},n_{\bm(2)}^{\bn(2)})$$
In fact, if we define:
$$\Psi_i:\Cone(g_{\bm(i-1)}^{\bn(i-1)}) \to (C_{\bm(i-3)}^{\bn(i-3)},d_{\bm(i-3)}^{\bn(i-3)})$$
\begin{equation*}
	\Psi_i:=\left(
	\begin{array}{c}
		n_{\bm(i-1)}^{\bn(i-1)}\\
		g_{\bm(i-2)}^{\bn(i-2)}
	\end{array}
	\right)
\end{equation*}	
then $\Psi_{2}\circ \Phi$ is chain homotopic to the isomorphism $q_{\bm(2)}^{\bn(2)}$ with the chain homotopy being $k_{\bm(2)}^{\bn(2)}$. Moreover, $\Phi\circ \Psi_{5}$ is chain homotopic to the isomorphism:
\begin{equation}\label{htpy-equiv}
	Q:=\left(
	\begin{array}{cc} 
		q_{\bm(4)}^{\bn(4)}&0\\
		n_{\bm(2)}^{\bn(2)}\circ n_{\bm(4)}^{\bn(4)}+k_{\bm(3)}^{\bn(3)}\circ g_{\bm(4)}^{\bn(4)}+g_{\bm(1)}^{\bn(1)} \circ k_{\bm(4)}^{\bn(4)}&
		q_{\bm(3)}^{\bn(3)}
	\end{array}
	\right) 
\end{equation}
with the chain homotopy:
\begin{equation}\label{htpies}
	K:=
	\left(\begin{array}{cc}
		k_{\bm(4)}^{\bn(4)}&n_{\bm(3)}^{\bn(3)}\\
		0&k_{\bm(3)}^{\bn(3)}
	\end{array}
	\right).
\end{equation}

According to Lemmas \ref{period-deg} and \ref{rel-two-gradings-pro}, the above chain maps and chain homotopies induce the following diagram: 

\begin{equation} \label{add-cross-iso}
	\xymatrix{
		\PKC(N')[2,2]\ar[rr]^{Q}\ar[dr]^{\Psi_5}&&\PKC(N')\ar[dr]^{\Psi_{2}}\ar@{=>}[d]^{k_{\bm(2)}^{\bn(2)}}&\\
		&\PKC(N)[1,1]\ar[ur]^{\Phi}\ar[rr]_{q_{\bm(2)}^{\bn(2)}}\ar@{=>}[u]^{K}&&\PKC(N)[-1,-1]
	}
\end{equation}

Next, we need to study how the maps in diagram (\ref{add-cross-iso}) behave with respect to the homological and the $\delta$-gradings. Firstly we start with the easier case of the $\delta$-grading. By Lemma \ref{delta-deg-family}, the degrees of the maps $\Phi$, $\Psi_i$, $k_{\bm(i)}^{\bn(i)}$, and $K$ with respect to the $\delta$-grading are equal to 1. The same lemma can be used to check that the $\delta$-degrees of the maps $q_{\bm(i)}^{\bn(i)}$ and $Q_i$ are equal to 2. In order to treat the case of the homological grading, we need to investigate further the behavior of the cobordism map $\tilde f_{\Gmn}$ when $h(\bm)=h(\bn)$:

\begin{lemma} \label{border-case-slef-int}
	If $(\bm,\bn)$ defines a cube in $\{0,1\}^n$, with $\Smn\cdot\Smn=2|\bm-\bn|_1$, then the map $\tilde f_{\Gmn}$ vanishes. 
\end{lemma}

\begin{proof}
	Let $D$ be a diagram such that $N$ is related to $N_D$ as in Definition \ref{pl-pseudo-diagram}. 
	For $\bk \in \{0,1\}^{|N|}$, the resolution $K_\bk$ of the set of crossings of $N$ can be 
	also regarded as a resolution $K_{\bk'}$ of the set of crossings $N_D$ with $\bk' \in \{0,1,2\}^{|N_D|}$. In particular, we can consider the projection of $K_{\bk'}$ 
	to $\rr^2 \times \{0\}\subset \rr^3$. Note that this planar projection has exactly $|N_D \backslash N|$ crossings.
	
	According to Lemma \ref{self-int-cob}, for $\bm\geq \bk \geq \bn$, the resolution $K_\bk$ has the same number of connected components as $K_\bm$ 
	(or equivalently $K_\bn$). 
	Therefore, if $N$ is $N_D$, then $\bm=\bn$. The map $\fGmn$ in this case vanishes because  the differential of the complex $\PFC(K_\bm)$ is zero.
	Now let $N$ be given by removing a crossing from $N_D$. An examination of different possible cases shows
	that $|\bm-\bn|_1 \leq 1$. Furthermore, the cobordism $\Smn$ is the union of a product cobordism between unlinks and 
	a twice punched $\rr P^2$ with self-intersection 2.
	Therefore, $\Wmn$, the branched double cover of $\Smn$ is diffeomorphic to $(I \times \#^l{S^2 \times S^1})\# \widebar{\cc P}^2$ . 
	Proposition \ref{mor-b^+} shows that $\PFH(\Wmn)$ (and hence $\fGmn$) is zero.
 
 	Next let $N_D \backslash N$ consists of two consecutive under-crossings or over-crossings. An examination of different possibilities shows that 
	$|\bm-\bn|_1 \leq 2$. If $|\bm-\bn|_1 \leq 1$, then the same argument as in the above cases verifies
	that $\fGmn=0$. If $|\bm-\bn|_1=2$, then there is an integer number $l$ such that $\Wmn$ is diffeomorphic to 
	$(\rr \times \#^l S^1 \times S^2) \# \widebar{\cc P}^2\# \widebar{\cc P}^2$. 
	Recall that $\Wmn$
	can be decomposed to two handles $\bar H_1$ and $\bar H_2$ and $\rr \times Z$ where $Z$ 
	is a 3-manifold whose boundary is the union of two tori $T_1$ and $T_2$.
	The family of metrics $\bWmn$ parametrized by the closed interval $\Gmn$ is given by gluing translations of $H_1$ and $H_2$ to $\rr \times Z$. 
	It is easy to check that
	in all possible cases there are two disjoint discs $D_1$ and $D_2$ in $Z$ that fill two longitudes in $T_1$ and $T_2$, respectively. 
	Furthermore, gluing $\rr \times N(D_i)$ to 
	$H_i$ produces $\tilde H_i$ whose boundary is diffeomorphic to $\rr \times S^2$. Here $N(D_i)$ is a regular neighborhood of the disc $D_i$. Removing $N(D_1)$
	and $N(D_2)$ from $Z$ produces $\tilde Z$, a three manifold whose boundary consists of two copies of the 2-dimensional sphere. 
	
	 \begin{figure}
	\centering
	\begin{subfigure}[b]{0.30 \textwidth}
		\includegraphics[width=\textwidth]{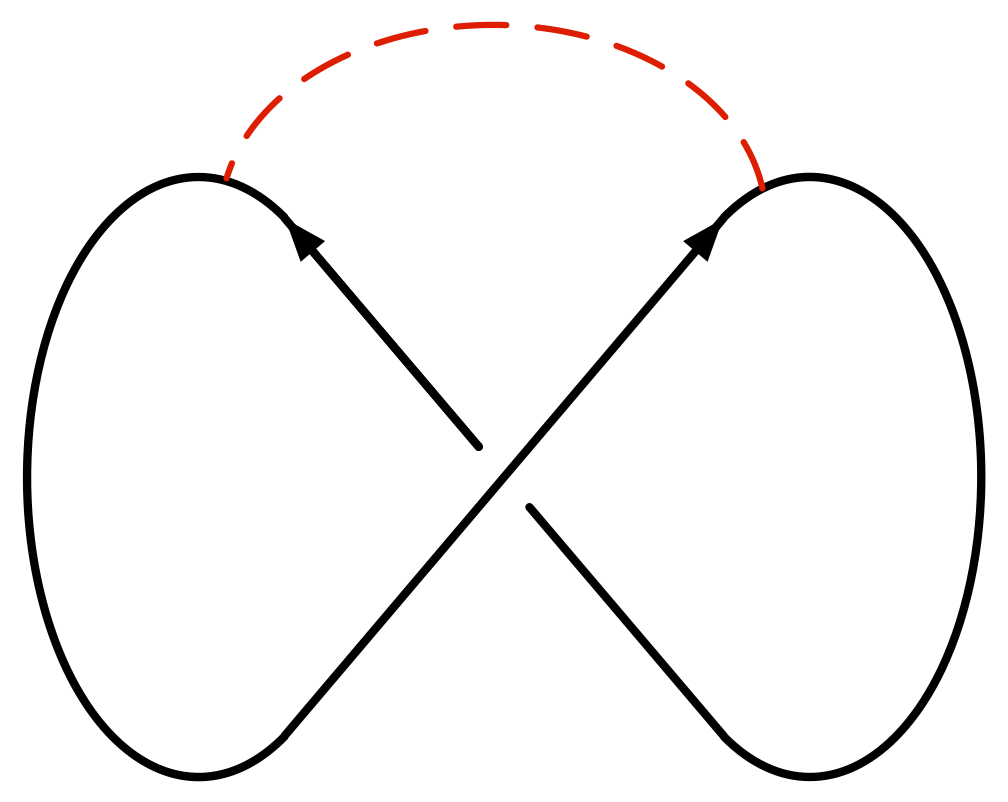}
		\caption{This diagram and a surgery along the dotted path define the 1- and 0-resolution of a pseudo-diagram with only one crossing. 
		The cobordism between these two resolutions has self-intersection 2.}
		\label{unknot-pseudodiagram-1}
	\end{subfigure}
	\begin{subfigure}[b]{0.20 \textwidth}
	\includegraphics[width=\textwidth]{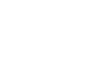}
	\end{subfigure}
	\begin{subfigure}[b]{0.40 \textwidth}
		\includegraphics[width=\textwidth]{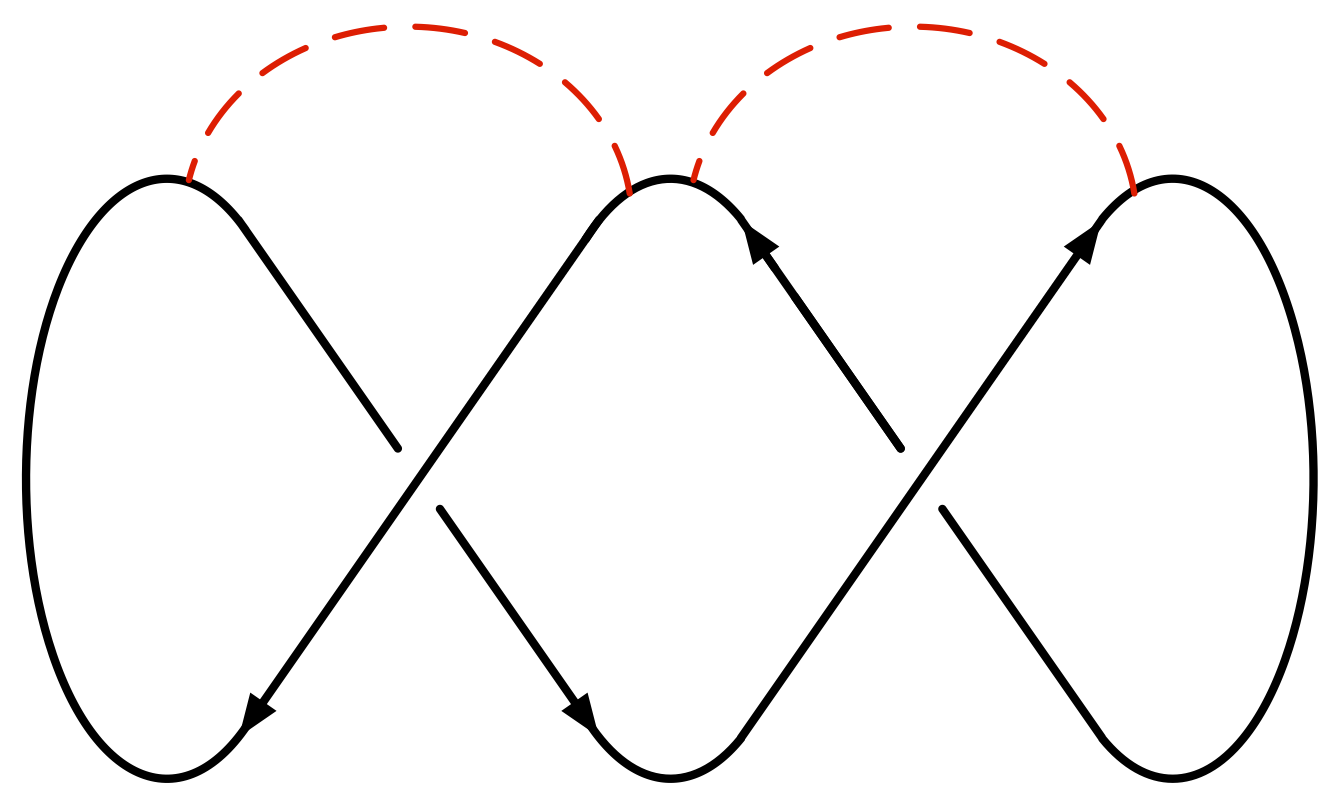}
		\caption{This diagram and surgery along either/both of the dotted paths define vertices of the only face of a pseudo-diagram with two crossings.
		There are four cobordisms associated with the edges of this face and all of these cobordisms have self-intersection 2.}
		\label{unknot-pseudodiagram-2}
	\end{subfigure}
	\caption{}
	\end{figure}

	The family of metrics $\bWmn$  can be alternatively constructed by gluing translations of $\tilde H_1$ and $\tilde H_2$ to $\rr \times \tilde Z$. The space $\tilde H_1$ 
	includes a 2-sphere that its self-intersection is equal to -1. Therefore, we can construct 
	a family of metrics on $\tilde H_1$ parametrized by $[0,\infty]$ given by stretching this 2-sphere.
	Gluing translation of this family of metrics on $\tilde H_1$ to $\rr \times \tilde Z$ and $\tilde H_2$ produces a 2-dimensional family of metrics 
	that is parametrized by $[0,\infty] \times \Gmn$. The four faces of $[0,\infty] \times \Gmn$ parametrizes the following  families:
	\begin{itemize}
		\item The face $\{0\}\times \Gmn$ parametrizes the family of metrics $\bWmn$.
		\item The face $\{\infty\}\times \Gmn$ parametrizes a family of broken metrics. One of the factors of this family is given by a fixed metric on a copy of 
		$\widebar {\cc P}^2\backslash D^4$. Therefore, Lemma \ref{special-fam} implies that the corresponding cobordism map is zero.
		\item There are $\bk,\bk'\in \{0,1\}^n$ such that $\bk\neq \bk'$ and $\bm>\bk,\bk'>\bn$. 
		The two remaining faces of $[0,\infty] \times \Gmn$ parametrizes two families of broken metrics of the form 
		$\bWkn \times \mathfrak V_\bk$ and $\ \mathfrak V_{\bk'}\times \mathfrak W_\bm^{\bk'}$ where $\mathfrak V_\bk$ and $\mathfrak V_\bk$
		are families of metrics on the cobordisms $W_\bm^\bk$ and $W_{\bk'}^\bn$, respectively. Therefore, the cobordism maps for these two faces of 
		$[0,\infty] \times \Gmn$ are equal to $\tilde f_{G_\bk^\bn} \circ \tilde f^\eta_{\mathfrak V_\bk}$ and $\tilde f^\eta_{\mathfrak V_{\bk'}} \circ \tilde f_{G_\bm^{\bk'}}$.
		Since the cobordism maps $\tilde f_{G_\bk^\bn}$ and $\tilde f_{G_\bm^{\bk'}}$ vanish,
		the cobordism maps for these faces are also zero.
	\end{itemize}
	Extend the perturbation term of the family of metrics $\Gmn$ to a perturbation for $[0,\infty] \times \Gmn$. 
	The identity (\ref{pair-signed-bdry-form-loc-sys}) and the above realization of the faces of the family of metrics parametrized by $[0,\infty] \times \Gmn$ 
	show that the induced cobordism map satisfies the following relation:
	$$\tilde f_{[0,\infty] \times \Gmn} \circ d_\bm+d_{\bn}\circ \tilde f_{[0,\infty] \times \Gmn}=\fGmn$$
	Vanishing of the differentials $d_\bm$ and $d_\bn$ proves the claim that $\fGmn=0$.
\end{proof}

Lemma \ref{border-case-slef-int} and the proof of Corollary \ref{h-filtered} imply that the cobordism maps $\fGmn$ increases the $h$-grading by at least 1. Given $\bm,\bn\in \{0,1\}^n$, the self-intersection of the cobordism $S_{\bm(i)}^{\bn(i-2)}$ can be bounded in the following way:
\begin{align}
	S_{\bm(i)}^{\bn(i-2)}\cdot S_{\bm(i)}^{\bn(i-2)}&=S_{\bm(i)}^{\bn(i)}\cdot S_{\bm(i)}^{\bn(i)}+S_{\bn(i)}^{\bn(i-2)}\cdot S_{\bn(i)}^{\bn(i-2)}\nonumber\\
	&\leq 2|\bm-\bn|_1+2\nonumber\\
	&\leq2|\bm(i)-\bn(i-2)|_1-2\nonumber
\end{align}
This shows that the cobordism map $\tilde f_{N_{\bm(i)}^{\bn(i-2)}}$ increases the homological grading by at least 1. Therefore, the maps $\Phi$ and $\Psi_1$ have degree 1 with respect with the homological grading associated with the pseudo-diagram $N$.  A similar argument can be used to show that $\tilde f_{K_{\bm(i)}^{\bn(i-3)}}$ increases the homological grading by at least 2, and hence all the maps involved in the definition of $k_{\bm(i)}^{\bn(i)}$, $K$, (respectively, $q_{\bm(i)}^{\bn(i)}$ and $Q$) increase the homological grading by at least 1 (respectively, 2). In summary, all the maps in Diagram (\ref{add-cross-iso}) have the right degree to show that $\Phi$ determines a filtered chain homotopy equivalence from $\PKC(N)$ to $\PKC(N')$. The following theorem summarizes our analyze of the homological and $\delta$-gradings of the maps in Diagram (\ref{add-cross-iso}):

\begin{theorem} \label{knot-invt}
	Suppose $N$ is a planar pseudo-diagram for a link $K$. Then the filtered chain homotopy type of $(\PKC(N),d_p)$ depends only on $K$.
\end{theorem}

A filtered $\zz$-graded chain complex $(C,d)$ determines a spectral sequence $\{(E_r,d_r)\}$ (cf., for example, \cite[Chapter XI]{MacLane:Hom}). Moreover, $f:(C,d) \to (C',d')$, a morphism of the category $\cgf$, induces a morphism $f_i:E_r \to E_r'$ of the corresponding spectral sequences. If $g:(C,d) \to (C',d')$ is another morphism in $\cgf$ that represents the same element as $f$, then for $i\geq 2$, $f_i=g_i$ (cf. \cite[Chapter XI, Proposition 3.5]{MacLane:Hom}). Therefore, for $i \geq 2$, there is a functor $\mathcal E_r$ form $\cgf$ to $\mathcal C_g(R)$, the category of bi-graded chain complexes over the ring $R$, that maps a filtered $\zz$-graded chain complex to the $i^{\rm th}$ page of the associated spectral sequence. Note that a morphism in $\mathcal C_g(R)$ is a chain map rather than a chain map up to chain homotopy. 

Motivated by Theorem \ref{knot-invt}, we will write $(\PKC(K),d_p)$ for the isomorphism class of $(\PKC(N),d_p)$ for any choice of a planar pseudo-diagram. The bi-graded chain complex $\mathcal E_r((\PKC(K),d_p))$ is also denoted by $(\PKE_r(K),d_r)$:
\begin{corollary} \label{spe-seq}
	For each $r\geq 2$, the isomorphism class of $(\PKE_r(K),d_r)$, as a bi-graded chain complex, is an invariant of $K$.
\end{corollary}
In the next section we shall explore the relationship between $\PKE_2(K)$ and Khovanov cohomology of the mirror image, $\widebar{K}$.


\section{Branched Double Covers and Odd Khovanov Homology} \label{odd-kh}
Suppose $D$ is a planar diagram for a link $K$ and $(\PKC(N_D),d_p)$ is the corresponding representative of $\PKC(K)$. The first page of the associated spectral sequence is equal to:
\begin{equation} \label{twisted-chain-com}
	(\bigoplus_{\bm\in \{0,1\}^{|N_D|}}\PFH(\Sigma(K_\bm)),\bigoplus_{(\bm,\bn)\in \mathcal E_c}
	\PFH(\Sigma(\Smn))
\end{equation}		
where $\mathcal E_c$ denotes the set of the edges in the cube $\{0,1\}^{|N_D|}$. This chain complex inherits the following gradings:
\begin{align} 
	\deg_h&=-|\bm|_1+n_-(D)\label{diagram-grading-h}\\
	\deg_\delta&=-\frac{1}{2}\deg_p-\frac{1}{2}|{\bm}|_1+\frac{1}{2}n_+(D) \label{diagram-grading-delta}
\end{align}
Note that the term $\sigma(\bm,\bn)$ does not contribute in the above formulas, because $D$ is a planar diagram and the self-intersection of the cobordism $\Smn$, for $(\bm,\bn)\in \{0,1\}^{|N_D|}$, is zero. Corollary \ref{spe-seq} asserts that the homology of this bi-graded complex is a link invariant. Our main goal in this section is to show that this link invariant, as a bi-graded module, is isomorphic to Khovanov homology of $\widebar K$ with coefficients in $\Lambda$. In fact, we shall prove a generalization of this claim when the coefficient ring is replaced with $\tLz$. 

Choose compatible homology orientations for the cobordisms $\Wmn=\Sigma(\Smn)$ as in Lemma \ref{fix-hom-ori}. The complex (\ref{twisted-chain-com}) can be lifted to the following chain complex with coefficients in $\tLz$:
\begin{equation} \label{odd-chain-com}
	(C_o(D),d_o):=(\bigoplus_{\bm\in \{0,1\}^n}\tPFH(\Sigma(K_\bm)),\bigoplus_{(\bm,\bn)\in \mathcal E_c}{\sgn(\bm,\bn)}\tPFH(\Sigma(\Smn))
\end{equation}
Recall that $\sgn(\bm,\bn)$, for $(\bm,\bn) \in \mathcal E_c$, is given by Formula (\ref{eta}) and is equal to the orientation of the 0-dimensional parametrizing set $\Gmn$. It is straightforward to check that $d_o$ defines a differential on $C_o(D)$. The same formulas as in (\ref{diagram-grading-h}) and (\ref{diagram-grading-delta}) define a bi-grading on $C_o(D)$ which are also called the homological and $\delta$-gradings.
 
In subsection \ref{odd-khov-rev} it is shown that the bi-graded chain complex (\ref{odd-chain-com}) is isomorphic to the odd Khovanov homology of the mirror image $\widebar K$ with coefficients in $\tLz$. This interpretation of odd Khovanov homology is used to construct a spectral sequence from odd Khovanov homology that converges to $\tPFH(\Sigma(K))$. This spectral sequence is a lift of the spectral sequence in Corollary \ref{spe-seq}. The subsection \ref{unlinks-cob-maps} is devoted to make some computations related to the plane Floer homology of the unlinks. These computations provide the main input for subsection \ref{odd-khov-rev}.

\subsection{Plane Floer Homology of Unlinks} \label{unlinks-cob-maps}
In this subsection, we study plane Floer homology of unlinks and maps associated with the elementary cobordisms. We take up this task firstly for a fixed link $U_n$. Let $U_n$ be the $n$-component link in $\rr^2\times \{0\} \subset \rr^3$ whose $i^{\rm th}$ component is a circle of radius $\frac{1}{4}$ centered around the point $(0,i)$. The link $U_n$ is obviously an unlink. The branched double cover $\Sigma(U_n)$ is diffeomorphic to $\#^{n-1}S^2 \times S^1$. This 3-manifold has a unique torsion $\spinc$ structure and its first Betti number is equal to $n-1$. Fix a metric on $\Sigma(U_n)$ and use the spin connection to fix an identification of  $\mathcal R(\Sigma(U_n))$ and $J(\Sigma(U_n))$. For each $1\leq i\leq n-1$, the path $p_i$ in Figure \ref{unlink} that connects the $i^{\rm th}$ connected component of $U_n$ to the $n^{\rm th}$ one can be lifted to an oriented loop $\gamma_i$ in the branched double cover. 
The set $\{\gamma_1,\dots,\gamma_{n-1}\}$ forms a basis of $H_1(\Sigma(U_n),\rr)$. Using the dual basis for $H^1(\Sigma(U_n),\rr)$, we have the following identification:
$$\mathcal R(\Sigma(U_n))=S_1 \times \dots \times S_{n-1}:=\{(z_1,\dots,z_{n-1})\in \cc^{n-1}\mid |z_i|=1\}$$
where $S_i$ is in correspondence with the loop $\gamma_i$. 

For $\tau=(\tau_1,\dots,\tau_{n-1}) \in \{-1,1\}^{n-1}$, let $T_\tau \subset \mathcal R(\Sigma(U_n))$ be the subspace: 
$$T_1\times \dots \times T_{n-1}\subseteq S_1\times \dots \times S_{n-1}$$ 
where $T_i=S_i$ if $\tau_i=1$, and otherwise $T_i=\{-1\}$. The normal bundle of this sub-manifold of $\mathcal R(\Sigma(U_n))$ is isomorphic to $l_1 \oplus \dots \oplus l_{n-1}$ where $l_i=\{0\}$ if $T_i=S_i$ and $l_i=\underline \rr$ if $T_i=\{-1\}$. In the latter case, orient $l_i$ with the standard orientation of the circle $S_i$. The product orientation defines a co-orientation for $T_\tau$. The poincar\'e duals of these co-oriented tori determine a basis for the cohomology of $\mathcal R(\Sigma(U_n))$ and hence a basis for $\tPFH(U_n)$.

\begin{figure}
	\centering
		\includegraphics[width=.6\textwidth]{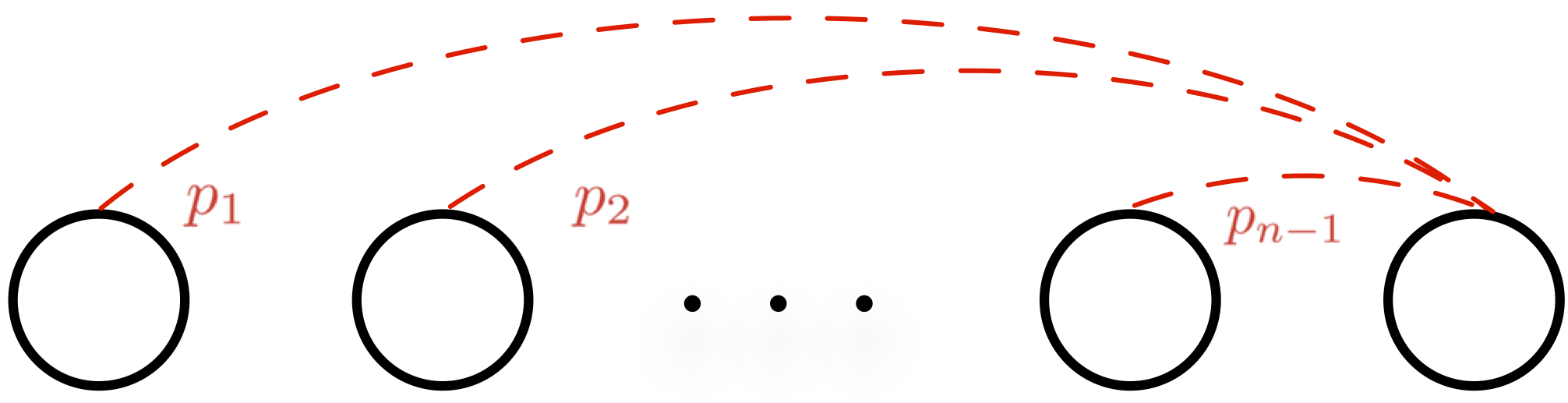}
		\caption{The unlink $U_n$}
		\label{unlink}
\end{figure}

This basis can be used to define an isomorphism $L_n:\tPFH(\Sigma(U_n)) \to \Lambda^*(\widetilde H^0(U_n))$, where $\widetilde H^0(U_n)$ is the reduced cohomology of $U_n$. For $1\leq i\leq n-1$, let $v_i\in\tPFH \widetilde H^0(U_n)$ be the element that assigns 1 to the $i^{\rm th}$ component of $U_n$, $-1$ to the $n^{\rm th}$ component, and 0 to the other components. The linear map $L_n$, by definition, sends the generator of $\tPFH(U_n)$, represented by $T_\tau$, to: 
$$(-1)^{k(n-1)} v_{i_1}\wedge \dots \wedge v_{i_k}\in \Lambda^*(\widetilde H^0(U_n))$$ 
where $i_1 <i_2<\dots <i_k$, and $l\in \{i_1,\dots,i_n\}$ if and only if $\tau_l=-1$.

\begin{remark}
	Equip $\mathcal R(\Sigma(U_n))$ with the standard product metric and consider the Morse-Smale function that maps $(z_1,\dots,z_{n-1})$ to $\sum (z_i+\bar z_i)$. 
	The set of the critical points of this function is equal to $\{-1,1\}^{n-1}$ and the closure of the unstable manifold of the critical point $\tau$ is equal to $T_\tau$. 
	Therefore, the basis of $\tPFH(U_n)$, mentioned above, consists of the cycles naturally 
	associated with the generators of the Morse complex of $\mathcal R(\Sigma(U_n))$.
\end{remark}

\begin{proposition} \label{ele-cob-maps}
	Suppose $C_n$ and $P_n$ are respectively the cap and the split cobordisms from $U_n$ to $U_{n+1}$ (Figure \ref{U_n-to-U_n+1}). 
	Then an appropriate homology orientations for these 
	cobordisms can be fixed such that the corresponding cobordism maps are equal to:
	\begin{equation} \label{cob-map-C-S}
		\tPFH(\Sigma(C_n))(v_{i_1}\wedge\dots\wedge v_{i_k})=v_{i_1}\wedge\dots\wedge v_{i_k}\hspace{.5cm}
		\tPFH(\Sigma(P_n))(v_{i_1}\wedge\dots\wedge v_{i_k})=v_n\wedge v_{i_1}\wedge\dots\wedge v_{i_k}
	\end{equation}	
	Similarly, the cup cobordism $\widebar C_n:U_{n+1}\to U_{n}$ and the merge cobordism $\widebar P_n:U_{n+1}\to U_{n}$ (Figure \ref{U_n+1-to-U_n}) can be
	 given homology orientations such that:
	$$\hspace{.7cm}\tPFH(\Sigma(\widebar C_n))(v_{i_1}\wedge\dots\wedge v_{i_k})=0\hspace{2cm} 
	\tPFH(\Sigma(\widebar P_n))(v_{i_1}\wedge\dots\wedge v_{i_k})=v_{i_1}\wedge\dots\wedge v_{i_k}$$
	\begin{equation} \label{cob-map-bC-M}
		\tPFH(\Sigma(\widebar C_n))(v_n \wedge v_{i_1}\wedge\dots\wedge v_{i_k})=v_{i_1}\wedge\dots	\wedge v_{i_k}\hspace{.5cm}
		\tPFH(\Sigma(\widebar P_n))(v_n \wedge v_{i_1}\wedge\dots\wedge v_{i_k})=0
	\end{equation}
	where $1\leq i_l\leq n-1$.
\end{proposition}

 \begin{figure}[t]
	\centering
	\begin{subfigure}[b]{0.35 \textwidth}
		\includegraphics[width=\textwidth]{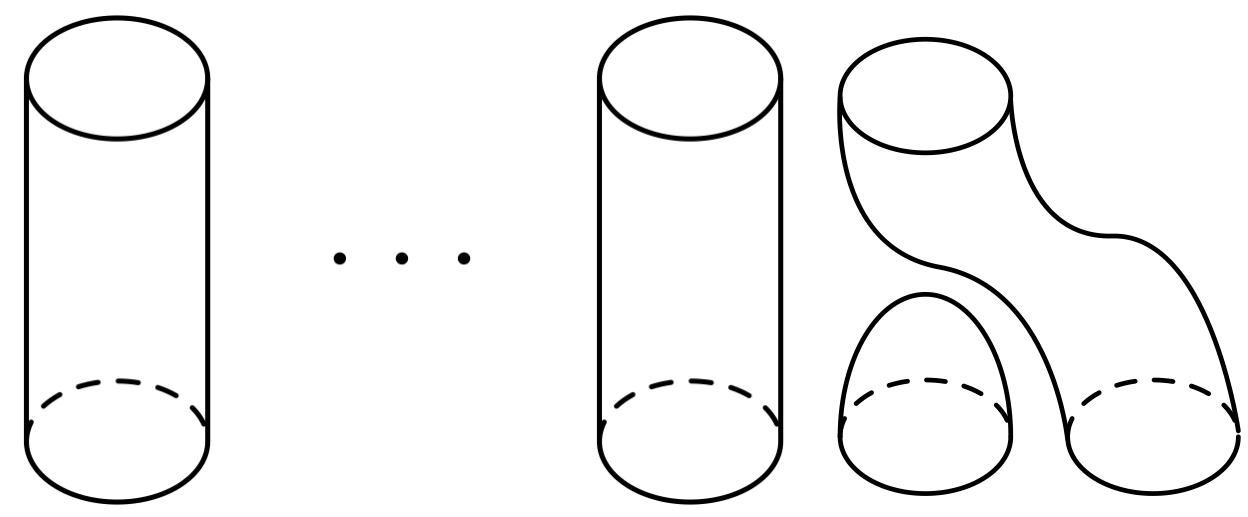}
		\caption{Cap cobordism $C_n$}
		\label{cap-cob}
	\end{subfigure}
	\begin{subfigure}[b]{0.10 \textwidth}
	\includegraphics[width=\textwidth]{empty}
	\end{subfigure}
	\hspace{7pt}
	\begin{subfigure}[b]{0.35 \textwidth}
		\includegraphics[width=\textwidth]{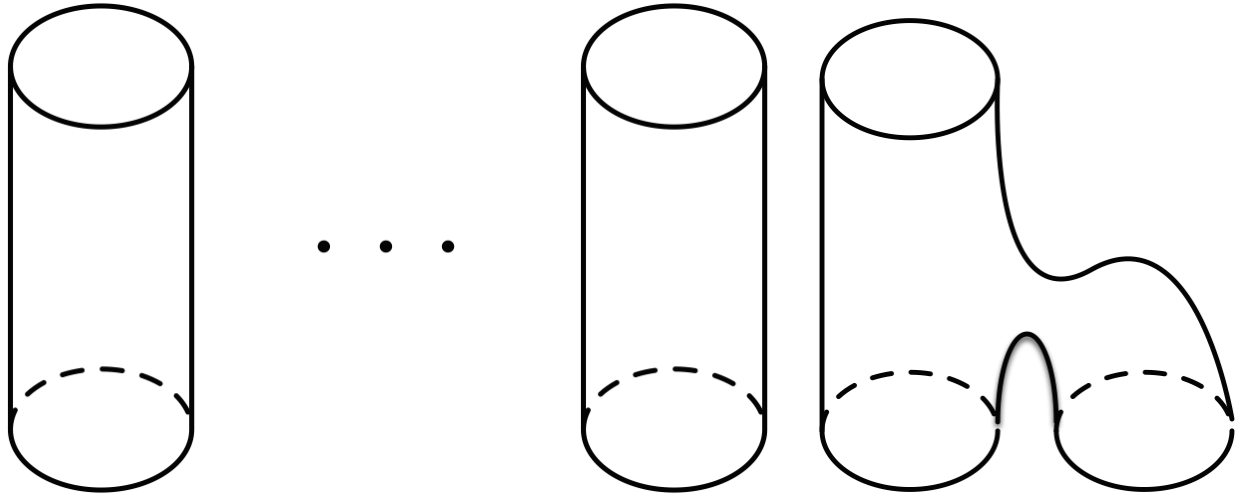}
		\caption{Split cobordism $P_n$}
		\label{split-cob}
	\end{subfigure}
	\caption{Elementary Cobordisms from $U_n$ to $U_{n+1}$}
	\label{U_n-to-U_n+1}
\end{figure}

 \begin{figure}[b]
	\centering
	\begin{subfigure}[b]{0.35 \textwidth}
		\includegraphics[width=\textwidth]{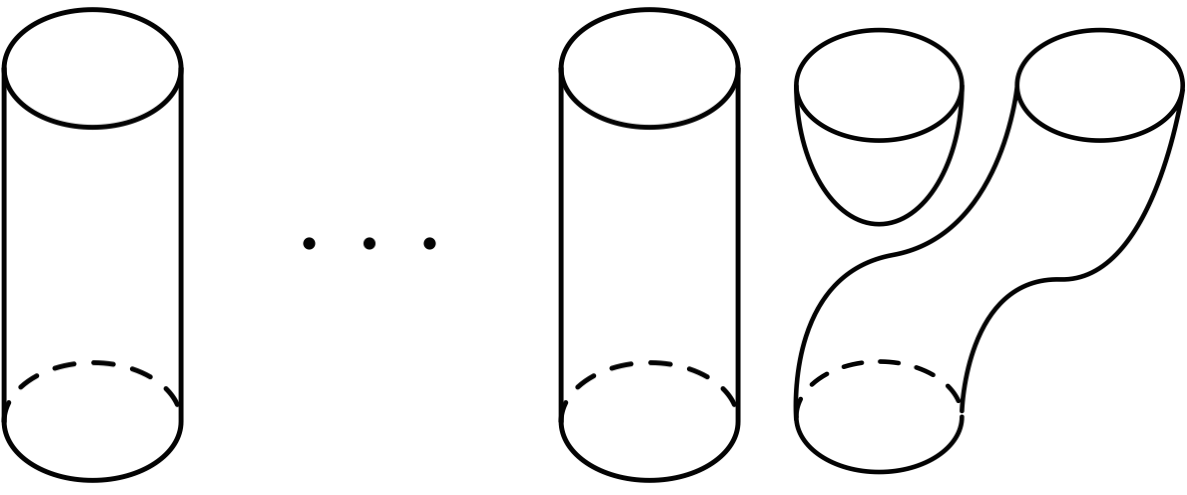}
		\caption{Cup cobordism $\widebar C_n$}
		\label{cup-cob}
	\end{subfigure}
	\begin{subfigure}[b]{0.10 \textwidth}
	\includegraphics[width=\textwidth]{empty}
	\end{subfigure}
	\hspace{7pt}
	\begin{subfigure}[b]{0.35 \textwidth}
		\includegraphics[width=\textwidth]{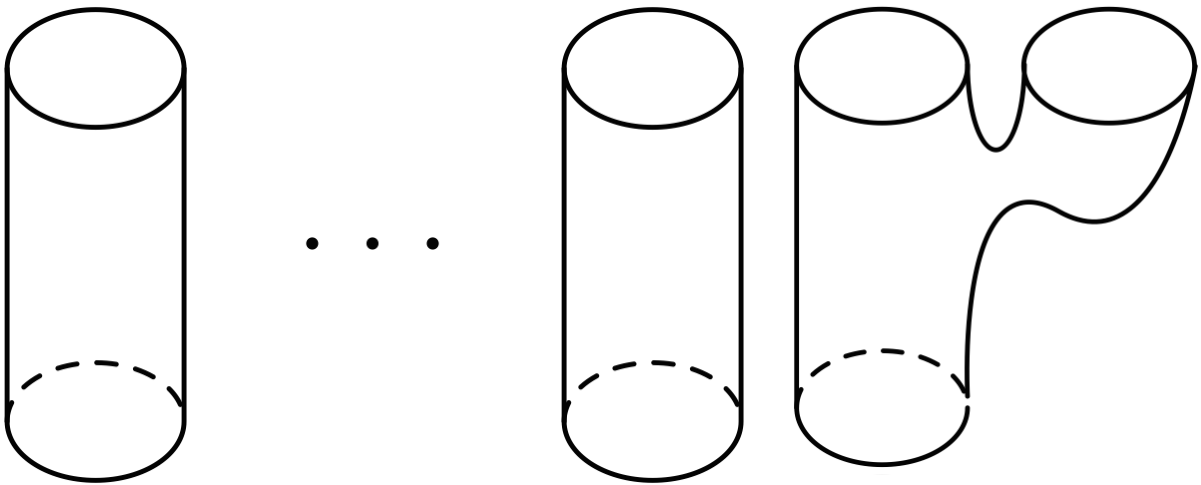}
		\caption{Merge cobordism $\widebar P_n$}
		\label{merge-cob}
	\end{subfigure}
	\caption{Elementary Cobordisms from $U_{n+1}$ to $U_n$}
	\label{U_n+1-to-U_n}
\end{figure}

\begin{proof}
	The cobordism $\Sigma(C_n)$ is diffeomorphic to: 
	$$((D^3 \times S^1)\backslash D^4)\natural (I\times (\#^{n-1}S^2\times S^1)):\#^{n-1} S^2\times S^1 \to \#^n S^2\times S^1 $$  
	
	Here $(D^3\times S^1)\backslash D^4$ is regarded as a cobordism from $S^3$ to $S^1 \times S^2$. Given two cobordisms $W:Y_0 \to Y_1$ and $W':Y_0' \to Y_1'$,
	$W \natural W'$ is a cobordism from the connected sum $Y_0 \# Y_0'$ to the connected sum $Y_1 \# Y_1'$. This cobordism is constructed by removing
	a neighborhood of paths in $W$ and $W'$ which connect the incoming and the outgoing ends and then gluing the resulted manifolds along the boundary of the 
	regular neighborhoods of the paths.
	Therefore, $b^+(\Sigma(C_n))=0$, $b^1(\Sigma(C_n))=n$, 
	and there is exactly one $\spinc$ structure that restricts to the torsion $\spinc$ structure on the boundary. This $\spinc$ structure is the lift of the spin structure on 
	$\Sigma(C_n)$. 
	Because $b^+(\Sigma(C_n))=0$, we can use the second case of Proposition \ref{mor-b^+} to compute the cobordism map associated with $\Sigma(C_n)$. 
	The corresponding pull-up-push-down diagram is given by:
	\begin{equation} \label{diag-C_n} 
		\xymatrix{
			&S_1 \times S_2 \times \dots \times S_{n-1}\times S_n\ar[dl]_{j_n} \ar[dr]^{id}&\\
			S_1 \times S_2 \times \dots \times S_{n-1} &&S_1 \times S_2 \times \dots S_{n-1}\times S_n\\
			}
	\end{equation}
	where the map $j_n$ sends the point $(z_1,\dots ,z_{n-1},z_n)$ to $(z_1,\dots ,z_{n-1})$. The spaces $M(\Sigma(C_n))$ and $\mathcal R(U_{n+1})$ 
	are the same and we orient them with the same orientations. Use these orientations of  $M(\Sigma(C_n))$ and $\mathcal R(U_{n+1})$ to fix a homology 
	orientation for $\Sigma(C_n)$.  Our conventions from subsection \ref{ori-loc} implies that:
	$$\tPFH(\Sigma(C_n))({\rm PD}[T_{\tau}])=(-1)^k {\rm PD}[T_{\tau}\times S_n]$$
	where $k$ is the codimension of $T_\tau$. Therefore, $\tPFH(\Sigma(C_n))$ has the form in (\ref{cob-map-C-S}).
	
	In the case of the split cobordism $P_n: U_{n}\to U_{n+1}$, the branched double cover $\Sigma(P_n)$ is diffeomorphic to:
	$$((D^2\times S^2)\backslash D^4) \natural (I\times (\#^{n-1} S^2 \times S^1)):\#^{n-1} S^2\times S^1 \to \#^n S^2\times S^1$$
	Here  $(D^2\times S^2)\backslash D^4$ is regarded as a cobordism from $S^3$ to $S^1 \times S^2$. 
	The diagram of the moduli space and the restriction maps is isomorphic to:
	\begin{equation} \label{diag-S_n}
		\xymatrix{
			&S_1 \times S_2 \times \dots \times S_{n-1}\ar[dl]_{id} \ar[dr]^{i_n}&\\
			S_1 \times S_2 \times \dots \times S_{n-1} &&S_1 \times S_2 \times \dots S_{n-1}\times S_n\\
			}
	\end{equation}
	where the map $i_n$ maps $(z_1,\dots,z_{n-1})$ to $(z_1,\dots,z_{n-1},1)$. Fix an arbitrary orientation of $M(\Sigma(P_n))$ and let 
	$\mathcal R(U_{n+1})=M(\Sigma(P_n)) \times S_n$ have the product orientation. Again, define a homology orientation of $\Sigma(P_n)$ using these orientations.
	Then we have:
	$$\tPFH(\Sigma(P_n))({\rm PD}[T_{\tau}])=(-1)^n {\rm PD}[T_{\tau}]$$
	which finishes the proof of (\ref{cob-map-C-S}).
	
	Note that $C_n \circ \widebar P_n$ and $P_n \circ \widebar C_n$ are product cobordisms. Fix homology orientations for $\widebar P_n$ and $\widebar C_n$
	such that the composed orientations on the product cobordisms $C_n \circ \widebar P_n$ and $P_n \circ \widebar C_n$ are trivial. 
	Therefore, the following compositions are identity:
	$$\tPFH(\Sigma(\widebar P_n)) \circ \tPFH(\Sigma(C_n)) \hspace{2cm} \tPFH(\Sigma(\widebar C_n)) \circ \tPFH(\Sigma(P_n))$$
	On the other hand, an application of Proposition \ref{mor-b^+} shows that the following cobordism maps vanish:
	$$\tPFH(\Sigma(\widebar P_n)) \circ \tPFH(\Sigma(P_n)) \hspace{2cm} \tPFH(\Sigma(\widebar C_n)) \circ \tPFH(\Sigma(C_n))$$
	This is true because $b^+(\Sigma(P_n \circ \widebar P_n))>0$ and $\Sigma(C_n \circ \widebar C_n)\cong (I \times \#^{n-1} S^2 \times S^1)\# S^3 \times S^1$. 
	Identity (\ref{cob-map-bC-M}) can be verified with the characterization of the above compositions.
\end{proof}

From now on we use the term unlink in the following more restrictive sense:
\begin{definition}
	An unlink is a pair $(\mU,\mathcal D)$ of a link $\mU$ with $n$ connected components and a Seifert surface $\mathcal D$ that consists of $n$ discs. 
	An {\it unlink isotopy} $\mathcal J$ from $(\mU,\mathcal D)$ to $(\mU',\mathcal D')$ is an isotopy of surfaces (with boundary) that starts from 
	$\mathcal D$ and ends at $\mathcal D'$. 
\end{definition}

The link $U_n$ comes with an obvious disc filling $D_n$ that is embedded in the plane $\rr^2 \times \{0\} \subset \rr^3$, and $(U_n,D_n)$ is our standard example of unlink with $n$ components. Note that restriction to the boundary of an unlink isotopy $\mathcal J:(\mU,\mathcal D) \to (\mU',\mathcal D')$ induces an isotopy $J:\mU \to \mU'$. If $\mathcal  J$ and $\mathcal  J'$ are two unlink isotopies, then a homotopy between $\mathcal  J$ and $\mathcal  J'$ is a smooth family of unlink isotopies $\mathcal J_t:(\mU,\mathcal D) \to (\mU',\mathcal D')$ such that $\mathcal J_0=\mathcal J$ and $\mathcal J_1=\mathcal J'$. The unlink isotopy $\mathcal J$ induces a bijection from the set of connected components of $\mU$ to that of $\mU'$ which does not change by homotopies of unlink isotopies. In fact, $\mathcal J$, up to homotopy, is determined by this bijection. This is a consequence of the fact that the fundamental group of configuration space of $n$ (labeled) points in $S^3$ is trivial. The cobordism $\Sigma(J)$ is diffeomorphic to $I \times \Sigma(\mU')$ and hence it comes with the product homology orientation. In the following, we use this homology orientation for the branched double cover of an unlink isotopy. Also, we will drop the disc filling $\mathcal D$ from our notation for an unlink unless there is a chance of confusion.

\begin{proposition} \label{cob-map-iso}
	Let $\mathcal J:U_n \to U_n$ be an unlink isotopy that induces the permutation $\sigma$ of the connected components of $U_n$. That is to say, $\mathcal J$ connects 
	the $i^{th}$ connected component of the incoming unlink to the $\sigma(i)^{th}$ connected component of the outgoing end. The permutation $\sigma$
	induces an automorphism of $H^0(U_n)$ and hence an automorphism of $\Lambda^*(\widetilde H^0(U_n)$. Once $\tPFH(\Sigma(U_n))$ is identified with 
	$\Lambda^*(\widetilde H^0(U_n)$ as above, this automorphism is equal to $\tPFH(\Sigma(J))$.
\end{proposition}

\begin{proof}
	The restriction maps from $M(\Sigma(J))$ to the representation varieties of the ends are diffeomorphisms of tori of dimension $n-1$. 
	However, these maps are not equal to each other. An appropriate parametrization of $M(\Sigma(J))$ 
	gives rise to the following restriction maps:
	\begin{equation} 
		\xymatrix{
			&S_1 \times S_2 \times \dots \times S_{n-1}\ar[dl]_{r_\sigma} \ar[dr]^{id}&\\
			S_1 \times S_2 \times \dots \times S_{n-1} &&S_1 \times S_2 \times \dots S_{n-1}\times S_{n-1}\\
			}
	\end{equation}
	where $r_\sigma(z_1,\dots,z_{n-1})$ is equal to:
	\begin{equation*}
		\left\{
		\begin{array}{ll}
			(z_{\sigma(1)},\dots,z_{\sigma(n-1)})& \text{\rm if } \sigma(n)=n\\
			(z_{\sigma(1)}z_i^{-1},\dots,z_{\sigma(j-1)}z_i^{-1},z_i^{-1},z_{\sigma(j+1)}z_i^{-1},\dots,z_{\sigma(n-1)}z_i^{-1})& \text{\rm if } i:=\sigma(n)\leq n-1, j:=
			\sigma^{-1}(n)
		\end{array}\right.
	\end{equation*}	
	Therefore, the map $\tPFH(\Sigma(J))$ is equal to $(r_\sigma)^*$. It is also true that the induced map by $\sigma$ on $\Lambda^*(\widetilde H^0(U_n)$ is the same
	as $(r_\sigma)^*$.
\end{proof}

\begin{corollary}
	If $(\mU,\mathcal D)$ is an arbitrary unlink, then $\tPFH(\Sigma(\mU))$ is canonically isomorphic to $\Lambda^*(\widetilde H^0(\mU)$.
\end{corollary}

\begin{proof}
	Let $\mU$ have $n$ components and fix an unlink isotopy $\mathcal J: (\mU,\mathcal D) \to (U_n,D_n)$. Since $J$ is the product cobordism, $\tPFH(\Sigma(J))$
	is an isomorphism. Furthermore, $\mathcal J$ defines a bijection between the connected components of $\mU$ and $U_n$. Therefore, it induces an isomorphism
	$i_{\mathcal J}: \Lambda^*(\widetilde H^0(\mU) \to \Lambda^*(\widetilde H^0(U_n)$. 
	The map $i_{\mathcal J}^{-1} \circ L_n \circ \PFH(\Sigma(J)):\tPFH(\Sigma(\mU)) \to \Lambda^*(\widetilde H^0(\mU)$ is an isomorphism. 
	Proposition \ref{cob-map-iso} shows that this isomorphism is independent of the choice of $\mathcal J$.
\end{proof}

Now that the plane Floer homology of an unlink is characterized more canonically, we can improve the statement of Proposition \ref{ele-cob-maps}. Let $c$ be a crossing of an unlink $\mU$ with $n$ components such that the two strands belong to the same component of $\mU$ and the intersection of the disc filling $\mathcal D$ with $c$ has the form in Figure \ref{filling-1}. Then replacing the crossing with Figure \ref{filling-2} introduces another unlink $\mU'$ with $n+1$ components and a disc filling $\mathcal D'$. The connected components of $\mU'$ that contains the strands in Figure \ref{filling-2} are denoted by $K_1$ and $K_2$. 

 \begin{figure}
	\centering
	\begin{subfigure}[b]{0.35 \textwidth}
		\includegraphics[width=\textwidth]{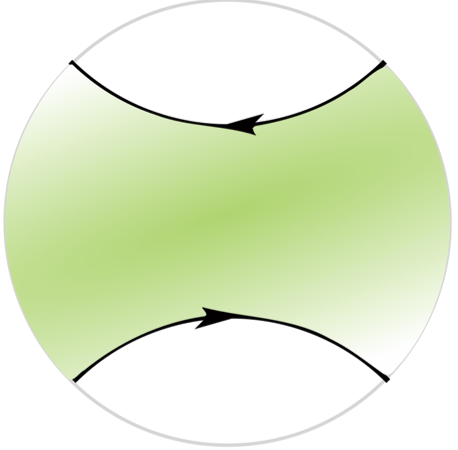}
		\caption{}
		\label{filling-1}
	\end{subfigure}
	\begin{subfigure}[b]{0.10 \textwidth}
	\includegraphics[width=\textwidth]{empty}
	\end{subfigure}
	\hspace{7pt}
	\begin{subfigure}[b]{0.35 \textwidth}
		\includegraphics[width=\textwidth]{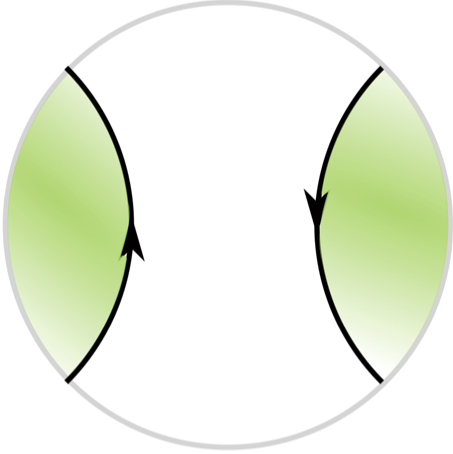}
		\caption{}
		\label{filling-2}
	\end{subfigure}
	\caption{}
	\label{U_n+1-to-U_n}
\end{figure}

As in subsection \ref{PFH-dbl-cover}, we can define a cobordism $\mathcal P :\mU \to \mU'$ which is a split cobordism in this case. Fix unlink isotopies $\mathcal J:\mU \to U_n$ and $\mathcal J':\mU' \to U_{n+1}$ such that $P_n \circ  J$ and $J' \circ \mathcal P$ are isotopic to each other with an isotopy that fixes the ends. Define the homology orientation of $\Sigma(\mathcal P)$ such that the composed orientations of $\Sigma(\mathcal P)$ and $\Sigma(\mathcal J')$ is equal to that of $\Sigma(\mathcal J)$ and $\Sigma(P_n)$. Proposition \ref{ele-cob-maps} shows that with the chosen homology orientation, the map $\tPFH(\Sigma(\mathcal P))$ acts in the following way:
\begin{equation} \label{gen-split-cob}
	F_{\mathcal P}: w \in \Lambda^*(\widetilde H^0(\mU)) \to v\wedge w \in \Lambda^*(\widetilde H^0(\mU'))
\end{equation}
where $v$, as an element of $\widetilde H^0(\mU)$, is non-zero only on the connected components $K_1$, $K_2$. It evaluates to 1 on one of these components and -1 on the other one. There are two possibilities for $v$ and both of them are feasible by different choices of $\mathcal J$ and $\mathcal J'$. In fact, if we change $\mathcal J'$ by composing it with an unlink isotopy that permutes the $n^{\rm th}$ and the $(n+1)^{\rm st}$ components, then $v$ turns into $-v$. However, if we fix one of the two possibilities for $v$, then all different choices of $\mathcal J$ and $\mathcal J'$ define the same cobordism map $F_{\mathcal P}$. Consequently, they determine the same homology orientation on $\mathcal P$.

Next, let $\mU$ be an unlink with $n+1$ connected components. Suppose $c$ is a crossing such that the two strands belong to the different connected components and the intersection of the disc filling $\mathcal D$ with the crossing is as in Figure \ref{filling-2}. As in the previous paragraph, we can construct an unlink $\mU'$ and a merge cobordism $\mathcal {\widebar P}: \mU \to \mU'$. Choose unlink isotopies $\mathcal J: \mU \to U_{n+1}$, $\mathcal J':\mU'\to U_n$, with $\mathcal {\widebar P}\circ J'$ and $J\circ \widebar P_n $ being isotopic to each other, and use these to define a homology orientation on $\Sigma(\mathcal {\widebar P})$ as in the previous case. Each connected component of $\mU$ is connected to one component of $\mU'$ via $\mathcal {\widebar P}$ and $j: H^0(\mU) \to H^0(\mU')$ is the map that sends the Poincar\'e dual of the fundamental class of a component of $\mU$ to the Poincar\'e dual of the fundamental class of the corresponding component of $\mU'$. This map induces:
\begin{equation} \label{gen-merge-cob}
	F_{\mathcal {\widebar P}}: w \in \Lambda^*(\widetilde H^0(\mU)) \to j_*(w) \in \Lambda^*(\widetilde H^0(\mU'))
\end{equation}
and by Proposition \ref{ele-cob-maps}, $\tPFH(\Sigma(\mathcal {\widebar P}))$ is equal to this map. 

Let $(\mU,\mathcal D)$, $(\mU',\mathcal D')$ be unlinks such that $(\mU, \mathcal D)$ is given by dropping one of the connected components of $(\mU',\mathcal D')$. We can construct a cap cobordism $\mathcal C: \mU \to \mU'$ that is the union of two parts. The first part is $I \times \mU$ and the second part is the pushing of the extra component of $\mathcal D'$ from $\{1\}\times S^3$ into $I \times S^3$. We can proceed as above to define a homology orientation for $\Sigma(\mathcal C)$. Let $i:H^0(\mU) \to H^0(\mU')$ be the map that extends an element of $H^0(\mU)$ by zero on the extra component of $\mU'$. Then the cobordism map associated with $\mathcal C$ is equal to:
\begin{equation} \label{gen-cap-cob}
	F_{\mathcal C}: w \in \Lambda^*(\widetilde H^0(\mU)) \to i_*(w) \in \Lambda^*(\widetilde H^0(\mU'))
\end{equation}

By reversing the role of the incoming and the outgoing ends in the previous paragraph, we can also construct a cup cobordism $\widebar {\mathcal C}$ from an unlink $\mU$ with $n+1$ components to an unlink $\mU'$ with $n$ components. By following the same strategy to define a homology orientation for $\Sigma(\widebar {\mathcal C})$, we will have the following cobordism map:
\begin{equation} \label{gen-cup-cob}
	F_{\widebar {\mathcal C}}: w \in \Lambda^*(\widetilde H^0(\mU)) \to \iota_K(w) \in \Lambda^*(\widetilde H^0(\mU'))
\end{equation}
where $\iota_{K}$ is contraction with respect to the element of $H_0(\mU)$ that is represented by the extra component. Note that in the case of merge, cup, and cap cobordisms the chosen homology orientations are independent of the unlink isotopes. Because in these cases the cobordism maps are independent of these choices. 

\subsection{Odd Khovanov Homology Revisited} \label{odd-khov-rev}

In this subsection, we firstly review the definition of reduced odd Khovanov homology. The reduced odd Khovanov homology of the mirror image with coefficients in $\tLz$ interacts better with plane Floer homology. Therefore, we focus on this version of odd Khovanov homology. The main input in the original definition of odd Khovanov homology is the notion of {\it decorated link diagrams}. In the context of this paper, oriented link diagrams can be used as a means to fix homology orientations for the cobordisms in the cube of resolutions:

\begin{definition}
	A {\it decorated link diagram} $D$ for a link $K$ is a planar diagram such that at each crossing of $D$, 
	one of the two possible orientations for the line segment $q$ in Figure \ref{dec-crossing} is fixed. 
	The resolutions of a decorated planar diagram are also equipped with extra decorations:
	if the line segment $q$ is oriented as Figure \ref{dec-crossing}, then orient the line segment connecting the two strands as in Figure \ref{dec-resolutions}. 
	Otherwise these orientations for the resolutions are reversed. 
\end{definition}

\begin{figure}
	\centering
	\begin{subfigure}[b]{0.3 \textwidth}
		\includegraphics[width=\textwidth]{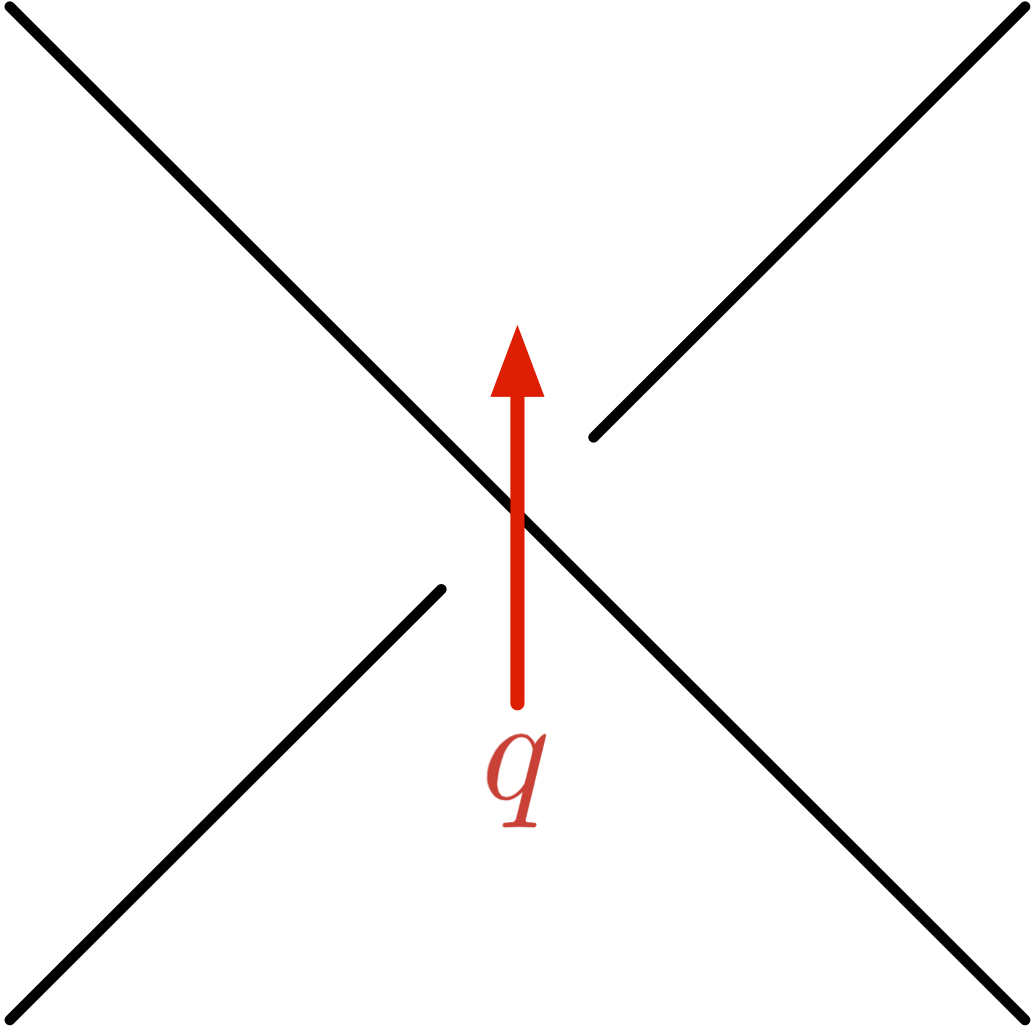}
		\caption{}
		\label{dec-crossing}
	\end{subfigure}
	\hspace{7pt}
	\begin{subfigure}[b]{0.65 \textwidth}
		\includegraphics[width=\textwidth]{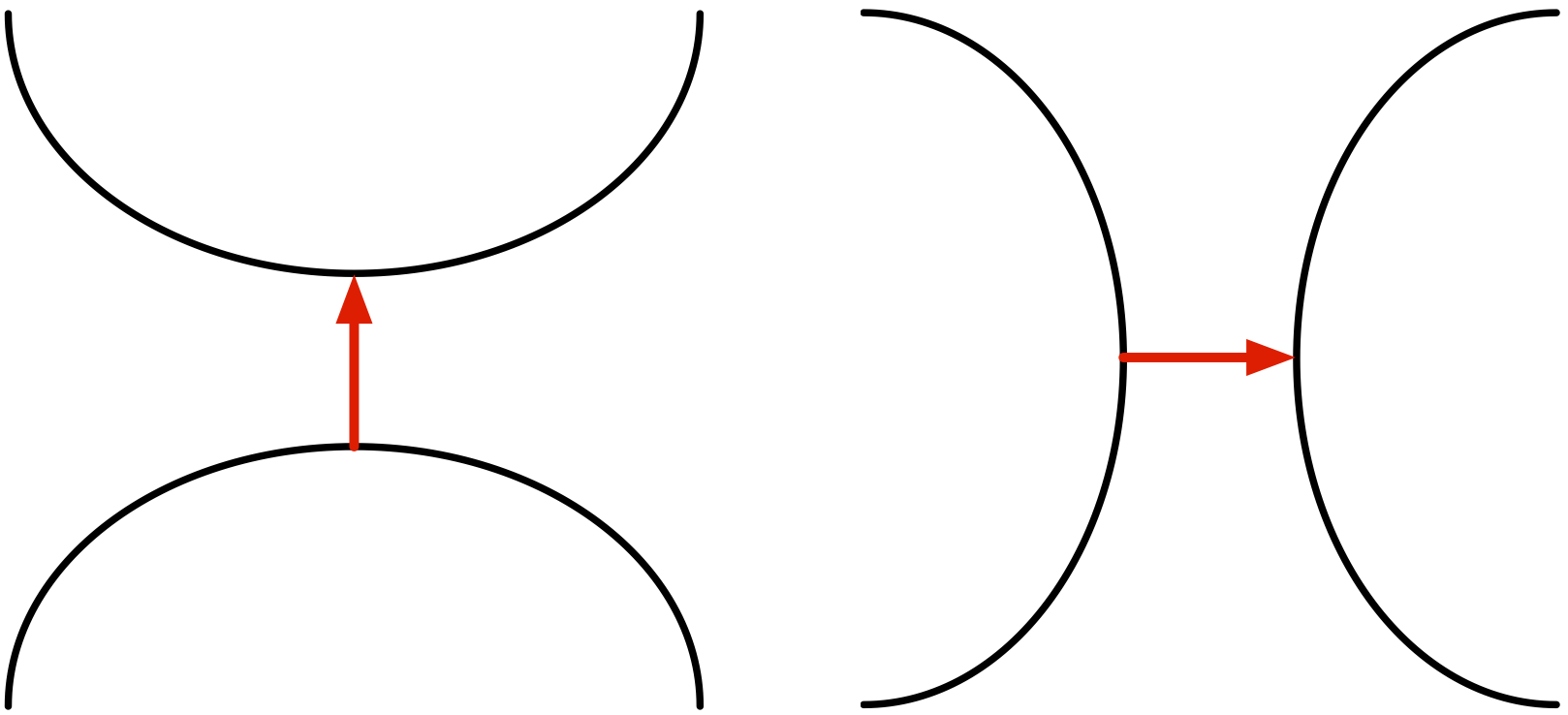}
		\caption{}
		\label{dec-resolutions}
	\end{subfigure}
	\caption{A decorated crossing and its decorated resolutions}
\end{figure}

Fix a decorated diagram $D$ for a link $K$. Any resolution $K_\bm$ of $K$, associated with $D$, is an unlink embedded in $\rr^2 \times \{0\}$ of $\rr^3$. A disc filling of $K_\bm$ can be fixed in the following way. The most inner circles in $K_\bm$ bound planar discs, i.e., discs that live in the plane $\rr^2 \times \{0\}$. Fix these discs as the filling of the corresponding components. Next, remove the filled circles and consider the discs that fill the most inner circles among the remaining ones. We can push the interior of such discs into $\rr^2 \times \rr^+$ such that they lie above the discs constructed in the first stage and hence they avoid intersecting with each other. Iteration of this construction determines a disc filling for $K_\bm$.

For an edge $(\bm,\bn)$ of the cube $\{0,1\}^{|N_D|}$, consider $\Smn: K_\bm \to K_\bn$ which is a pair of pants cobordism. A straightforward examination of the different possible cases shows that $\Smn$, with respect to the isotopy class of the fixed disc fillings, is either $\mathcal P$ or $\mathcal {\widebar P}$ constructed in subsection \ref{unlinks-cob-maps}. Homology orientations for the merge cobordisms are already fixed in the previous subsection. If $\Smn$ is a split cobordism, the ambiguity of the homology orientation of $\Sigma(\Smn)$ can be fixed by ordering the two components of the outgoing end that intersects the crossing. Denote these two components by $A_1$ and $A_2$. The decorated resolution of the crossing provides us with such an ordering. Suppose that the decoration of the line segment between $A_1$ and $A_2$ points from $A_1$ to $A_2$. We fix the homology orientation such that the cobordism map $\tPFH(\Sigma(\Smn))$ is equal to:
$$F_{\mathcal S}: w \in \Lambda^*(\widetilde H^0(K_\bm)) \to v\wedge w \in \Lambda^*(\widetilde H^0(K_\bn))$$
where $v\in \widetilde H^0(K_{\bn})$ evaluates to 1 (respectively, -1) on the representative of the component $A_1$ (respectively, $A_2$). 

In \cite{OzRaSz:Odd-Kh}, the first step to construct the reduced odd Khovanov chain complex of $\widebar K$ is to assign a module to each resolution $K_\bm$. It can be easily checked that this module is the same as $\tPFH(\Sigma(K_\bm))= \Lambda^*(\widetilde H^0(K_\bm))$. The next step in the definition of odd Khovanov homology is to define a map from $\Lambda^*(\widetilde H^0(K_\bm))$ to $\Lambda^*(\widetilde H^0(K_\bn))$ for each edge $(\bm,\bn)$ of the cube. It is also straightforward to check that this map is the same as $\tPFH(\Sigma(\Smn))$, using the fixed homology orientations. 

Suppose $(\bm,\bn)$ is a face of the cube $\{0,1\}^{|N_D|}$ and $\bk$, $\bk'$ are the two vertices that $\bm>\bk,\bk'>\bn$. Homology orientations of the cobordisms $S_\bm^\bk$, $S_\bm^{\bk'}$, $S_\bk^\bn$ and $S_{\bk'}^\bn$ are already fixed. This face of the cube is called {\it commutative} if the composed homology orientations of $S_\bm^\bk$ and $S_\bk^\bn$ is equal to the composed homology orientations of $S_\bm^{\bk'}$ and $S_{\bk'}^\bn$. Otherwise this face is called {\it non-commutative}. If the composition $\tPFH(\Sigma(S_\bk^\bn))\circ \tPFH(\Sigma(S_\bm^\bk))$ (or equivalently $\tPFH(\Sigma(S_{\bk'}^\bn))\circ \tPFH(\Sigma(S_\bm^{\bk'}))$) is non-zero then plane Floer homology can be used to detect whether $(\bm,\bn)$ is commutative or not. To be more specific, if:
$$\tPFH(\Sigma(S_\bk^\bn))\circ \tPFH(\Sigma(S_\bm^\bk))=\tPFH(\Sigma(S_{\bk'}^\bn))\circ \tPFH(\Sigma(S_\bm^{\bk'}))$$
then $(\bm,\bn)$ is commutative and it is non-commutative if:
$$\tPFH(\Sigma(S_\bk^\bn))\circ \tPFH(\Sigma(S_\bm^\bk))=-\tPFH(\Sigma(S_{\bk'}^\bn))\circ \tPFH(\Sigma(S_\bm^{\bk'})).$$

Any face $(\bm,\bn)$ of the cube $\{0,1\}^{|N_D|}$ is determined by the resolution of the vertex $K_\bm$ and the corresponding decorations of the two crossings associated with the face. A type $X$ face (respectively, type $Y$ face) is a face that its decorated resolution $K_\bm$ is equal to the union of the arrangement in Figure \ref{type-X} (respectively, Figure \ref{type-Y}) and several (possibly zero) other components that do not intersect the crossings of the face. Type $X$ and type $Y$ crossings are the only types of the faces such that $\tPFH(\Sigma(S_\bk^\bn))\circ \tPFH(\Sigma(S_\bm^\bk))$ vanishes.

 \begin{figure}
	\centering
	\begin{subfigure}[b]{0.25 \textwidth}
		\includegraphics[width=\textwidth]{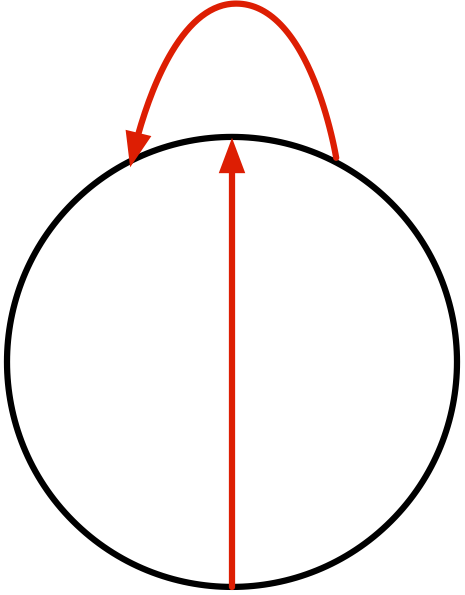}
		\caption{A type $X$ face}
		\label{type-X}
	\end{subfigure}
	\begin{subfigure}[b]{0.30 \textwidth}
	\includegraphics[width=\textwidth]{empty}
	\end{subfigure}
	\begin{subfigure}[b]{0.25 \textwidth}
		\includegraphics[width=\textwidth]{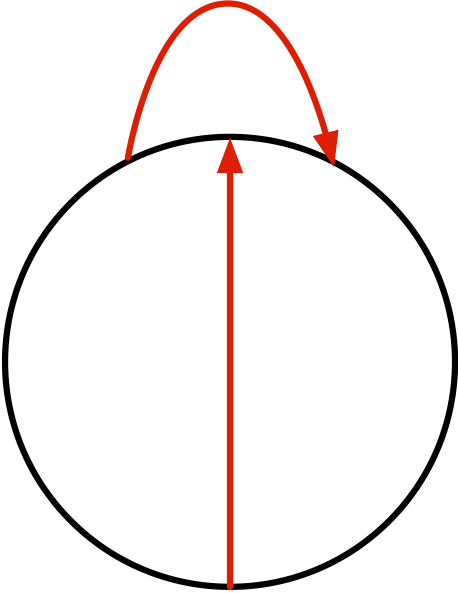}
		\caption{A type $Y$ face}
		\label{type-Y}
	\end{subfigure}
	\caption{}
\end{figure}

\begin{lemma} \label{typeX-even}
	Out of type $X$ and type $Y$ faces, faces of one type are commutative and faces of the other type are non-commutative.
\end{lemma}

\begin{proof}
	Suppose that $L$ is the decorated planar diagram of a link consisting of two crossings such that its $(1,1)$ resolution has only one component and it has 
	the form of Figure \ref{type-X} . Therefore, the only face of this planar diagram has type $X$. Label these crossings by requiring that the first crossing is represented by 
	the vertical red line in Figure \ref{type-X} and the second crossing is the other one. Let $L'$
	be another decorated diagram for the same link that is given by Figure \ref{type-Y}. The diagram $L'$ has only one face that is of type $Y$.
	We can use these two decorated diagrams to define two homology orientations for each of the cobordisms $S^{10}_{11}$, $S^{01}_{11}$, $S^{00}_{10}$ and 
	$S^{00}_{01}$. Out of these cobordisms, only the two homology orientations of the cobordism $S^{10}_{11}$ differ from each other. 
	Therefore, one of $L$ and $L'$ has a commutative face and the other one has a non-commutative face. 
	In the following, we assume that $L$ is the decorated diagram with a commutative face. The proof in the other case is similar.
	
	 Let $(\bm,\bn)$ be a face of a decorated diagram $K$ with type $X$. Therefore $K_\bm$ has a component that intersects the two crossings determined by
	  this face. This component has the form in Figure \ref{type-X} and we fix the labeling such that:
	 $$\bm=(1,1,0\dots,0)\hspace{2cm}\bn=(0,0,0\dots,0)$$
	 and if:
	 $$\bk=(1,0,0\dots,0)\hspace{2cm}\bk'=(0,1,0\dots,0)$$
	 then $K_{\bk}$ is given by the 1-resolution of the first crossing and the 0-resolution of the second crossing in Figure \ref{type-X}.
	Suppose that $S: K_\bm \to L_{11}$, $S': K_\bn \to L_{00}$, $S_1'':K_\bk \to L_{10}$ and $S_2'':K_{\bk'} \to L_{01}$ are composition 
	of cup cobordisms that fill the components which do not intersect the crossings. Furthermore, we have:
	\begin{equation} \label{decom-1}
		S_{\bm}^{\bk'}\circ  S_{\bk'}^\bn\circ  S' = S_{\bm}^{\bk'}\circ S_2''\circ S_{01}^{00}= S\circ S_{11}^{01} \circ S_{01}^{00}=
		 S\circ S_{11}^{10} \circ S_{10}^{00}= S_{\bm}^\bk\circ S_1''\circ S_{10}^{00}=  S_{\bm}^\bk\circ  S_{\bk}^\bn\circ S'
	\end{equation}	
	
	We use the homology orientation of the cup cobordisms to define homology orientations for $\Sigma(S)$, $\Sigma(S')$, $\Sigma(S_1'')$ and $\Sigma(S_2'')$. 
	The cobordism compositions in (\ref{decom-1}) define 6 homology orientations on the same cobordism. By assumption the homology orientation determined by the 
	compositions $ S\circ S_{11}^{10} \circ S_{10}^{00}$ and $ S\circ S_{11}^{01} \circ S_{01}^{00}$ are the same. One can also easily check that the non-zero maps
	$\tPFH(\Sigma(S')) \circ \tPFH(\Sigma(S_{\bk'}^\bn))$ and $\tPFH(\Sigma(S_{01}^{00}))\circ \tPFH(\Sigma(S_2''))$ are equal to each other. Thus 
	$S_{\bm}^{\bk'} \circ  S_{\bk'}^\bn\circ S'$ and $S_{\bm}^{\bk'} \circ S_2''\circ S_{01}^{00}$ determine the same homology orientations. A similar argument 
	shows that the remaining pair of compositions that appear as consecutive terms in (\ref{decom-1}) determine the same homology orientations. In particular,
	the homology orientations of the compositions $ S_{\bm}^{\bk'} \circ  S_{\bk'}^\bn\circ S'$ and $ S_{\bm}^\bk \circ  S_{\bk}^\bn\circ S'$ are the same. This proves
	the claim that
	the face $(\bm,\bn)$ is even.
\end{proof}

The following definition is a reformulation of Definition 1.1. in \cite{OzRaSz:Odd-Kh}:
\begin{definition}
	A type $X$ edge assignment for $D$ is a map $\sigma$ from the set of edge of the cube $\{0,1\}^n$ to $\{1,-1\}$ such 
	that for any face $(\bm,\bn)$: 
	\begin{equation} \label{edge-assign-rel}
		\sigma(\bm,\bk)\sigma(\bk,\bn)\tPFH(\Sigma(S_\bk^\bn))\circ \tPFH(\Sigma(S_\bm^\bk))+
		\sigma(\bm,\bk')\sigma(\bk',\bn)\tPFH(\Sigma(S_{\bk'}^\bn))\circ \tPFH(\Sigma(S_\bm^{\bk'}))=0.
	\end{equation}
	Furthermore, we require that if $(\bm,\bn)$ is a face of type $X$, then:
	\begin{equation*}
		\sigma(\bm,\bk)\cdot \sigma(\bk,\bn)=\sigma(\bm,\bk')\cdot \sigma(\bk',\bn)
	\end{equation*}	
	and if it is a face of type $Y$, then:
	\begin{equation*}
		\sigma(\bm,\bk)\cdot \sigma(\bk,\bn)=-\sigma(\bm,\bk')\cdot \sigma(\bk',\bn).
	\end{equation*}	
	In the definition, if we reverse the role of type $X$ and type $Y$ faces, then the resulting edge assignment has type $Y$.
\end{definition}

Given a type $X$ or a type $Y$ edge assignment $\sigma$, the reduced odd Khovanov chain complex of $\widebar K$ associated with the decorated link diagram $D$ is equal to: 
\begin{equation} \label{odd-chain-com-2}
	(\bigoplus_{\bm\in \{0,1\}^n}\tPFH(\Sigma(K_\bm)),\bigoplus_{(\bm,\bn)\in \mathcal E_c}\sigma(\bm,\bn)\tPFH(\Sigma(\Smn))
\end{equation}	
Relation (\ref{edge-assign-rel}) implies that (\ref{odd-chain-com-2}) is in fact a chain complex. The chain complex  (\ref{odd-chain-com-2}) is equipped with two gradings in \cite{OzRaSz:Odd-Kh} which are called the homological and quantum gradings. If we will write $h_o$ and $q_o$ for these gradings, then these gradings are equal to:
\begin{align} \label{diagram-grading}
	\deg_{h_o}&=(|\widebar D|-|\bm|_1)-n_-(\widebar D)\\
	\deg_{q_o}&=(1+\deg_p)+(|\widebar D|-|{\bm}|_1)+n_+(\widebar D)-2n_-(\widebar D)
\end{align}
Here $\widebar D$ is the digram for $\widebar K$ that is induced by $D$. This term appears in our formulas because the chain complex (\ref{odd-chain-com-2}) competes the odd Khovanov homology of $\widebar K$ using the diagram $\widebar D$.

\begin{theorem}
	There is an isomorphism between the chain complexes (\ref{odd-chain-com}) and (\ref{odd-chain-com-2}). 
	Furthermore, this isomorphism relates the bi-gradings in the following way:
	$$\deg_{h_o}=\deg_{h},\hspace{2cm}\deg_{q_o}=2(\deg_{h}-\deg_{\delta})+1$$
\end{theorem}

\begin{proof}
	Note that the definition of the chain complex (\ref{odd-chain-com}) depends on $D$ and a choice of homology orientations on the cobordisms $\Sigma(\Smn)$ 
	that satisfies the assumption of Lemma \ref{fix-hom-ori}. On the other hand, the chain complex (\ref{odd-chain-com-2}) depends on $D$ and a type $X$ or type $Y$
	edge assignment. In particular, the cobordism maps $\tPFH(\Sigma(\Smn))$, involved in the definition, are defined by homology orientation $o_{N_D}(\bm,\bn)$
	 determined by $D$. It is shown in \cite{OzRaSz:Odd-Kh} that for any decorated diagram, type $X$ and type $Y$ edge assignments exist and 
	 the corresponding chain complexes for a fixed diagram are isomorphic.
	
	Now suppose type $X$ faces are non-commutative and a type $X$ edge assignment $\sigma$ of the diagram is fixed. For an edge $(\bm,\bn)$ of the cube 
	$\{0,1\}^n$ define:
	$$o(\bm,\bn):=\sigma(\bm,\bn)\sgn(\bm,\bn)o_{N_D}(\bm,\bn) \in o(\Sigma(\Smn))$$
	Suppose $(\bm,\bn)$ is a face of the cube $\{0,1\}^n$ and $\bk$ and $\bk'$ are the vertices that $\bm>\bk,\bk'>\bn$. Then we have:
	\begin{equation} \label{com-hom}
		o(\bk,\bn)\circ o(\bm,\bk)=o(\bk',\bn)\circ o(\bm,\bk').
	\end{equation}	
	In the case of a type $X$ face it holds by definition. The case of type $Y$ face is a consequence of Lemma \ref{typeX-even}. 
	For the remaining faces it can be derived from (\ref{edge-assign-rel}). Relation (\ref{com-hom}) implies that $o$ can be extended to a homology orientation for
	the cobordisms $\Sigma(\Smn)$ for all cubes $(\bm,\bn)$ in $\zz^n$ such that:
	$$o(\bk,\bn)\circ o(\bm,\bk)=o(\bm,\bn).$$
	Therefore, we can use this homology orientation to define the chain complex (\ref{odd-chain-com}). With this choice of homology orientation the identity map
	defines an isomorphism of the chain complexes. In the case that type $X$ faces are commutative, the same argument can be used with the aid of a type $Y$
	edge assignment. The claim about the relations between gradings is also an immediate consequence of the definition.
\end{proof}

\begin{corollary}
	Given a link $K$, there is a spectral sequence that its first page is equal to $Kh_o(\widebar K) \otimes \tL$ and 
	it converges to
	$\tPFH(\Sigma(K))$.
\end{corollary}

\begin{proof}
	In section 2, starting from a pair $(Y,\zeta)$ and a framed link $L$ with $n$ components, we constructed a pair $(Y_\bm,\zeta_\bm)$ for each $\bm \in \zz^n$.
	The 1-manifold $\zeta_\bm$ is the union $\cup_{ i}\zeta_{\bm,i}$ and $\zeta_{\bm,i}$ 
	is non-empty if and only if $\bm_i \equiv 0$ mod 3, and in this case $\zeta_{\bm,i}$
	is the core of the tours $D_{\bm,i}$. In the definition of $\zeta_{\bm}$, we could replace the condition $\bm_i \equiv 0$ mod 3 with  $\bm_i \equiv 2$ mod 3.
	Then the definition of $\Zmn$ can be modified accordingly, and another version of Theorem \ref{PFH-tri-1} can be proved. This version can be 
	used to prove a variation of Theorem \ref{crossing-set-complex}. In particular, this variation states that a planar diagram $D$ determines a chain complex 
	$(\oplus_{\bm \in \{0,1\}^n}\tPFC(\Sigma(K_\bm)),d_p)$ which has the same chain homotopy type as $\tPFC(\Sigma(K),\zeta)$. 
	The differential $d_p$ is given by the maps $\dmn:\tPFC(\Sigma(K_\bm)) \to \tPFC(\Sigma(K_\bn))$ where $\dmn$ is the cobordism map for the family of metrics 
	$\bWmn$ parametrized by $\Gmn$ (constructed in subsection \ref{3-man-link}). Furthermore, the 1-manifold $\zeta$ is the branched double cover of 
	$\gamma= \cup_i p_i$ where $p_i$ is the path that connects the two strands in the $i^{\rm th}$ crossing of $D$ (Fig. \ref{crossing}). 
	
	 Let $[\zeta]$ denote the $\zz/2\zz$-homology class of the 1-manifold $\zeta$.
	The $\zz/2\zz$ action on $\Sigma(K)$ induces a $\zz/2\zz$ action $\tau$ on the set of simplicial complexes $C_*^{simp}(\Sigma(K))$. Therefore, we
	have a map: 
	$$i:H_*((1+\tau)C_*^{simp}(\Sigma(K));\zz/2\zz) \to H_*(\Sigma(K);\zz/2\zz)$$ 
	and $[\zeta]$ lies in the image of $i$. The homology group
	$H_*((1+\tau)C_*^{simp}(\Sigma(K));\zz/2\zz)$ is isomorphic to $H_*(S^3,K)$, and under this isomorphism $[\zeta]$ corresponds to $[\gamma]\in H_1(S^3,K)$.
	The boundary of $[\gamma]$ represents the trivial class in $H_0(K)$, because the set of crossings is induced by a diagram of the link $K$. 
	Therefore, $[\zeta]$ vanishes and by Lemma \ref{mod-2-dep-loc-coef},  $\tPFH(\Sigma(K),\zeta)=\tPFH(\Sigma(K))$. 
	Consequently, the chain complex $(\oplus_{\bm \in \{0,1\}^n}\tPFC(\Sigma(K_\bm)),d_p)$ gives rise to a spectral sequence whose first page is the complex 
	(\ref{odd-chain-com}) and it abuts to $\tPFH(\Sigma(K))$.
\end{proof}
\newpage


\bibliography{references}
\bibliographystyle{hplain}

\end{document}